\documentclass{amsart}
\usepackage{amssymb}
\usepackage{amsfonts}
\usepackage{amsmath}
\usepackage{amsthm, amsmath, amssymb, enumerate}
\usepackage{amssymb, enumitem}
\usepackage{bbm}
\usepackage{amscd}
\usepackage[all]{xy}
\usepackage[hidelinks]{hyperref}
\usepackage{amsmath}
\usepackage{amssymb}
\usepackage{amsthm}
\usepackage[left=2cm,right=2cm,bottom=3cm,top=3cm]{geometry}
\usepackage[sort&compress,numbers]{natbib}
\usepackage{tikz}
\usepackage{tikz-cd}
\usepackage{wrapfig}
\usepackage{comment}

\usepackage{centernot}
\usepackage{mathtools}
\usepackage{ stmaryrd }

\newtheorem{theorem}{Theorem}[section]
\newtheorem{proposition}[theorem]{Proposition}
\theoremstyle{definition}
\newtheorem{lemma}[theorem]{Lemma}
\newtheorem{claim*}[theorem]{claim*}
\newtheorem{definition}[theorem]{Definition}
\newtheorem{corollary}[theorem]{Corollary}

\newtheorem{remark}[theorem]{Remark}

\newtheorem{example}[theorem]{Example}

\newtheorem*{clm}{Claim}
\AtBeginDocument{   \def\MR#1{}}

\begin{document}
\def\cprime{$'$}

\title{The definable content of homological invariants I: 
$\mathrm{Ext}$ $\&$  $\mathrm{lim}^1$}

\author{Jeffrey Bergfalk}
\address{Departament de Matem\`{a}tiques i Inform\`{a}tica \\
Universitat de Barcelona \\
Gran Via de les Corts Catalanes, 585 \\ 08007 Barcelona, Catalonia}
\email{bergfalk@ub.edu}

\author{Martino Lupini}
\address{Dipartimento di Matematica\\
Universit\`{a} di Bologna\\
Piazza di Porta S. Donato, 5\\
40126 Bologna BO\\
Italy}
\email{lupini@tutanota.com}
\urladdr{http://www.lupini.org/}

\author{Aristotelis Panagiotopoulos}
\address{Kurt G\"odel Research Center\\ Faculty of Mathematics\\ University of Vienna\\ Kolingasse 14-16\\ 1090 Wien\\ Austria}
\email{aristotelis.panagiotopoulos@gmail.com}

\thanks{Part of this work was done during visits of J.B. and M.L at the
California Institute of Technology, and of A.P. at Victoria University of
Wellington. The authors gratefully acknowledge the hospitality and the
financial support of these institutions. J.B.\ was partially supported by the MSC grant CatT 101110452.  M.L.\ was partially supported by the
NSF Grant DMS-1600186, by a Research Establishment Grant from Victoria
University of Wellington, by the Marsden Fund Fast-Start
Grant VUW1816, by the Rutherford Discovery Fellowship VUW2002 from the Royal
Society of New Zealand, and by the Starting Grant 101077154 ``Definable
Algebraic Topology'' from the European Research Council.}

\subjclass[2010]{Primary 54H05, 18G10, 18G60;  Secondary 55N05, 55N07,  37A20}

\keywords{Polish group, group with a Polish cover, profinite group,
non-Archimedean Polish group, Ulam stability, cocycle superrigidity, derived
functor, $\mathrm{Ext}$, $\mathrm{lim}^{1}$, definable homomorphism, countable Borel equivalence
relation, Borel reduction, Property (T), finite rank abelian group}

\date{}

\begin{abstract}This is the first installment in a series of papers illustrating how classical invariants of homological algebra and algebraic topology may be enriched with  additional descriptive set theoretic information.   To effect this enrichment, we show that many of these invariants may be naturally regarded as functors to the category, introduced herein, of \emph{groups with a Polish cover}. The resulting \emph{definable invariants}  provide far stronger means of classification.

 In the present work we focus on the first derived functors of $\mathrm{Hom}(-,-)$  and $\mathrm{lim}(-)$. The resulting \emph{definable $\mathrm{Ext}(B,F)$} for pairs of countable abelian groups $B,F$ and  \emph{definable $\mathrm{lim}^{1}(\boldsymbol{A})$} for towers $\boldsymbol{A}$ of Polish abelian groups  substantially refine their classical counterparts. We show, for example, that the definable $\textrm{Ext}( -,\mathbb{Z}%
)$ is a fully faithful contravariant functor from the category of finite rank torsion-free abelian groups $\Lambda$ with no free summands; this contrasts with the fact that there are uncountably many non-isomorphic such groups $\Lambda$ with isomorphic classical invariants $\textrm{Ext}( \Lambda ,\mathbb{Z}) 
$. 
To facilitate our analysis, we introduce a general \emph{Ulam stability framework} for groups with a Polish cover; within this framework we prove several rigidity results for non-Archimedean abelian groups with a Polish cover.  A special case of our main result answers a question of Kanovei and Reeken regarding quotients of $p$-adic groups.
 Finally, using cocycle
superrigidity methods for profinite actions of property (T) groups, we obtain a hierarchy of
complexity degrees for the  problem $\mathcal{R}(\mathrm{Aut}( \Lambda) \curvearrowright \mathrm{Ext}( \Lambda ,\mathbb{Z%
}))$ of classifying all group extensions of $\Lambda$ by $\mathbb{Z}$ up to \emph{base-free isomorphism}, when $\Lambda =\mathbb{Z}[1/p]^{d}$ for prime numbers $p$ and $
d\geq 1$.
\end{abstract}

\maketitle







\tableofcontents
\section{Introduction}

Some of the best-known and most versatile invariants in mathematics arise as the (co)homology groups of (co)chain complexes. Such invariants are computed by 
 first associating to each object $X$ a chain complex 
\[C_{\bullet}:= \quad 0\overset{\partial_0}{\longleftarrow} C_0\overset{\partial_1}{\longleftarrow} C_1\overset{\partial_2}{\longleftarrow} C_2 \overset{\partial_3}{\longleftarrow} \cdots\]
of abelian groups encoding the relevant data about $X$.   The \emph{$n^{\mathrm{th}}$ homology group of $C_{\bullet}$} is then defined as the quotient
\[\mathrm{H}_n:=\mathrm{Z}_n/\mathrm{B}_n:= \mathrm{ker}(\partial_n)/\mathrm{ran}(\partial_{n+1}).\]
The overwhelming tendency is to regard these groups as discrete objects. This is despite the fact that the \emph{chain groups} $C_n$ often carry a natural topology, and even one encoding data about $X$ not captured by their group structures alone. It is also despite the \textquotedblleft \emph{trend%
}\textquotedblright\ which Dieudonn\'{e} 
recalls in the opposite direction \textquotedblleft \emph{that
was very popular until around 1950 (although later all but abandoned),
namely, to consider homology groups as topological groups for suitably
chosen topologies}\textquotedblright\ \cite[p.\ 67]{dieudonne_history_2009}. Dooming this approach seems mainly to have been the fact that the aforementioned natural topologies may badly fail to induce Hausdorff topologies on $\mathrm{H}_n$. A case of some midcentury prominence, for example, was that of the reduced $0$-dimensional Steenrod homology group $\tilde{\mathrm{H}}_0(S)$ of a dyadic solenoid $S$: here, not only is the boundary group $B_0$ not closed in the natural topology on $Z_0$; it is dense therein (see \cite{eilenberg_group_1942}).

\subsection{The definable content of homological invariants}
What follows is the first in a series of papers in which we show how to endow various homological invariants with a finer structure than the quotient topology, as well as some of the benefits of doing so \cite{BLPII, BLPIII,lupini_definable_2020}. A main resource for this refinement is the field of \emph{invariant descriptive set theory}, an area of mathematics which studies the Borel complexity of classification problems; see Section \ref{S:EquivalenceRelation} below. The critical contexts for that field are \emph{Polish spaces}, and the fundamental, initiating recognition for our work is the fact that many of the (co)homology groups from algebraic topology and homological algebra are naturally viewed as \emph{groups with a Polish cover}, i.e., as quotients $G/N$ of a Polish group $G$ by a Polishable subgroup $N$. Examples include:
\begin{enumerate}
\item the \emph{strong} or \emph{Steenrod homology groups} $\mathrm{H}_n(K)$ of a compact metrizable space $K$;
\item the \emph{\v{C}ech} or \emph{sheaf cohomology groups} $\mathrm{H}^n(L)$ of a locally compact metrizable space $L$;
\item the \emph{first derived group $\mathrm{Ext}(B,F)$ of the $\mathrm{Hom}(-,-)$-bifunctor}  applied to countable abelian groups $B,F$;
\item the \emph{first derived group $\underleftarrow{\mathrm{lim}}^{1}(\boldsymbol{A})$ of the $\underleftarrow{\mathrm{lim}}(-)$-functor}  applied to towers of  abelian Polish groups $\boldsymbol{A}$.
\end{enumerate}
In this first paper we restrict our attention to examples (3) and (4). These arise in purely algebraic settings and form the backbone of a variety of computations in algebraic topology --- typically by way of the Universal Coefficient Theorem \cite{eilenberg_group_1942} and Milnor's Exact Sequence \cite{milnor_axiomatic_1962}, respectively --- and, in particular, of many involving (1) and (2). Our results
 below will be applied in \cite{BLPII} to show that, in contrast to the classical \v{C}ech cohomology theory, \emph{definable \v{C}ech cohomology theory} provides complete homotopy invariants for mapping telescopes of $n$-spheres and $n$-tori, as well as for maps to spheres from the latter.

We view groups with a Polish cover, and hence all of the aforementioned invariants, as objects in the \emph{category of groups with a Polish cover}. Morphisms in this category are 
\emph{definable homomorphisms}. A \emph{definable homomorphism} between groups with a Polish cover $G/N$ and $G'/N'$  is any 
group homomorphism $f\colon G/N \to G'/N'$ which lifts to a Borel map $\hat{f}\colon G\to G'$. We do not require $\hat{f}$ to be a homomorphism from $G$ to $G'$. Definable homomorphisms may be thought of as those homomorphisms which can be described explicitly by a (potentially infinitary) formula, one making no essential appeal to the axiom of choice. For example, $\mathbb{R}/\mathbb{Q}$ admits $2^{2^{\aleph_0}}$ many endomorphisms as an abstract group, since it is a $\mathbb{Q}$-vector space of dimension $2^{\aleph_0}$.  However, only continuum many of these endomorphisms are actually definable; see \cite{kanovei_baire_2000}. For $B, F$, and $\mathbf{A}$ as above, the classical assignments $(B,F)\mapsto\mathrm{Ext}(B,F)$ and $\boldsymbol{A}\mapsto\underleftarrow{\mathrm{lim}}^{1}(\boldsymbol{A})$ factor through a bifunctor and functor, respectively,  which take values in the additive category of abelian groups with a Polish cover. 
The resulting \emph{definable $\mathrm{Ext}$} invariants and \emph{definable $\underleftarrow{\mathrm{lim}}^{1}$} invariants  record much more information than their purely algebraic counterparts.

Consider, for example, the problem of classifying all finite-rank torsion-free abelian groups $\Lambda$ up to isomorphism. Notice that the invariant $\mathrm{Hom}(\Lambda,\mathbb{Z})$ is a free abelian group whose rank coincides with the rank of the largest free direct summand of $\Lambda$.  By the following theorem, the discrete group $\mathrm{Hom}(\Lambda,\mathbb{Z})$ together with the \emph{definable $\mathrm{Ext}(\Lambda,\mathbb{Z})$ group} completely classify  all finite rank torsion-free abelian groups; see Corollary \ref{Corollary:TF-Ext}.

\begin{theorem}
\label{Theorem:main1}The functor $\mathrm{Ext}( -,\mathbb{Z}) $
is a fully faithful  functor from the category of finite rank
torsion-free abelian groups with no free summands to the category of groups with a Polish
cover.

In particular, $\mathrm{Ext}( \Lambda,\mathbb{Z})$, up to definable
isomorphism, together with the rank of $\mathrm{Hom}(\Lambda,\mathbb{Z})$, form a complete set of 
invariants for finite rank torsion-free abelian
groups up to isomorphism.
\end{theorem}

Theorem \ref{Theorem:main1} should be contrasted with the fact that, as a discrete group, $
\mathrm{Ext}( \Lambda ,\mathbb{Z})$
records comparatively little about a given finite-rank torsion-free group $\Lambda$. Indeed, there is a size-continuum family of non-isomorphic such groups $\Lambda$ whose
corresponding invariants $\mathrm{Ext}( \Lambda ,\mathbb{Z}) $ are
isomorphic as abstract groups, as we show in Corollary \ref{Corollary:nonisomorphic-Ext} below. In the process of proving  Theorem \ref{Theorem:main1}, we develop a definable version of Jensen's theorem which relates the definable $\mathrm{Ext}$-functor  to the definable $\underleftarrow{\mathrm{lim}}^{1}$-functor. The definable content of the $\underleftarrow{\mathrm{lim}}^{1}$-functor is studied  in Section \ref{Section:towers}, where we record a further result in the spirit of Theorem \ref{Theorem:main1}; see Corollary \ref{Corollary:lim1-injective}. We also provide an explicit description of $\underleftarrow{\mathrm{lim}}^{1}(\boldsymbol{A})$ when   $\boldsymbol{A}$ is a monomorphic tower of abelian groups; see Theorem \ref{Theorem:lim1-injective}. For further applications of these analyses, see the second author's \cite{Lupini_TAG,Lupini_TfAG}, in which the definable $\mathrm{Ext}$ functor is applied to obtain new purely algebraic results; see also his recent \cite{Lupini_heart}, which underscores how canonical the framework introduced herein in fact is: in the latter, the category of groups with an abelian Polish cover is shown to form a minimal abelian extension, or more precisely \emph{heart}, of the category of Polish abelian groups.

\subsection{An Ulam stability framework} 
The main ingredient in the aforementioned theorems and corollaries is  a new rigidity result for non-Archimedian abelian groups with a Polish cover. 
To state and prove this result, we introduce a general \emph{Ulam stability} framework which specializes to the one appearing in \cite{kanovei_baire_2000,
kanovei_ulams_2000,kanovei_ulam_2001} when the cover $G$ of the relevant group $G/N$ is a connected Polish group. 

\emph{Ulam stability phenomena} may be studied in any setting where one has: (1) a notion of \emph{morphisms}; (2)  a notion of \emph{approximate morphisms}; and (3) a notion of closeness relating morphisms and approximate morphisms. Such phenomena have received
considerable attention over the last 50 years, especially in the settings of
groups, Boolean algebras, and C*-algebras; see \cite%
{kazhdan_varepsilon-representations_1982,hyers_stability_1941,farah_completely_1998,farah_liftings_2000,farah_analytic_2000,farah_approximate_1998,farah_approximate_2000,kalton_uniformly_1983,alekseev_approximations_1999,grove_jacobi_1974,shtern_almost_1998,shtern_rigidity_1999,ghasemi_isomorphisms_2015}.
An Ulam stability framework was introduced in \cite{kanovei_baire_2000} for studying definable homomorphisms $\mathbb{R}/N\to \mathbb{R}/N'$, when  $N$ and $N'$ are countable subgroups of $(\mathbb{R},+)$. 
There, approximate morphisms are Baire-measurable maps $\mathbb{R} \to \mathbb{R}$ which are lifts of homomorphisms $\mathbb{R}/N \to \mathbb{R}/N'$; morphisms are approximate morphisms of the form $x \mapsto cx$ for some $c \in \mathbb{R}$; and two morphisms are ``close'' to each other if they are lifts of the same homomorphisms $\mathbb{R}/N \to \mathbb{R}/N'$. The main theorem in \cite{kanovei_baire_2000} is that every approximate morphism is close to a morphism. Hence every definable homomorphism $\mathbb{R}/N \to \mathbb{R}/N'$ is of the form $x \mapsto cx$.

The main theorem in \cite{kanovei_baire_2000} is that every approximate morphism is close to a morphism. Hence every definable homomorphism $\mathbb{R}/N\to \mathbb{R}/N'$ is of the form $x\mapsto c\cdot x$. The question of whether similar Ulam stability phenomena exist for quotients of the $p$-adic groups appears in   \cite[Section 8]{kanovei_ulam_2001}. The next theorem answers this  in the more general context of quotients of arbitrary abelian pro-countable groups. For definitions of \emph{trivial}, \emph{approximately trivial}, and  \emph{approximately generically trivial} homomorphisms, as well as for the proof of the theorem, we refer the reader to Section \ref{S:Ulam}, where we develop the appropriate Ulam stability framework.

\begin{theorem}
\label{Theorem:main2} Let $f\colon G/N\to G'/N'$ be a definable homomorphism between groups with a Polish cover, where $G$ is abelian and non-Archimedean and $N$ is dense in $G$.
\begin{enumerate}
\item If $N'$ is countable then $f$ is trivial;
\item If $N'$ is locally profinite then $f$ is approximately trivial;
\item  If $N'$ is non-Archimedean in its Polish topology then $f$ is approximately generically trivial.
\end{enumerate}
\end{theorem}
Results (1) and (2) above are optimal, as illustrated by  Examples \ref{Example:trivial} and \ref{Example:trivial2}.

\subsection{Cocycle superrigidity methods and Borel reducibility complexity} 

Lastly, we establish complexity bounds for various classification problems within the Borel reduction hierarchy. Many classification problems in mathematics naturally parametrize as pairs $(X,E)$, where $X$ is a Polish space and $E$ is a Borel equivalence relation.
Invariant descriptive set theory studies the ordering of the collection of all such classification problems according to their relative complexity, as measured by the Borel reduction relation $(X,E)\leq_B (Y,F)$; see Section \ref{S:EquivalenceRelation}.

One classification problem that naturally arises in homological computations pertinent, e.g., to shape theory \cite{mardesic_shape_1982}, is that of classifying objects of a pro-$\mathrm{Ho}(\mathrm{Top})$ subcategory up to isomorphism. As a biproduct of the proof of Theorem \ref{Theorem:main1} and the results from \cite{adams_linear_2000,hjorth_1999_rank2,thomas_classification_2003}, we establish that even for very simple full subcategories of  pro-$\mathrm{Ho}(\mathrm{Top})$, the complexity of the isomorphism relation $\simeq_{\mathrm{pro}}$ on objects can be highly non-trivial; see Corollary \ref{Corollary:proCategory}.

Let $\mathrm{Z}( \Lambda ,%
\mathbb{Z})/\mathrm{B}( \Lambda ,%
\mathbb{Z})$ be the presentation of   $\mathrm{Ext}( \Lambda,\mathbb{Z})$ as a group with a Polish cover. Since $\Lambda$ is a countable abelian group, the problem $\mathcal{R}(\mathrm{Ext}( \Lambda,\mathbb{Z}))$ of classifying all \emph{parametrized extensions} $x,y\in \mathrm{Z}( \Lambda,
\mathbb{Z})$ of $\Lambda$ by $\mathbb{Z}$ up to \emph{base-preserving isomorphism} (i.e., by whether $x-y\in\mathrm{B}( \Lambda ,
\mathbb{Z})$) is hyperfinite, and therefore of comparatively low complexity; see Section \ref{S:EquivalenceRelation}.  We show that the \emph{base-free isomorphism} classification problem   $\mathcal{R}(\mathrm{\mathrm{Aut}%
}( \Lambda) \curvearrowright \mathrm{Ext}( \Lambda ,\mathbb{
Z}))$ can, in contrast, be much more complex, where $\mathcal{R}(\mathrm{\mathrm{Aut}%
}( \Lambda) \curvearrowright \mathrm{Ext}( \Lambda ,\mathbb{
Z}))$ denotes the coarsening of $\mathcal{R}(\mathrm{Ext}( \Lambda ,\mathbb{
Z}))$ on $\mathrm{Z}( \Lambda,
\mathbb{Z})$ given by accounting for the natural action $\mathrm{\mathrm{Aut}
}( \Lambda)\curvearrowright\mathrm{Ext}( \Lambda,\mathbb{Z})$  provided by Theorem \ref{Theorem:main1}.

In particular, in the case of the
groups $\Lambda _{p}^{d}:=\mathbb{Z}[1/p]^{d}$ for a prime number $p$ and $
d\geq 1$, we prove that the orbit equivalence relations $E_{p}^{d}:=\mathcal{R}(\mathrm{\mathrm{Aut}}( \Lambda _{p}^{d}) \curvearrowright \mathrm{
Ext}( \Lambda _{p}^{d},\mathbb{Z}))$ array in the following
hierarchy.

\begin{theorem}
\label{Theorem:main3}Adopt the notations above. Fix $d,m\geq 1$ and primes $%
p,q$.

\begin{enumerate}
\item $E_{q}^{m}$ is not Borel reducible to $E_{p}^{d}$ if $m>d$;

\item $E_{q}^{m}$ is not Borel reducible to $E_{p}^{d}$ if $p$ and $q$ are
distinct and $m,d\geq 3$;

\item $E_{p}^{d}$ is hyperfinite and not smooth if $d=1$;

\item $E_{p}^{d}$ is not treeable if $d\geq 2$.
\end{enumerate}
\end{theorem}

The proof of Theorem \ref{Theorem:main3} relies on the description of $%
E_{p}^{d}$ as the orbit equivalence relation of the affine action $\mathrm{GL%
}_{d}( \mathbb{Z}[1/p]) \ltimes \mathbb{Z}[1/p]^{d}%
\curvearrowright \mathbb{Q}_{p}^{d}$, where $\mathbb{Q}_{p}$ is the additive
group of the field of $p$-adic numbers, as well as Ioana's cocycle
superrigidity result for profinite actions \cite{ioana_cocycle_2011} for property (T) groups, and  
ideas of Coskey and Thomas \cite%
{thomas_classification_2003,coskey_borel_2010}.

\subsection{The structure of the paper} In Section \ref{S:Prelim}, we 
give a brief review of the necessary definitions and collect some standard facts which we are going to need from the literature. In Section \ref{S:GroupsWithPCover}, we introduce the category of groups with a Polish cover. In Section \ref{S:Ulam}, we develop a general Ulam stability framework for groups with a Polish cover and prove Theorem \ref{Theorem:main2}, which is our main technical tool for analyzing the definable content of $\mathrm{Ext}(-,-)$  and $\mathrm{lim}^{1}(-)$. In Section \ref{Section:towers}, we analyze the definable content of $\mathrm{lim}^{1}(-)$, a functor playing an essential bookkeeping role in many computations in algebraic topology and constructions in homological algebra. In Section \ref{Section:locally-profinite}, we apply the rigidity results from Sections \ref{S:Ulam} and \ref{Section:towers} in the special case of profinite completions of $\mathbb{Z}^d$. We also establish a lifting property for actions on the quotients of such completions,  which will play a crucial role in Section \ref{S:Ext}. In Section \ref{S:Ext},  we analyze  the definable content of $\mathrm{Ext}(-,-)$ and prove Theorem \ref{Theorem:main1}. We also compute the \emph{abstract} group-isomorphism-type of the quotient $\hat{\mathbb{Z}}^d/\mathbb{Z}^d$ for an arbitrary profinite completion $\hat{\mathbb{Z}}^d$ of $\mathbb{Z}^d$. This will show how much information may be lost by neglecting the definable data about $\hat{\mathbb{Z}}^d/\mathbb{Z}^d$ contained in its Polish cover. In Section \ref{S:ActionsCocycleRig}, we use cocycle superrigidity  methods to study the Borel  complexity of the action $\mathrm{\mathrm{Aut}
}( \Lambda)\curvearrowright\mathrm{Ext}( \Lambda,\mathbb{Z})$, and conclude with the proof of Theorem \ref{Theorem:main3}.

\subsubsection*{Acknowledgments}

We would like to thank Sam Coskey for enlightening conversations regarding superrigidity, and Alexander Kechris for his feedback on our work. We are also grateful  to Vladimir Kanovei for his comments on an earlier version of this work and for referring us to his joint paper with Reeken \cite{kanovei_ulams_2000}. And we would like to thank the paper's two referees for their valuable readings of an earlier draft.


\section{Preliminaries}\label{S:Prelim}

In this section we review the basic facts from invariant descriptive set theory and category theory that we will require below. Standard references are \cite{kechris_classical_1995, gao_invariant_2009} for the former  and \cite{mac_lane_categories_1998} for the later.

\subsection{Polish spaces}

A \emph{Polish space} is a second countable topological space whose topology is induced by a complete metric. Let $X$ be a Polish
space. The \emph{$\sigma $-algebra of Borel subsets of $X$} is the smallest $\sigma $%
-algebra of subsets of $X$ that contains all the open subsets of $X$. Subsets of $X$ are
\emph{Borel }if they belong to its Borel $\sigma $-algebra. A function $f:X\rightarrow Y$
between Polish spaces is \emph{Borel} if the $f$-preimages of Borel sets are all Borel. The image of a Borel subset of $X$ under such a Borel function $f$
need not be a Borel subset of $Y$ but it will be, by definition, \emph{analytic}. If $f$, though, is an \emph{injective }%
Borel function then $f$ will map Borel subsets of $X$ to Borel subsets of $Y$ 
\cite[Theorem 15.1]{kechris_classical_1995}. The subspace topology renders a subset $A$ of a Polish space $X$ Polish if and only if $A$ is a $G_{\delta }$ subset
of $X$, i.e., if and only if $A$ is the intersection of a countable family of open subsets
of $X$. Hence all closed subspaces of a Polish space are Polish. A Polish space $X$ is \emph{locally compact} if
every point of $X$ has an open neighborhood with compact closure. This is
equivalent (in the Polish setting) to the assertion that $X$ can be written as an increasing union of compact subsets.

A subset $A$ of a topological space $X$ is \emph{meager} it is contained in the
union of countably many closed nowhere dense sets. It is \emph{nonmeager} if it is not
meager and \emph{comeager} if its complement is meager.
We say that $A$ has the \emph{Baire property} if it is contained in the smallest $\sigma$-algebra on $X$ generated by the open subsets and the meager subsets of $X$. A function $f\colon X\to Y$ between topological spaces is \emph{Baire measurable} if $f^{-1}(U)$ has the Baire property for every open set $U\subseteq Y$.
 By the Baire Category
Theorem \cite[Theorem 8.4]{kechris_classical_1995}, every Polish space $X$
is a \emph{Baire space}, i.e. every nonempty open subset of such an $X$ is
nonmeager. Hence the intersection of countably many dense $G_{\delta}$ subsets of $X$ will again be dense $G_{\delta}$. Observe lastly
that a subset of $X$ is comeager if and only if it contains a dense $%
G_{\delta }$ subset of $X$.

\subsection{Polish groups}

A \emph{topological group} is a group endowed with a topology rendering the group operations $(x,y)\mapsto x\cdot y$ and $x\mapsto x^{-1}$ continuous. 
A \emph{Polish group} is a topological group whose topology is Polish.  A subgroup $H$ of a Polish group $G$ endowed with the subspace topology is a Polish
group if and only if it is closed in 
$G$; see \cite[Proposition 2.2.1]{gao_invariant_2009}.
  In this case the quotient
group $G/H$, endowed with the quotient topology, is also a Polish group; see \cite[Theorem 2.2.10]{gao_invariant_2009}.
By \emph{the category of Polish groups} we mean the category with Polish groups as objects and continuous
homomorphisms as morphisms. The following fact, known  as  \emph{Pettis' Lemma}, will be used repeatedly; see \cite[Theorem 2.3.2]{gao_invariant_2009} for a proof. 
\begin{lemma}[Pettis] \label{Lemma: Pettis}
If a nonmeager subset $A$ of a Polish group $G$ has the Baire property then $AA^{-1}$ contains an open neighborhood of the identity of $G$.
\end{lemma}
\begin{corollary}\label{Corollary: Pettis}
If $\varphi :G\rightarrow H$ is a Baire-measurable group homomorphism between Polish groups, then $\varphi $ is continuous.
\end{corollary}
  Hence the morphisms of the aforementioned category might equivalently be taken to be the Baire measurable or  ``definable" homomorphisms between Polish groups. 

In the following, we will
regard any countable group as a topologically discrete, and therefore locally compact, non-Archimedean, Polish
group. Recall that 
a Polish group $G$ is \emph{locally compact} if it is a  locally compact
topological space and is \emph{non-Archimedean} if its
identity element admits a neighborhood basis consisting of open subgroups.  The main focus below will be on non-Archimedean abelian Polish groups. An abelian Polish group is non-Archimedean if and only if it is \emph{procountable}, i.e., is an inverse limit of a sequence of countable
groups \cite[Lemma 2]{malicki_abelian_2016}. Recall that a Polish group is \emph{profinite} if it is compact and non-Archimedean or, equivalently, it is isomorphic to the inverse limit of a tower of finite groups. We say that $G$ is \emph{locally profinite} if it admits a basis of neighborhoods of the identity consisting of profinite groups.
 Let $G$ be a non-Archimedean abelian Polish group and let $(V_n)$ be a decreasing sequence of open subgroups, beginning with $V_0=G$, whose intersection is the identity. It is easy to check that letting $d(g,h)=2^{-n}$ if and only if $n$ is the largest natural number with $gh^{-1}\in V_n$ defines a metric $d\leq 1$  on $G$ compatible with its topology such that:
\begin{enumerate}
\item $d$ is both \emph{left and right invariant}, i.e., $d(fg,fh)=d(g,h)$ and $d(hg,hf)=d(g,h)$, for all $f,g,h\in G$;
\item $d$ is an \emph{ultrametric}, i.e.,  $d(f,h)\leq \max\{d(f,g),d(g,h)\}$, for all $f,g,h\in G$.
\end{enumerate} 
Moreover, as is straightforward to verify, this metric is also complete; see \cite[Corollary 2.2.2]{gao_invariant_2009}.

A \emph{standard Borel structure} on a set $X$ is a collection $\mathrm{B}$
of subsets of $X$ forming the $\sigma $-algebra of Borel sets with
respect to some Polish topology on $X$. A \emph{standard Borel space }$\left( X,\mathrm{B}\right) $ is a set $X$ endowed with a Borel structure $%
\mathrm{B}$. If $\left( X,\mathrm{B}\right) $ and $\left( X^{\prime },%
\mathrm{B}^{\prime }\right) $ are standard Borel spaces, then a \emph{Borel
function from $\left( X,\mathrm{B}\right) $ to $\left( X^{\prime },\mathrm{B}^{\prime }\right)$} is, as above, one that is
measurable with respect to $\mathrm{B}$ and $\mathrm{B}^{\prime }$. 
A 
\emph{standard Borel group} is a group endowed with a standard Borel
structure with the property that all the group operations are Borel. A
standard Borel group is \emph{Polishable} if there exists a Polish topology
that induces its Borel structure \emph{and} renders it a Polish group. It follows from Pettis' Lemma that such a
topology is always unique. 
The image of any continuous group homomorphism from one Polish group to another is a Polishable standard Borel
group. This fact forms part of our next lemma, which we state after recalling a few more definitions.

If $X$ is a standard Borel space and 
$E$ is an equivalence relation on $X$ then a Borel \emph{selector} $s$ for $%
E$ is a Borel function $s:X\rightarrow X$ with the property that, for every $%
x,y\in X$, $x\,E\,s( x) $, and $x\, E\, y$ if and only if $s( x)
=s( y) $ (thus $s$ selects, in a Borel manner, a point from each equivalence class).
Any subgroup $H$ of a Polish group $G$ induces a canonical equivalence relation on $G$, namely that whose classes are the cosets of $H$; this equivalence relation admits
a Borel selector if and only if $H$ is a closed subgroup of $G$ \cite[Theorem 12.17]
{kechris_classical_1995}.
Suppose
that $X$ and $Y$ are standard Borel spaces and $f:X\rightarrow Y$ is a Borel function
such that $f( X) $ is a Borel subset of $Y$. Then a Borel \emph{section}
for $f:X\rightarrow Y$ is a Borel function $g:f( X) \rightarrow X$ such
that $f\circ g$ is the identity of $f( X) $. An important special
case is when $X=Y\times Z$ for some standard Borel space $Z$ and $f$ is the
first-coordinate projection map. The following lemma is a corollary
of the existence of Borel selectors for closed coset equivalence relations.

\begin{lemma}
\label{Lemma:Borel-inverse}Let  $A,B$ be Polish groups and  let $\pi
:A\rightarrow B$ be a continuous homomorphism. Then $\mathrm{ran}(\pi )$ is
a Polishable Borel subgroup of $B$, and $\pi :A\rightarrow \mathrm{ran}(\pi) 
$ has a Borel section.
\end{lemma}
\begin{proof}
Observe that $\pi $ induces an injective Borel map $\hat{\pi}\colon A/
\mathrm{ker}( \pi) \rightarrow B$ whose image is $\textrm{ran}\left( \pi \right) $. 
 Hence $\mathrm{ran}( \pi) $ is a Borel subgroup of $B$. Let 
$\rho \colon A\rightarrow A$ be the Borel selector for the coset equivalence
relation of $\mathrm{ker}(\pi )$ in $A$. Since $\rho $ is constant on cosets
of $\mathrm{ker}(\pi )$, it induces Borel map $\hat{\rho}\colon A/\mathrm{ker%
}(\pi )\rightarrow A$. Thus $\hat{\rho}\circ (\hat{\pi})^{-1}$ is a Borel
section of $\pi :A\rightarrow \mathrm{ran}( \pi) $.
Finally, since $\mathrm{ker}(\pi)$ is closed in $A$ the group $A/\mathrm{ker}( \pi)$ is Polish in the quotient topology; see  \cite[Proposition 2.2.1]{gao_invariant_2009}. Hence, $\mathrm{ran}(\pi )$ admits a Polish topology; namely,  the push-forward topology of $A/\mathrm{ker}( \pi)$ under the injective map  $\hat{\pi}$. The fact that this new topology induces the same Borel structure on $\mathrm{ran}(\pi )$ follows from the existence of the Borel section.
\end{proof}

A more general version of Lemma \ref{Lemma:Borel-inverse} is the following
result, which is a particular instance of \cite[Lemma 3.8]{kechris_borel_2016}.
\begin{lemma}
\label{Lemma:Borel-inverse2}

Let  $A,B$ be Polish groups, let $A_{0}\subseteq A$ and $B_{0}\subseteq B$ Polishable subgroups, and  let $
C$ be a Borel subgroup of $A$. If  $f:C\rightarrow B$ is a
Borel function with $f^{-1}( B_{0}) =A_{0}\cap C$, $f(
C) +B_{0}=B$, and $f( x)\,f( y)\,f( xy)
^{-1}\in B_{0}$ for every $x,y\in C$, then $f$ has a Borel section $g$
which furthermore satisfies $g( x)\,g( y)\, g(
xy) ^{-1}\in A_{0}$ for every $x,y\in B$.
\end{lemma}

\subsection{Definable equivalence relations}\label{S:EquivalenceRelation}

Here we review the basic notions of the Borel complexity
theory of classification problems. Formally, a \emph{classification problem} is a pair $
\left( X,E\right)$ where $X$ is a Polish space and $E$ is an equivalence
relation on $X$ which is analytic (and, in many cases, Borel) as a subset
of $X\times X$. 
 A \emph{Borel homomorphism }(respectively, a \emph{Borel
reduction}) from $\left( X,E\right) $ to $\left( X^{\prime },E^{\prime
}\right) $ is a  function (respectively, an injective function) $%
X/E\rightarrow X^{\prime }/E^{\prime }$ induced by some Borel \emph{lift} $%
X\rightarrow X^{\prime }$.  In the standard reading, the existence of a
Borel reduction from $\left( X,E\right) $ to $\left( X^{\prime },E^{\prime}\right) $ is tantamount to the assertion that the classification problem $(X,E)$ is at most as hard as the classification problem $(X'
E^{\prime})$.  If such a Borel reduction exists, we say that $E$ is \emph{Borel reducible} to $E^{\prime }$ and write $E\leq _{B}E^{\prime }$.  If additionally $E^{\prime }\leq _{B}E$ then $E$ and $E^{\prime }$ are \emph{bireducible}; if not, then we write $E<_{B}E^{\prime
}$. Some of the Borel reductions appearing below will reflect yet stronger relations between classification problems. Following \cite%
{motto_ros_descriptive_2013,motto_ros_complexity_2012,camerlo_isometry_2018}
we say that $E$ and $E^{\prime }$ are\emph{\ classwise Borel isomorphic} if
there is a bijection $f:X/E\rightarrow X^{\prime }/E^{\prime }$ such that $f$
is induced by a Borel map $X\rightarrow X^{\prime }$, and $f^{-1}$ is
induced by a Borel map $X^{\prime }\rightarrow X$. Such an $f$ is called a \emph{%
classwise Borel isomorphism} from $E$ to $E^{\prime }$.

\begin{example}\label{Example:Rank_d}
For each $d\in\omega$, let $(\mathcal{R}(d),\simeq_{\mathrm{iso}})$ be the problem of classifying all rank $d$ torsion-free abelian  groups up to isomorphism. $\mathcal{R}(d)$ readily identifies with a closed subset of the product $\{0,1\}^{\mathbb{Q}^d}$ and thereby admits a natural Polish structure. It is not hard then to see that  $\simeq_{\mathrm{iso}}$ is a Borel equivalence relation on $\mathcal{R}(d)$. While  Baer's analysis of rank $1$ torsion-free abelian groups provides concrete invariants completely classifying $(\mathcal{R}(1),\simeq_{\mathrm{iso}})$, no satisfactory\footnote{The complete invariants originating from the work of Derry \cite{Der}, Mal'cev \cite{Mal},  and  Kurosch \cite{Kur} are for all practical purposes  as  complicated  as the original classification problem is; see \cite{fuchs_abelian_2015}.}  complete invariants are known for $(\mathcal{R}(d),\simeq_{\mathrm{iso}})$ when $d$ is greater than $1$.
 The framework of Borel complexity theory helps to explain this fact: by the cumulative work of several authors \cite{adams_linear_2000,hjorth_1999_rank2,thomas_classification_2003} 
we now know that the problems $(\mathcal{R}(d),\simeq_{\mathrm{iso}})$ form a strictly increasing chain in the Borel reducibility order:
\[(\mathcal{R}(1),\simeq_{\mathrm{iso}})<_B (\mathcal{R}(2),\simeq_{\mathrm{iso}})<_B (\mathcal{R}(3),\simeq_{\mathrm{iso}})<_B \cdots\]
\end{example}

Within this framework, a number of complexity classes of equivalence relations are of a sufficient importance to merit separate names. The lower part of this hierarchy is stratified  by the following complexity classes:
\[\mathrm{smooth}\subsetneq\mathrm{essentially\; hyperfinite}\subsetneq\mathrm{essentially\; treeable}\subsetneq\mathrm{essentially\; countable},\]
where a Borel equivalence relation $(X,E)$ is:

\begin{itemize}
\item \emph{smooth} if it is Borel reducible to the relation $\left( \mathbb{%
R},=_{\mathbb{R}}\right) $ of equality on the real numbers (or,
equivalently, to the relation of equality on any other uncountable Polish space);

\item \emph{hyperfinite} if it can be written as an increasing union of
equivalence relations with finite classes or, equivalently, if it is Borel
reducible to the orbit equivalence relation of a Borel action of the
additive group of integers $\mathbb{Z}$ \cite[Proposition 1.2]%
{jackson_countable_2002}; see also \cite[Theorem 5.1]%
{dougherty_structure_1994};

\item \emph{treeable }if it is countable and there exists a Borel relation $%
R\subseteq X\times X$ such that $\left( X,R\right) $ is an acyclic graph,
and the connected components of $\left( X,R\right) $ are precisely its
equivalence classes or, equivalently, it is Borel reducible to the orbit
equivalence relation induced by a \emph{free }action of a free countable
group \cite[Section 3]{jackson_countable_2002};

\item \emph{countable }if each one of its equivalence classes is countable;

\item \emph{essentially }countable/hyperfinite/treeable, respectively, if it
is Borel reducible to an equivalence relation that is
countable/hyperfinite/treeable, respectively.
\end{itemize}
Among the non-smooth Borel equivalence relations, there is a least one up to
Borel bireducibility, which is the relation $\left( \left\{ 0,1\right\}
^{\omega },E_{0}\right),$ of eventual equality of binary sequences: $xE_{0}y \iff \exists m\in\omega \; \forall n >m \; x(n)=y(n)$.
This is
a hyperfinite Borel equivalence relation. Any other non-smooth
hyperfinite Borel equivalence relation is Borel bireducible with $E_{0}$ 
\cite[Theorem 7.1]{dougherty_structure_1994}.  

Many classification problems are induced by continuous actions of Polish group actions. 

\begin{definition}Corresponding to any Borel action $G\curvearrowright X$ of a Polish group $G$ on a standard Borel space $X$ is a classification problem $(X,\mathcal{R}(
G\curvearrowright X))$, where $\mathcal{R}(
G\curvearrowright X) $ is the associated \emph{orbit equivalence
relation}, i.e., the equivalence relation given by $( x,y) \in \mathcal{R}(
G\curvearrowright X) $ if and only if $g\cdot x=y$ for some $g\in G$.
\end{definition}
For example, $E_0$ is simply the orbit equivalence relation of the action of the countable group $\oplus_{n\in\omega}\mathbb{Z}/2\mathbb{Z}$ on  $\prod_{n\in\omega}\mathbb{Z}/2\mathbb{Z}$ by addition and the relation $\simeq_{\mathrm{iso}}$ on $\mathcal{R}(d)$ is induced by an action of $\mathrm{GL}_{d}(\mathbb{Q})$ on $\mathcal{R}(d)$.

\subsection{Categories}

Let  $F\colon \mathcal{C}\to \mathcal{D}$ be a functor  between categories, and for every two objects $x,y$ from $\mathcal{C}$ let $F_{x,y}$ be the associated map  from $\mathrm{Hom}(x,y)$  to $\mathrm{Hom}(F(x),F(y))$. The functor $F$ is called \emph{full} (respectively: \emph{faithful}, or \emph{fully faithful}), if $F_{x,y}$ is surjective (respectively: injective, or bijective) for every $x,y\in\mathcal{C}$. It is called \emph{essentially surjective}, if for every object $d\in\mathcal{D}$ there is an object $x\in \mathcal{C}$ so that $F(x)$ and $d$ are isomorphic in $\mathcal{D}$. The categories $\mathcal{C}$ and $\mathcal{D}$ are \emph{equivalent} if there are functors $F\colon \mathcal{C}\to \mathcal{D}$ and $G\colon \mathcal{D}\to \mathcal{C}$ so that both $G\circ F$ and $F\circ G$ are naturally isomorphic to the corresponding identity functors. Note that $\mathcal{C}$ and $\mathcal{D}$  are equivalent if and only if there is a fully faithful and essentially surjective functor $F\colon \mathcal{C}\to \mathcal{D}$; see \cite{mac_lane_categories_1998}.

Recall that an $\mathbf{Ab}$-category is a category such that every hom-set
is an abelian group, and such that composition with any given arrow defines
a group homomorphism (i.e., composition distributes over addition) \cite[Section 1]{weibel_introduction_1994}. A functor between $\mathbf{Ab}$-categories is \emph{additive} if it induces group homomorphisms at the level of
the hom-sets. A \emph{zero object} in a category is an object that is both
initial and terminal. An \emph{additive category} is an $\mathbf{Ab}$-category possessing a zero object and finite products. Abelian Polish groups form an
additive category, in which the sum of two continuous group homomorphisms is just
their pointwise sum. Here the zero object is the zero group, and the product of a
sequence $\left( G_{n}\right) $ of Polish group is the product group $\prod_{n}G_{n}$ endowed with the product topology. Notice that the
category of abelian Polish groups is not an abelian category. Indeed, in this category the kernel of a continuous group homomorphism $\varphi :G\rightarrow H$ is $\{ x\in G:\varphi (x)=0\} $, while its cokernel is the quotient of $H$ by the \emph{closure} of the image of $\varphi $. In an abelian category, an arrow whose kernel and cokernel are both zero must be an isomorphism. This fails in the category of abelian Polish groups, as for instance, if $i$ is
the inclusion of $\mathbb{Q}$ in $\mathbb{R}$ then $i$ is has trivial kernel and cokernel, but it is not an isomorphism.

\section{Groups with a Polish cover}\label{S:GroupsWithPCover}

The central objects in all that follows are \emph{groups with a Polish cover}. These are topological groups augmented with a Polish group extension (i.e., a \emph{cover}) which serves, in practice, as the formal setting for definability analyses. We describe in this section the \emph{definable homomorphisms} which together with these objects comprise the \emph{category of groups with a Polish cover}. We will see in later sections that the aforementioned larger groups or covers arise naturally throughout algebraic topology and homological algebra, and that systematic attention to them can substantially refine many of the classical invariants of both fields.
 
 \begin{definition}
A \emph{group with  Polish cover} is a pair $\mathcal{G}=(N,G)$ where $G$ is a Polish group and $N\subseteq G$ is a Polishable normal subgroup.  We  often represent a group with a Polish cover $\mathcal{G}=\left(N,G\right) 
$ simply by its \emph{quotient} $G/N$.
\end{definition}

Let $\mathcal{G}=(N,G)$ and $\mathcal{G}^{\prime }=(N^{\prime },G^{\prime })$
be groups with a Polish cover and let $f\colon G/N\rightarrow G^{\prime
}/N^{\prime }$ be a group homomorphism. A function $\hat{f}:G\rightarrow
G^{\prime }$ is a \emph{lift of $f$} if $f(xN)=\hat{f}(x)N^{\prime }$, for
every $x\in G$. Notice that a lift $\hat{f}:G\rightarrow G^{\prime }$ of $f$
is not necessarily a group homomorphism. Necessary and sufficient conditions
for a function $\varphi \colon G\rightarrow G^{\prime }$ to be a lift of
some homomorphism $G/N\rightarrow G^{\prime }/N^{\prime }$ are for it to
satisfy $\varphi (N)\subseteq N^{\prime }$ and $\varphi (xy)^{-1}\,\varphi
(x)\,\varphi (y)\in N^{\prime }$ for all $x,y\in G$. Indeed, if $\varphi $
is a lift of a group homomorphism $f$ then, since $f$ maps the trivial
element to the trivial element, we must have $\varphi \left( N\right)
\subseteq N^{\prime }$. Furthermore, for $x,y\in G$, $f\left( xy\right)
=f\left( x\right) f\left( y\right) $, whence $\varphi \left( xy\right)
N^{\prime }=\varphi \left( x\right) \varphi \left( y\right) N^{\prime }$,
and $\varphi \left( xy\right) ^{-1}\,\varphi \left( x\right) \,\varphi
\left( y\right) \in N^{\prime }$. Conversely, if $\varphi :G\rightarrow
G^{\prime }$ satisfies $\varphi (N)\subseteq N^{\prime }$ and, for $%
x,x^{\prime }\in G$, $\,\varphi (xy)^{-1}{}\varphi (x)\,\varphi (y)\,\in
N^{\prime }$ and in particular $\varphi \left( x^{-1}\right) {}\varphi
\left( x\right) \in N^{\prime }$, then we can define $f:G/N\rightarrow
G^{\prime }/N^{\prime }$ by setting $f\left( xN\right) :=\varphi \left(
x\right) N^{\prime }$. This is well-defined since whenever $x,x^{\prime }\in
G$ are such that $x^{-1}x^{\prime }\in N$ then $\varphi \left(
x^{-1}x^{\prime }\right) \in N^{\prime }$ and $\varphi \left( x\right)
N^{\prime }=\varphi \left( x\right) \varphi \left( x^{-1}x^{\prime }\right)
N=\varphi \left( x^{\prime }\right) N$. Furthermore, it is a homomorphism
since for $x,y\in G$ we have that $\varphi \left( xy\right) N^{\prime
}=\varphi \left( xy\right) \left( \varphi (xy)^{-1}{}\varphi (x)\,\varphi
(y)\right) N^{\prime }=\varphi \left( x\right) \varphi \left( y\right)
N^{\prime }$.

\begin{definition}\label{Def:definable}
Let $\mathcal{G}=(N,G)$  and $\mathcal{G}'=(N',G')$ be groups with a Polish cover.
\begin{itemize}
\item A \emph{Borel-definable homomorphism}, or more succinctly a \emph{definable homomorphism}, from $\mathcal{G}$ to $\mathcal{G}'$  is  a  group homomorphism $f\colon G/N \to G'/N'$ which admits a Borel function $\hat{f}\colon G\to G'$ as a lift. 
\item A \emph{Borel-definable isomorphism}, or more succinctly a \emph{definable isomorphism}, from $\mathcal{G}$ to $\mathcal{G}'$  is  a  group isomorphism $f\colon G/N \to G'/N'$ which admits a Borel function $\hat{f}\colon G\to G'$ as a lift. 
\end{itemize}
The \emph{category of groups with a Polish cover} is the category whose objects are groups with a Polish cover and whose morphisms are the definable homomorphisms.
\end{definition}
\begin{remark}By Lemma \ref{Lemma:Borel-inverse2}, a group homomorphism  $f\colon G/N \to G'/N'$ is a definable isomorphism from $\mathcal{G}$ to $\mathcal{G}'$ if and only if it is an isomorphism in the category of groups with a Polish cover.
\end{remark}

We will occasionally also consider \emph{topological homomorphisms} from $\mathcal{G}$ to $\mathcal{G}'$; these are group homomorphisms $f\colon G/N \to G'/N'$ which admit a \emph{continuous} topological group homomorphism $\hat{f}\colon G\to G'$ as a lift.
Notice that the collection of all topological homomorphisms between groups with a Polish cover forms a subcategory of the category from Definition \ref{Def:definable}; we term this category the \emph{topological category of groups with a Polish cover}. 

\begin{example}\label{Example:Kanovei}
Consider the group with a Polish cover $(\mathbb{Q},\mathbb{R})$. When $\mathbb{R}/\mathbb{Q}$ is viewed as an abstract group then it is isomorphic to a $\mathbb{Q}$-vector space of dimension $2^{\aleph_0}$. Hence there are exactly $2^{2^{\aleph_0}}$ abstract group homomorphisms from   $\mathbb{R}/\mathbb{Q}$ to $\mathbb{R}/\mathbb{Q}$. Of course, the majority of these homomorphisms are not ``definable" by any concrete formula but exist, rather, as essentially abstract effects of the axiom of choice.
One may also view $\mathbb{R}/\mathbb{Q}$ as a topological group with the quotient topology induced from $\mathbb{R}$. This is a decidedly uninformative topology, since its only open sets are $\emptyset$ and $\mathbb{R}$. In contrast, viewing $\mathbb{R}/\mathbb{Q}$ as a \emph{group with a Polish cover} affords us access to the ``definable" subsets of $\mathbb{R}/\mathbb{Q}$: these are precisely the $\mathbb{Q}$-invariant Borel subsets of $\mathbb{R}$.
Note that since $\mathbb{Q}$ is dense in $\mathbb{R}$, such sets are always meager or comeager in $\mathbb{R}$; see \cite[Proposition 6.1.9]{gao_invariant_2009}. Moreover, the definable homomorphisms from $\mathbb{R}/\mathbb{Q}$ to $\mathbb{R}/\mathbb{Q}$ turn out to be exactly those which are induced by maps from $\mathbb{R}$ to $\mathbb{R}$ of the form  $x\mapsto c \cdot x$ for some $c\in\mathbb{R}$; see  \cite{kanovei_baire_2000}.
\end{example}

\begin{example}
Suppose that $\mathcal{I}$ and $\mathcal{J}$ are Polishable ideals of $\mathcal{P}(
\omega)$; for definitions, see \cite{SOLECKI199951}.  The symmetric difference operation $\triangle$ then endows $\mathcal{I}$, $\mathcal{J}$, and $\mathcal{P} ( \omega) $ with Polish abelian group structures. Furthermore, $\mathcal{P}
( \omega) /\mathcal{I}$ and $\mathcal{P}
( \omega) /\mathcal{J}$ are then groups with  Polish cover, and an
injective definable homomorphism from $\mathcal{P} ( \omega ) /\mathcal{I}$
to $\mathcal{P} ( \omega ) /\mathcal{J}$ is the same as a Borel $( 
\mathcal{I},\mathcal{J}) $-approximate $\triangle $-homomorphism
in the sense of \cite[Section 3.2]{kanovei_borel_2008}.
\end{example}

We identify any Polish group $G$ with its corresponding \emph{group with a} trivial \emph{Polish cover} $G=G/N$ where $N$ denotes its trivial subgroup. Since, as noted above, every definable homomorphism between Polish groups is a topological homomorphism this identification realizes the category of Polish groups as a full
subcategory of the category of groups with  Polish cover.
To each group with  Polish cover $\mathcal{G}$ we assign the following two groups with  Polish cover $\mathcal{G}^{\mathrm{w}},\mathcal{G}^{\infty }$ which measure in some sense how far $G/N$ is from being Polish:
\begin{itemize}
\item The \emph{weak} \emph{group} of $\mathcal{G}$ is the Polish group $
\mathcal{G}^{\mathrm{w}}:=G/\overline{N}$, where $\overline{N}$
is the closure of $N$ inside $G$.
\item The \emph{asymptotic group} of $\mathcal{G}$ is the group with a
Polish cover $\mathcal{G}^{\infty}:=\overline{N}/N$.
\end{itemize}
Clearly the inclusion map $\overline{N}/N \rightarrow G/N$
is an injective topological homomorphism $\mathcal{G}^{\infty }\rightarrow 
\mathcal{G}$, while the quotient map $G/N \rightarrow G/\overline{N}$ is a surjective topological homomorphism $\mathcal{G}\rightarrow 
\mathcal{G}^{\mathrm{w}}$.

\begin{lemma}
The assignments $\mathcal{G}\mapsto \mathcal{G}^{\infty }$ and  $\mathcal{G}\mapsto \mathcal{G}^{\mathrm{w}}$ are functors from
the topological category of homomorphisms between groups with a Polish cover to itself and  $\mathcal{G}^{\mathrm{w}}$ is always a Polish group.
\end{lemma}

\begin{proof}
By definition, a topological homomorphism $f:G/N\rightarrow G'/N'$ admits a lift to a
continuous homomorphism $\hat{f}:G\rightarrow G'$ such that $\hat{f}
( N) \subseteq N'$. Since $\hat{f}$ is continuous, it follows that $\hat{f}(\overline{N})\subseteq \overline{N'}$, which in
turn implies that $\hat{f}$ induces a topological homomorphism $\mathcal{G}^{\infty }\rightarrow (\mathcal{G}')^{\infty }$ and a continuous homomorphism $\mathcal{G}^{\mathrm{w}}\rightarrow (\mathcal{G}')^{\mathrm{w}}$. The last claim follows from \cite[Proposition 2.2.10]{gao_invariant_2009}.
\end{proof}

To each group with a Polish cover one may assign a classification problem in the sense of Section \ref{S:EquivalenceRelation}:

\begin{definition}
To each group  with a Polish cover $\mathcal{G}=G/N$ we associate the classification problem $(G,\mathcal{R}( \mathcal{G}))$, where  $\mathcal{R}( \mathcal{G}) $ is the \emph{coset equivalence relation} of $%
N$ inside $G$: for  $x,y\in G$ set $\left( x,y\right) \in \mathcal{R}( \mathcal{G})$ if and only if  $N
x= N y$.
\end{definition}

Notice that  $\mathcal{R}( \mathcal{G}) $ may be viewed as the orbit equivalence relation $\mathcal{R}(N\curvearrowright G )$ of the action of the Polish group $N$ on $G$ by translation $(h,g)\mapsto h\cdot g$. It turns out that, at least for abelian groups, a necessary condition for the coset equivalence relation of a Borel subgroup $N$ of some Polish group $G$ to be Borel reducible to an orbit equivalence relation $\mathcal{R}(H\curvearrowright Y)$ of a continuous Polish group action $H\curvearrowright Y$, is that $N$ is Polishable; see \cite{Sol}. In analogy with the terminology from Section \ref{S:EquivalenceRelation}, we say that a group with a Polish cover $\mathcal{G}=G/N$ is \emph{smooth}
if $N$ is a closed subgroup of $G$, in which case the quotient topology renders $G/N$ a
Polish group. Notice that if $\mathcal{G}=(N,G)$ is smooth then it is definably isomorphic to the group with the trivial Polish cover $(1_{G/N},G/N)$. In fact, the following is true:

\begin{lemma}
If  $g$ is a definable isomorphism between the groups with a Polish cover  $\mathcal{G}= (N,G)$ and $\mathcal{G}' = (N',G')$ and either $ G/N$ or $G'/N'$ is a smooth group with a Polish cover, then so is the other.
\end{lemma}
\begin{proof}
It will suffice by symmetry to assume that $\mathcal{G}'$ is a smooth group with a 
Polish cover. We may further assume that $\mathcal{G}'$ is the trivial 
Polish cover, i.e.,  $N'$ is the trivial subgroup $1_{G'}$ of $G'$. By assumption,  there a group isomorphism $g\colon G/N \to G'$ and a  Borel function $\widehat{g}:G\rightarrow G'$ so that $g \circ \pi =\widehat{g}$, where $\pi :G\rightarrow G/N$ is the quotient map. So $\widehat{g}$ is a homomorphism. By Corollary \ref{Corollary: Pettis}, $\widehat{g}$ is a continuous homomorphism and therefore $N= 
\mathrm{ker}( \widehat{g})$ is closed.
\end{proof}

Notice that the category of groups with a Polish cover admits an object that is both initial and  terminal; namely, the \emph{identity} group with a Polish cover $G/N$, where both $G$ and $N$ are trivial. It is moreover closed under countable products. The \emph{product} of any  given countable collection $G_n/N_n$ of groups with a Polish cover is the group with a Polish cover $G/N$,
where
\begin{equation*}
N=\prod_{n\in \omega }N_n \subseteq \prod_{n\in
\omega }G_n=G,
\end{equation*}%

By an \emph{abelian group with a Polish cover} we mean a group with a Polish cover $\mathcal{G}=(N, G)$ so that $G/N$ is an abelian group. Note that this  does not necessarily imply that $G$ is abelian. For example, if $A$ is a separable stable continuous-trace C*-algebra, $\mathrm{Aut}(A)$ is the group of automorphisms of $A$, and $\mathrm{Inn}(A)$ is the Polishable subgroup of $\emph{Aut}(A)$ consisting of automorphisms of $A$ that are \emph{inner}, then $\emph{Aut}(A)$ is in general not commutative. However, the group with a Polish cover $\emph{Out}(A):=\emph{Aut}(A)/\emph{Inn}(A)$ is abelian, being isomorphic to the second \v{C}ech cohomology group with integer coefficients of the spectrum of $A$ by the main theorem of \cite{phillips_automorphisms_1980}.

 Suppose that $G/N$ and $G'/N'$ are abelian
(additively denoted) groups with a Polish cover, where $G$ and $G'$
are additively denoted as well (but not necessarily abelian). Suppose that $f ,g :G/N\rightarrow G'/N'$ are definable homomorphisms.
Then $f +g :G/N\rightarrow G'/N'$ is the definable
homomorphism defined by $( f +g) ( x)
=f( x) +g( x) $. Notice that $f +g 
$ is indeed a definable homomorphism. If $\hat{f}$ and $\hat{g}$ are Borel lifts of $f 
$ and $g$, respectively, then the function $G\rightarrow G'$, $
x\mapsto \hat{f}( x) +\hat{g}( x) $ is a Borel lift of $f
+g$. This turns the $\mathrm{Hom}(G/N,G'/N')$-set of all definable  homomorphisms from $G/N$ to $G'/N'$ into an abelian group. Hence, the category of  \emph{abelian groups with a Polish cover} is an additive
category. So too is its subcategory the \emph{topological category of  abelian groups with a Polish cover}.

\section{Ulam stability of non-Archimedean abelian groups}\label{S:Ulam}

In this section we prove several \emph{Ulam stability} results for definable homomorphisms $f\colon G/N\to G'/N'$ between groups with a Polish cover, where $N$ is a dense subgroup of a non-Archimedean abelian Polish group. Our main result is Theorem \ref{Theorem:rigid}; it handles cases in which $N'$ is countable. In Theorem \ref{Theorem:rigid2} we attain similar results under either of the weaker assumptions that $N'$ is locally profinite in its Polish topology, or that $N'$ is non-Archimedean in its Polish topology.
In all this section's paragraphs except the next, we adopt an additive notation for group operations.

Ulam stability phenomena were first considered in \cite{ulam_problems_1964,mauldin_mathematical_1987} in the context of homomorphisms  between metric groups. A map $f\colon G\to  G$ from a metric group $(G,d)$ to itself is  an \emph{$\varepsilon$-automorphism} if  $d(f(x)\cdot f(y),f(x\cdot y))<\varepsilon$ for all $x,y\in G$. Ulam asked for which such metric groups $(G,d)$ and  $\varepsilon>0$ is there some $\delta>0$ so that each  $\varepsilon$-automorphism is $\delta$-close to an actual  homomorphism.
More generally, \emph{Ulam stability} phenomena may be considered in any context where one has notions of \emph{morphisms} and \emph{approximate morphisms}, as well
as a way to measure how far apart a given morphism and a given approximate morphism are
from each other, as described in our introduction.
The series \cite{kanovei_baire_2000,
kanovei_ulams_2000,kanovei_ulam_2001} may be regarded as an exploration of Ulam stability in the setting of groups with a Polish cover. As mentioned, for example, in \cite{kanovei_baire_2000} it is shown that if $N,N'$ are subgroups of the Polish group $(\mathbb{R},+)$ and $N'$ is countable, then every definable homomorphism $f\colon\mathbb{R}/N\to \mathbb{R}/N'$ lifts to a map of the form $x\mapsto c\cdot x$. 
 In this example, a \emph{morphism} is any continuous group homomorphism  $\psi\colon \mathbb{R}\to \mathbb{R}$ with $\psi(N)\subseteq N'$, an \emph{approximate morphism} is any Borel map $\varphi\colon \mathbb{R}\to \mathbb{R}$ that is a lift of a (definable) group homomorphism $f\colon \mathbb{R}/N\to \mathbb{R}/N'$, and $\varphi,\psi\colon\mathbb{R}\to \mathbb{R}$ are considered to be close to each other if they are lifts of the same  group homomorphism.  The question of whether similar rigidity phenomena hold in the case of the \emph{$p$-adic
groups} is posed in \cite[Section 8]{kanovei_ulam_2001}.  Theorem \ref{Theorem:rigid} below settles the more general case where the cover $G$ of $G/N$ is an arbitrary non-Archimedean abelian Polish group. Some implications of this theorem for the case of the $p$-adic groups  are discussed separately in Section \ref{Section:locally-profinite}.

We begin by describing a framework for studying Ulam stability phenomena surrounding definable homomorphisms between arbitrary groups with a Polish cover. 
Let $G/N$ be a group with a Polish cover. An \emph{essential retract} of $G/N$ is a clopen subgroup $H$ of $G$ with the property that $H$ intersects every $N$-coset in $G$. Notice that $H/(H\cap N)$ is a group with a Polish cover and  the map $i^*\colon H/(H\cap N)\to G/N$, induced by the inclusion $i\colon H\to G,$ is a definable isomorphism.  Essential retracts satisfy the following extension property, showing that an essential retract of $G/N$ is `essentially' the same group with a Polish cover as $G/N$ up to an isomorphism induced by the inclusion map.

\begin{lemma}\label{L:RetractProperty}
Let $H$ be an essential retract of the group with a Polish cover  $G/N$. Then the definable isomorphism  $(i^*)^{-1}\colon G/N \to H/(H\cap N)$ lifts to a Borel map $r\colon G\to H$ with $ r \circ i=\mathrm{id}_H$. In particular, if   $\hat{f}$ is a Borel lift of  a definable homomorphism $f\colon H/(H\cap N) \to G'/N'$ then $\hat{f}$ extends to a Borel lift of $f\circ (i^*)^{-1}\colon G/N\to G'/N'$.
\end{lemma}
\begin{proof}

The Effros (standard) Borel space $\mathcal{F}(X)$ of closed subspaces of a
Polish space $X$ has the closed subsets of $X$ as points, and it is endowed
with the Borel structure generated by the families $\{F\in \mathcal{F}%
(X):F\cap U\neq \emptyset \}$, where $U$ varies among the open subsets of $X$ 
\cite[Section 12.C]{kechris_classical_1995}. Notice that, for every $g\in G$%
, $N\cap g^{-1}H$ is a closed subset of $N$ with respect to its Polish
group topology. Furthermore, the assignment $H\rightarrow \mathcal{F}\left(
N\right) $, $x\mapsto N\cap x^{-1}H$ is Borel. Indeed, let $\left\{
W_{n}:n\in \omega \right\} $ be a basis of open subsets of $N$. We use the
notation for the Vaught transform in reference to the continuous right
action of $N$ on $G$ by right translation \cite[Section 16.B]%
{kechris_classical_1995}. If $U$ is an open subset of $N$, then $N\cap
x^{-1}H\cap U$ is nonempty if and only if there exists $n\in \omega $ such
that $W_{n}\subseteq x^{-1}H$. For $n\in \omega $, we have that, since $H$ is closed in $G$, $W_{n}\subseteq x^{-1}H\Leftrightarrow xW_{n}\subseteq H\Leftrightarrow
\forall ^{\ast }w\in W_{n}$, $xw\in H\Leftrightarrow x\in H^{\ast W_{n}}$. This concludes the proof that the assignment $H\rightarrow \mathcal{F}%
\left( N\right) $, $x\mapsto N\cap x^{-1}H$ is Borel. The
Kuratowski--Ryll-Nardzewski theorem \cite[Theorem 12.13]%
{kechris_classical_1995} implies that there is a Borel function $\sigma :G\rightarrow N$
such that $\sigma (g)\in N\cap g^{-1}H$ for every $g\in G$ and $\sigma(g)=g$ for
all $g\in H$. Define then $r(g):=g\sigma (g)$ for every $g\in G$. Finally,
notice that $\hat{f}\circ r$ is the desired lift of $f\circ (i^{\ast })^{-1}$%
.
\end{proof}

\begin{definition}
\label{Definition:trivial}  Let  $f\colon G/N\to G'/N'$ 
be a definable homomorphism between groups with a Polish cover. 
We say that  $f$ is \emph{trivial} if it admits a Borel lift $\hat{f}\colon G\to G'$ whose restriction $(\hat{f}\upharpoonright H) \colon H\to G'$ is a continuous group homomorphism on some essential retract $H$ of $G/N$. We say that a
 group with a Polish cover  $G/N$ is \emph{Ulam stable} or \emph{rigid}, if every definable automorphism $\varphi\colon G/N\to G/N $ is trivial.
\end{definition}

Notice that when $G$ is connected, then $H=G$ for any  clopen non-empty $H\subseteq G$. In consequence, in the aforementioned Ulam stability results from \cite{kanovei_baire_2000} a lift of $f\colon\mathbb{R}/N\to \mathbb{R}/N'$ is trivial precisely when it is of the form $x\mapsto c\cdot x$. Indeed, the framework that we develop here subsumes other notions of rigidity previously considered in the literature \cite{kanovei_baire_2000,
kanovei_ulams_2000,kanovei_ulam_2001}.
The following theorem is the main result of this section.  It can be viewed as an Ulam stability result in which the \emph{approximate morphisms} are Borel lifts $G\to G'$ of some definable homomorphism $G/N\to G'/N'$ and the \emph{morphisms} are Borel lifts which are continuous homomorphisms after passing to an essential retract.

\begin{theorem}
\label{Theorem:rigid} Let $f: G/N \rightarrow G^{\prime }/N ^{\prime }$ be a definable homomorphism between
 groups with a Polish cover where $G$ is a
non-Archimedean abelian Polish group and $N$ is dense in $G$. If $N ^{\prime }$ is countable, then  $f$ is
trivial.
\end{theorem}

The following is an immediate consequence of Theorem \ref{Theorem:rigid}.

\begin{corollary}
\label{Corollary:rigid} If $N$ is a countable dense subgroup of a
non-Archimedean abelian Polish group  $G$ then the group with a Polish cover $G/N$ is Ulam stable, namely every definable automorphism of $G/N$ has a lift that is a continuous group homomorphism on some open subgroup of $G$.
\end{corollary}

The arguments used in the proof  of Theorem \ref{Theorem:rigid} can be adjusted to prove weaker rigidity results under the weaker assumption that $N ^{\prime }$ is  locally profinite in its Polish topology, or that $N ^{\prime }$ is non-Archimedean in its Polish topology. In both cases, $N'$ has a basis of neighborhoods of the identity consisting of open subgroups. Let  $f\colon G/N\to G'/N'$ 
be a definable homomorphism between groups with a Polish cover. We say that  $f$ is \emph{approximately trivial} if it admits a Borel lift $\hat{f}\colon G\to G'$ with the property that for every open subgroup $V$  of $N'$ (in its Polish topology) there is an essential retract $H$ of $G/N$ so that $(\hat{f}\upharpoonright H)$ is continuous and $\hat{f}(x+y)-\hat{f}(x)-\hat{f}(y)\in V$, for all $x,y\in H$. We   say that  $f$ is  \emph{approximately generically trivial} if for every $V$ as above, the same statement holds, but only for a comeager collection of pairs $(x,y)\in H\times H$.

\begin{theorem}\label{Theorem:rigid2}
Let $f: G/N \rightarrow G^{\prime }/N ^{\prime }$ be a definable homomorphism between
 groups with a Polish cover where $G$ is a
non-Archimedean abelian Polish group and $N$ is dense in $G$. 
\begin{enumerate}
\item If $N'$ is locally profinite in its Polish topology, then  $f$ is approximately trivial.
\item If $N'$ is non-Archimedean in its Polish topology, then  $f$ is approximately generically trivial.
\end{enumerate}
\end{theorem}

The proofs of Theorem \ref{Theorem:rigid} and Theorem \ref{Theorem:rigid2} will span the rest of this section.  We start by showing that whenever
$G$  is a non-Archimedean abelian Polish group and $N$ is a Polishable subgroup then any definable homomorphism from $G/N$ to any group with a Polish cover has a continuous lift.

Recall that every non-Archimedean abelian group $G$ admits a compatible complete invariant metric $d$ with $d\leq 1$.  For every $x\in G$ we denote by $B(x,\varepsilon) $ the corresponding ball of center $x$ and radius $\varepsilon$. A map 
$\gamma\colon G\to G$ is  $1$-Lipschitz if $d(\gamma(x),\gamma(y))\leq d(x,y)$ for all $x,y\in G$.


\begin{proposition}
\label{Proposition:continuous-lift} Let $f\colon G/N\rightarrow G^{\prime
}/N^{\prime }$ be a definable homomorphism between two groups with a Polish
cover. If $G$ is a non-Archimedean abelian Polish group then $f$ lifts to a
continuous function $\varphi \colon G\rightarrow G^{\prime }$.
\end{proposition}

\begin{proof}
Let $\psi :G\rightarrow G^{\prime }$ be a Borel function that is a lift of $f
$. Since $\psi $ is Borel, there exists a decreasing sequence $\left(
U_{i}\right) _{i\in \omega }$ of dense open subsets of $G$ such that the
restriction of $\psi $ to $U:=\bigcap_{i\in \omega }U_{i}$ is continuous 
\cite[Section 2.3]{gao_invariant_2009}. Fix a compatible invariant
ultrametric $d$ on $G$ with $d\leq 1$. We will define $1$-Lipschitz
functions $\gamma _{n}:G\rightarrow G$ and functions $\rho _{n}:G\rightarrow
(0,1]$ so that:

\begin{enumerate}
\item $\gamma _{n}(x)\in B(0,\rho _{n}\left( x\right) )$ and $\rho
_{n+1}\left( x\right) \leq \rho _{n}\left( x\right) \leq 2^{-n}$, for all $%
n\in \mathbb{N}$ and all $x\in G$;

\item $B(x+\gamma _{0}(x)+\cdots +\gamma _{n}(x),\rho _{n}(x))\subseteq U_{n}
$, for all $n\in \mathbb{N}$ and all $x\in G$;

\item $B(-(\gamma _{0}(x)+\cdots +\gamma _{n}(x)),\rho _{n}(x))\subseteq
U_{n}$, for all $n\in \mathbb{N}$ and all $x\in G$.
\end{enumerate}

Assuming now that we have defined the sequence $(\gamma _{n})$ so that it
satisfies the above. By (1), the sequences 
\begin{equation*}
(x+\gamma _{0}(x)+\cdots +\gamma _{n}(x))_{n\in \omega }\quad \text{ and }%
\quad (-(\gamma _{0}(x)+\cdots +\gamma _{n}(x)))_{n\in \omega }
\end{equation*}%
converge for every $x\in G$. We can therefore define the maps $%
f_{0},f_{1}\colon G\rightarrow G$ by setting: 
\begin{equation*}
f_{0}(x):=\mathrm{lim}_{n\rightarrow \infty }\,(x+\gamma _{0}(x)+\cdots
+\gamma _{n}(x))\quad \text{ and }\quad f_{1}(x):=\mathrm{lim}_{n\rightarrow
\infty }\,(-(\gamma _{0}(x)+\cdots +\gamma _{n}(x)))\text{.}
\end{equation*}%
Notice that the space of all $1$-Lipschitz maps is closed under taking
pointwise limits. Since $d$ is an ultrametric, it is also closed under
finite sums. It follows that the functions $f_{0}$ and $f_{1}$ are $1$%
-Lipschitz and hence continuous. 

Fix $n\in \omega $. For $k\geq n$ we have that%
\begin{equation*}
\gamma _{n+1}\left( x\right) +\cdots +\gamma _{k}\left( x\right) \in B\left(
0,\rho _{n}\left( x\right) \right) 
\end{equation*}%
and hence%
\begin{equation*}
x+\gamma _{0}\left( x\right) +\cdots +\gamma _{k}\left( x\right) \in B\left(
x+\gamma _{0}(x)+\cdots +\gamma _{n}(x),\rho _{n}\left( x\right) \right) 
\end{equation*}%
and%
\begin{equation*}
-(\gamma _{0}(x)+\cdots +\gamma _{k}(x))\in B\left( -\left( \gamma
_{0}(x)+\cdots +\gamma _{n}(x)\right) ,\rho _{n}\left( x\right) \right) 
\end{equation*}%
As this holds for every $k\geq n$, we have%
\begin{equation*}
f_{0}\left( x\right) \in B\left( x+\gamma _{0}(x)+\cdots +\gamma
_{n}(x),\rho _{n}\left( x\right) \right) \subseteq U_{n}
\end{equation*}%
and%
\begin{equation*}
f_{1}\left( x\right) \in B\left( -\left( \gamma _{0}(x)+\cdots +\gamma
_{n}(x)\right) ,\rho _{n}\left( x\right) \right) \subseteq U_{n}\text{.}
\end{equation*}%
As $n$ was arbitrary, we have that%
\begin{equation*}
f_{0}\left( x\right) \in \bigcap_{n\in \omega }U_{n}=U
\end{equation*}%
and%
\begin{equation*}
f_{1}\left( x\right) \in \bigcap_{n\in \omega }U_{n}=U\text{.}
\end{equation*}

Since $\psi |_{U}:U\rightarrow G^{\prime }$ is continuous, by (2), (3) we
then have that the map $\psi \colon G\rightarrow G^{\prime }$ defined 
\begin{equation*}
\varphi (x):=\psi (f_{0}(x))+\psi (f_{1}(x))\text{.}
\end{equation*}%
is continuous. Since $x=f_{0}(x)+f_{1}(x)$ for every $x\in G$ we have that $%
\varphi $ is the desired map. Indeed: 
\begin{eqnarray*}
\varphi (x)+N^{\prime } &=&\psi (f_{0}(x))+\psi (f_{1}(x))+N^{\prime } \\
&=&\psi (f_{0}(x)+f_{1}(x))+N^{\prime } \\
&=&\psi (x)+N^{\prime }\text{.}
\end{eqnarray*}

To conclude the proof, we need to define a sequence $(\gamma_n)$ of $1$%
-Lipschitz maps from $G$ to itself which satisfy properties (1), (2), (3)
above. Set $\gamma_{-1}\colon G\to G$ be the map that is constantly $0$, set 
$\rho_{-1}\colon G\to (0,1]$ be the map that is constantly equal to $1$, and
set $\mathcal{A}_{-1}=\{G\}$. By recursion we define for each $n\in \omega
$ a $1$-Lipschitz map $\gamma_n\colon G\to G$, a function $\rho_n\colon G\to
(0,1]$ and a partition $\mathcal{A}$ of $G$ into $d$-balls, so that for all $%
n\in\omega $ we have that

\begin{itemize}
\item $\mathcal{A}_{n}$ is a refinement of $\mathcal{A}_{n-1}$;

\item $\rho _{n}$ and $\gamma _{n}$ are both constant on elements of $%
\mathcal{A}_{n}$; and

\item for every $x\in G$, $\rho _{n}(x)\leq \min \left\{ 2^{-n},\rho _{n-1}\left( x\right)
\right\} $, $\gamma _{n}(x)\in B(0,\rho _{n}(x))$, and the
entire balls $B(x+\gamma _{0}(x)+\cdots +\gamma _{n}(x),\rho _{n}(x))$ and $%
B(-(\gamma _{0}(x)+\cdots +\gamma _{n}(x)),\rho _{n}(x))$ are subsets
of $U_{n}$.
\end{itemize}

Assume that for some $n\in\omega $ we have defined $\mathcal{A}%
_{0},\ldots ,\mathcal{A}_{n-1}$, $\gamma _{0},\ldots ,\gamma _{n-1}$, and $%
\rho _{0},\ldots ,\rho _{n-1}$ as above. We attain $(\mathcal{A}%
_{n},\gamma_n,\rho_n)$ in the form of $(\mathcal{A}^{\prime },\gamma^{\prime
},\rho^{\prime })$ by invoking the next lemma with $(\mathcal{A}%
,\gamma,\rho):=(\mathcal{A}_{n-1},\gamma_{n-1},\rho_{n-1})$.
\end{proof}

\begin{lemma}
\label{Lemma:arisbobelis} Let $G$ be a non-Archimedean abelian Polish group
and let $d$ be a compatible invariant ultrametric on $G$ with values in 
$\left[ 0,1\right] $. Let $\mathcal{A}$ be a partition of $G$ into $d$-balls
and assume that $\gamma:G\rightarrow G$ and $\rho :G\rightarrow (0,1]$ are
maps such that $\gamma$ is $1$-Lipschitz and both $\gamma$ and $\rho$ are
constant on elements of $\mathcal{A}$. Let $U$ be a dense open subset of $G$
and fix $r>0$. Then there exist:

\begin{itemize}
\item a partition $\mathcal{A}^{\prime }$ of $G$ into balls that refines $%
\mathcal{A}$,

\item a function $\gamma^{\prime }:G\rightarrow G$, and

\item a function $\rho ^{\prime }:G\rightarrow (0,1]$
\end{itemize}

which together satisfy the following:

\begin{enumerate}
\item $\gamma^{\prime }( x) \in B( 0,\rho ( x)) $ for $x\in G$;

\item $\rho ^{\prime }\left( x\right) <\min \left\{ \rho \left( x\right)
,r\right\} $ for $x\in G$;

\item $\gamma^{\prime }$ is $1$-Lipschitz;

\item $\gamma^{\prime }$ and $\rho ^{\prime }$ are constant on elements of $%
\mathcal{A}^{\prime }$;

\item for every $x\in G$, the balls $B(x+\gamma
(x)+\gamma ^{\prime }(x), \rho ^{\prime }(x))$ and $B(-(\gamma (x)+\gamma ^{\prime }(x)),\rho ^{\prime }(x))$ are subsets of $U$.
\end{enumerate}
\end{lemma}

\begin{proof}
Without loss of generality we may assume that $\rho ( x) $ is less than or
equal to the diameter of that ball $A\in\mathcal{A}$ which contains $x$.
Define $\mathcal{C}$ to be the partition of $G$ consisting of balls of the
form $B( x,\rho ( x)) $ for $x\in G$. By our assumptions on $\rho ( x)$, the
partition $\mathcal{C}$ refines $\mathcal{A}$. Notice that if $C$ is the
element $B( x,\rho ( x) ) $ of the partition $\mathcal{C}$ then 
\begin{equation*}
\gamma ( x) +C+B( -\gamma ( x) ,\rho ( x)) =C\text{.}
\end{equation*}%
For each such $C\in \mathcal{C}$, since $\rho $ and $\gamma $ are constant
on elements of $\mathcal{A}$ and hence on elements of $C$, we may choose $%
R_{C}<\rho ( x) $ small enough that $U\cap ( \gamma( x) +C) $ contains a
ball of radius $R_{C}$. Define then $\mathcal{A}^{\prime }$ to be the
partition of $G$ consisting of balls $B( x,R_{C})$ for $x\in C$, and $C\in 
\mathcal{C}$. Since $R_{C}<\rho ( x) $ for any $x\in C$ and $C\in \mathcal{C}
$, the partition $\mathcal{A}^{\prime }$ refines $\mathcal{C}$, and it
therefore refines $\mathcal{A}$ as well.

Fix $B=B( x,R_{C}) \in \mathcal{A}^{\prime }$ for $x\in C$ with $C\in 
\mathcal{C}$, and $A\in \mathcal{A}$ with $B\subseteq A$. Since $U$ is open
dense, we have that 
\begin{equation*}
( U-( x+\gamma ( x) ) ) \cap ( U+\gamma ( x ) )
\end{equation*}
is open dense. Thus there exists a $g_{B}$ such that 
\begin{equation*}
g_{B}\in ( U-( x+\gamma ( x) ) ) \cap ( U+\gamma ( x) ) \cap B( 0,R_{C}) ,
\end{equation*}
and hence 
\begin{equation*}
x+\gamma ( x ) +g_{B}\in U\text{,}
\end{equation*}%
\begin{equation*}
-( \gamma ( x) +g_{B}) \in U\text{,}
\end{equation*}
and 
\begin{equation*}
g_{B}\in B( 0,R_{C}) \subseteq A\text{.}
\end{equation*}
Since $U$ is open, there also exists an $R_{B}<\min \{ \rho ( A) ,r\} $ such
that 
\begin{equation*}
B( x+\gamma ( x) +g_{B},R_{B}) \subseteq U
\end{equation*}
and 
\begin{equation*}
B( -( \gamma ( x) +g_{B}) ,R_{B}) \subseteq U\text{.}
\end{equation*}
Define then%
\begin{equation*}
\gamma^{\prime }( x) :=g_{B}
\end{equation*}
and 
\begin{equation*}
\rho ^{\prime }( x) :=R_{B}
\end{equation*}%
for each $x\in B$ and $B\in \mathcal{A}^{\prime }$. In particular, $%
\gamma^{\prime }$ and $\rho ^{\prime }$ are constant on elements of $%
\mathcal{A}^{\prime }$; the terms in question in fact satisfy (1), (2), (4),
(5). It remains to prove that $\gamma^{\prime }$ is $1$-Lipschitz.

Suppose now that $x$ and $y$ are elements of $G$. Since $d$ is an
ultrametric and $\gamma $ is $1$-Lipschitz, the map $x\mapsto x+\gamma
\left( x\right) $ is $1$-Lipschitz as well. It follows that if $x,y\in G$
and $C,D\in \mathcal{C}$ are such that $x\in C$ and $y\in D$, then%
\begin{eqnarray*}
&&d( \gamma^{\prime }( x) ,\gamma^{\prime }( y) ) \\
&=&d( x+\gamma ( x) +\gamma^{\prime }( x) ,x+\gamma ( x) +\gamma^{\prime }(
y) ) \\
&\leq &\max \{ d( x+\gamma ( x) +\gamma^{\prime }( x) ,y+\gamma ( y) ) ,d(
y+\gamma ( y) ,x+\gamma ( x) +\gamma^{\prime }( y) ) \} \\
&\leq &\max \{ d( \gamma^{\prime }( x) ,0) ,d( x+\gamma ( x) ,y+\gamma ( y)
) ,d( \gamma^{\prime }( y) ,0) ,d( y+\gamma ( y) ,x+\gamma ( x) ) \} \\
&\leq &\max \{ R_{C},R_{D},d( x,y) \} \text{.}
\end{eqnarray*}%
If $d( x,y) <\min \{ R_{C},R_{D}\} $ then $C=D$ and $\gamma^{\prime }( x)
=\gamma^{\prime }( y) $. Hence, by the above inequality  $\gamma^{\prime }$
is $1$-Lipschitz.
\end{proof}

The following is the second ingredient that goes into the proof of Theorem \ref{Theorem:rigid}.

\begin{lemma}
\label{Lemma:non-Arch} 
Let $f: G/N \rightarrow G^{\prime }/N ^{\prime }$ be a  homomorphism between
 groups with a Polish cover, which lifts to a continuous map $G\to G'$. If  $G$  is abelian and non-Archimedean, $N$ is dense in $G$, and  $N ^{\prime }$ is a countable, then $f$ is trivial.
\end{lemma}
\begin{proof}
Let $\varphi\colon G\to G$ be the continuous lift of $f$. 
Then the function $C:G\times G\rightarrow N' $, $(
x,y) \mapsto \varphi ( x+y) -\varphi ( x)
-\varphi ( y) $ is continuous. By the Baire Category Theorem, there exists an $%
m\in N ^{\prime }$ such that $A:=\left\{ ( x,y) \in G\times
G:C( x,y) =m\right\} $ has nonempty interior. As $N $ is
dense in $G$ and $G$ has a basis of neighborhoods of the identity
consisting of clopen subgroups, there exists an $( x_{0},y_{0})
\in N \times N $ and clopen  subgroup  $H$ of $G$ such that  $\left( H\times H\right)
+\left( x_{0},y_{0}\right) \subseteq A$. This implies that $\varphi ( x+y+x_{0}+y_{0})
=\varphi ( x+x_{0}) +\varphi ( y+y_{0}) +m$ for every $
x,y\in H$.

Notice that the function $\varphi ^{\prime }:G\rightarrow G^{\prime }$, $%
x\mapsto \varphi ( x) +m$ satisfies the same assumptions as $%
\varphi ^{\prime }$, and furthermore that $\varphi ^{\prime }( x)
-\varphi ( x) \in N ^{\prime }$ for every $x\in G$. Notice
also that for $x,y\in G$,%
\begin{eqnarray*}
\varphi ^{\prime }( x+y+x_{0}+y_{0}) &=&\varphi (
x+y+x_{0}+y_{0}) +m \\
&=&\varphi ( x+x_{0}) +\varphi ( y+y_{0}) +2m \\
&=&( \varphi ( x+x_{0}) +m) +( \varphi (
y+y_{0}) +m) \\
&=&\varphi ^{\prime }( x+x_{0}) +\varphi ^{\prime }(
y+y_{0}) \text{.}
\end{eqnarray*}%
Therefore, by replacing $\varphi $ with $\varphi ^{\prime }$ if necessary, we are free to assume that for every $x,y\in H$,%
\begin{equation*}
\varphi ( x_{0}+y_{0}+x+y) =\varphi ( x_{0}+x)
+\varphi ( y_{0}+y) \text{.}
\end{equation*}%
Define now%
\begin{equation*}
\psi ( z) :=\varphi ( x_{0}+y_{0}+z) -\varphi (
x_{0}+y_{0}) \text{.}
\end{equation*}%
Then we have that%
\begin{equation*}
\psi ( z) -\varphi ( z) \in N ^{\prime }
\end{equation*}%
for all $z$ in $H$. Furthermore, for $x,y\in H$,%
\begin{eqnarray*}
&&\psi ( x) +\psi ( y) \\
&=&( \varphi ( x_{0}+y_{0}+x) -\varphi (
x_{0}+y_{0}) ) +( \varphi ( x_{0}+y_{0}+y)
-\varphi ( x_{0}+y_{0}) ) \\
&=&\varphi ( x+x_{0}) +\varphi ( y_{0}) -\varphi
( x_{0}+y_{0}) +\varphi ( y+y_{0}) +\varphi (
x_{0}) -\varphi ( x_{0}+y_{0}) \\
&=&\varphi ( x+x_{0}) +\varphi ( y+y_{0}) -\varphi
( x_{0}+y_{0}) \\
&=&\varphi ( x+y+x_{0}+y_{0}) -\varphi ( x_{0}+y_{0})
\\
&=&\psi ( x+y) \text{.}
\end{eqnarray*}%
This concludes the proof.
\end{proof}

We may now conclude the proof of Theorem \ref{Theorem:rigid}.

\begin{proof}[Proof of Theorem \ref{Theorem:rigid}]
By Proposition \ref{Proposition:continuous-lift} $f$ admits a continuous lift. The rest follows from  Lemma \ref{Lemma:non-Arch}.
\end{proof}

For the proof of Theorem \ref{Theorem:rigid2}, we need the following variant of  Lemma \ref{Lemma:non-Arch}. Recall that the notation $\forall^* x \in X$ stands for ``there is a  comeager $C\subseteq X$ so that $\forall x \in C$".

\begin{lemma}
\label{Lemma:non-Arch2} 
Let $\varphi\colon G\to G'$ be a continuous map that is a lift of 
a homomorphism $f: G/N \rightarrow G^{\prime }/N ^{\prime }$ 
between groups with a Polish cover. If $G$  is  abelian and non-Archimedean, $N$ is dense in $G$, and  $N ^{\prime }$ is non-Archimedean in its Polish topology then for every clopen subgroup $V$ of $N ^{\prime }$  there exist:
\begin{itemize}
\item a clopen subgroup $H$ of $G$;
\item an element $m$ of $N'$;
\item $x_{0},y_{0}\in G$;
\end{itemize}
so that  if $\psi :G\rightarrow G^{\prime }$ is given by
$\psi \left( z\right) :=\varphi \left( x_{0}+y_{0}+z\right) -\varphi \left(
x_{0}+y_{0}\right) -m$,
then  $\forall^* x \in H \; \forall^* y\in H$, we have:
\begin{equation*}
\psi \left( x+y\right) -\psi \left( x\right) -\psi \left( y\right) \in V%
\text{.}
\end{equation*}
Moreover, if $N'$ is additionally locally profinite then the last identity holds for all $x$ and $y$ in $H$. 
\end{lemma}
\begin{proof}
The proof is essentially identical to the proof of Lemma \ref{Lemma:non-Arch}. We define $C\colon G\to N'/V$ by  $C(x,y):=(\varphi(x+y)-\varphi(x)-\varphi(y))+V$. Since $N'/V$ is countable, the Baire Category Theorem implies that there is $m\in N'$ so that $A=\{(x,y)\in G\times G \colon C(x,y)=m+V\}$ is non-meager. Since $A$ is Borel, we may find $x_0,y_0\in G$ and a  clopen subgroup $H$ of $G$ so that $\forall^* x \in H \;  \forall^* y\in H$ $(x+x_0,y+y_0)\in A$. The rest of the argument is the same. 

If $N'$ was additionally locally profinite then we can always arrange so that $V$ is compact in the Polish topology of  $N'$. But this implies that $V$ is a closed subgroup of $G'$. By continuity of $\varphi$ this implies that $A$ is closed. Hence, we may find $x_0,y_0\in G$ and a  clopen subgroup $H$ of $G$ so that $\forall x \in H \;  \forall y\in H$ $(x+x_0,y+y_0)\in A$.  
\end{proof}

\begin{proof}[Proof of Theorem \ref{Theorem:rigid2}]
We record the argument for (1). The proof of (2) is similar and is left to the reader.

By Proposition \ref{Proposition:continuous-lift}, we may fix some continuous lift  $\varphi\colon G\to G'$ of $f$.  Let $(V_k)_{k\in\omega}$ be a decreasing
sequence  of \emph{compact open} subgroups of $N ^{\prime }$ that forms a basis of
neighborhoods of the identity, and let  $\left(
W_{k}\right) _{k\in \omega }$ be  a decreasing sequence of open subgroups of $G$ that forms a basis of
neighborhoods of the identity. 

Set $H_{-1}=G$, $\varphi _{-1}=\varphi $, $%
V_{-1}=N ^{\prime }$. 
Applying Lemma \ref{Lemma:non-Arch2}, we define by recursion on $k\in \omega$:

\begin{itemize}
\item an open subgroup $H_{k}$ of $G$ such that $H_{k}\subseteq W_{k}$;
\item elements $x_{k},y_{k}\in H_{k-1}$,
\item elements $m_{k}\in V_{k-1}$,
\item a continuous function $\varphi _{k}:G\rightarrow G^{\prime }$,
\end{itemize}

such that, for every $k\in \omega $ one has that:

\begin{enumerate}
\item for every $x,y\in H_{k-1}$, 
\begin{equation*}
\varphi _{k-1}\left( x+y+x_{k}+y_{k}\right) +V_{k}=\varphi _{k-1}\left(
x+x_{k}\right) +\varphi \left( y+y_{k}\right) +m_{k}+V_{k}
\end{equation*}

\item for every $z\in G$,%
\begin{equation*}
\varphi _{k}\left( z\right) =\varphi _{k-1}\left( x_{k}+y_{k}+z\right)
-\varphi _{k-1}\left( x_{k}+y_{k}\right) -m_{k}\text{, and}
\end{equation*}%

\item for every $x,y\in H_{k}$,%
\begin{equation*}
\varphi _{k}\left( x+y\right) -\varphi _{k}\left( x\right) -\varphi
_{k}\left( y\right) \in V_{k}\text{.}
\end{equation*}
\end{enumerate}

Define now, for every $k\in \omega $,%
\begin{equation*}
z_{k}:=\left( x_{0}+y_{0}\right) +\cdots +\left( x_{k}+y_{k}\right) \text{.}
\end{equation*}%
We prove by induction on $k\in \omega $ that, for every $z\in G$,%
\begin{equation*}
\varphi _{k}\left( z\right) =\varphi \left( z_{k}+z\right) -\varphi \left(
z_{k}\right) -m_{k}\text{.}
\end{equation*}%
For $k=0$ this holds by definition of $\varphi _{0}$. Indeed, (2) we have
that%
\begin{eqnarray*}
\varphi _{0}\left( z\right) &=&\varphi \left( x_{0}+y_{0}+z\right) -\varphi
\left( x_{0}+y_{0}\right) -m_{0} \\
&=&\varphi \left( z_{0}+z\right) -\varphi \left( z_{0}\right) -m_{0}\text{.}
\end{eqnarray*}%
Suppose that the conclusion holds for $k-1$. Then we have that, by the
induction hypothesis,%
\begin{equation*}
\varphi _{k-1}\left( x_{k}+y_{k}+z\right) =\varphi \left(
z_{k-1}+x_{k}+y_{k}+z\right) -\varphi \left( z_{k-1}\right) -m_{k-1}
\end{equation*}%
and%
\begin{equation*}
\varphi _{k-1}\left( x_{k}+y_{k}\right) =\varphi \left(
z_{k-1}+x_{k}+y_{k}\right) -\varphi \left( z_{k-1}\right) -m_{k-1}
\end{equation*}%
Therefore, by definition%
\begin{eqnarray*}
\varphi _{k}\left( z\right) &=&\varphi _{k-1}\left( x_{k}+y_{k}+z\right)
-\varphi _{k-1}\left( x_{k}+y_{k}\right) -m_{k} \\
&=&\varphi \left( z_{k-1}+x_{k}+y_{k}+z\right) -\varphi \left(
z_{k-1}+x_{k}+y_{k}\right) -m_{k} \\
&=&\varphi \left( z_{k}+z\right) -\varphi \left( z_{k}\right) -m_{k}\text{.}
\end{eqnarray*}%
Since $H_{k}\subseteq W_{k}$ for every $k\in \omega $ and $\left(
W_{k}\right) $ is a basis of neighborhoods of the identity, we have that
the sequence $\left( x_{k}+y_{k}\right) _{k\in \omega }$ converges to $%
0 $, and hence the sequence $\left( z_{k}\right) _{k\in \omega }$ converges
to some element $z_{\infty }$ of $G$.

Now, we claim that, for every $k\in \omega $, $x,y\in H_{k}$, and $i\geq k$,%
\begin{equation*}
\varphi _{i}\left( x+y\right) -\varphi _{i}\left( x\right) -\varphi
_{i}\left( y\right) \in V_{k}\text{.}
\end{equation*}%
We prove this by induction on $i\geq k$. For $k=i$ this holds for (3).
Suppose it holds for $i-1\geq k$. Then we have that, for $x,y\in H_{k}$,%
\begin{equation*}
\varphi _{i}\left( x+y\right) =\varphi _{i-1}\left( x_{i}+y_{i}+x+y\right)
-\varphi _{i-1}\left( x_{i}+y_{i}\right) -m_{i}\text{,}
\end{equation*}%
\begin{equation*}
\varphi _{i}\left( x\right) =\varphi _{i-1}\left( x_{i}+y_{i}+x\right)
-\varphi _{i-1}\left( x_{k}+y_{k}\right) -m_{i}\text{,}
\end{equation*}%
\begin{equation*}
\varphi _{i}\left( y\right) =\varphi _{i-1}\left( x_{i}+y_{i}+y\right)
-\varphi _{i-1}\left( x_{k}+y_{k}\right) -m_{i}\text{,}
\end{equation*}%
Considering that $x_{i},y_{i}\in H_{i-1}\subseteq H_{k}$ and $m_{i}\in
V_{i-1}\subseteq V_{k}$, we have that by the inductive hypothesis%
\begin{eqnarray*}
\varphi _{i}\left( x+y\right) +V_{k} &=&\varphi _{i-1}\left(
x_{i}+y_{i}+x+y\right) -\varphi _{i-1}\left( x_{i}+y_{i}\right) -m_{i}+V_{k}
\\
&=&\varphi _{i-1}\left( x_{i}+y_{i}+x\right) -\varphi _{i-1}\left(
x_{k}+y_{k}\right) -m_{i}+V_{k} \\
&&+\varphi _{i-1}\left( x_{i}+y_{i}+y\right) -\varphi _{i-1}\left(
x_{k}+y_{k}\right) -m_{i}+V_{k} \\
&=&\varphi _{i}\left( x\right) +\varphi _{i}\left( y\right) +V_{k}\text{.}
\end{eqnarray*}%
This concludes the proof. Therefore, setting%
\begin{equation*}
\varphi _{\infty }\left( z\right) =\varphi \left( z+z_{\infty }\right)
-\varphi \left( z_{\infty }\right) =\mathrm{lim}_{i\rightarrow \infty
}\varphi \left( z+z_{i}\right) -\varphi \left( z_{i}\right) =\mathrm{lim}%
_{i\rightarrow \infty }\varphi _{i}\left( z\right) \text{,}
\end{equation*}%
we have that, for every $k\in \omega $ and $x,y\in H_{k}$, 
\begin{equation*}
\varphi _{\infty }\left( x+y\right) +V_{k}=\varphi _{\infty }\left( x\right)
+\varphi _{\infty }\left( y\right) +V_{k}\text{.}
\end{equation*}%
This concludes the proof.
\end{proof}

We close this section with some examples which demonstrate that in 
Theorem \ref{Theorem:rigid} and in 
Theorem \ref{Theorem:rigid2}(1) the conclusions cannot be strengthened any further. It is at present unclear if the same is true for  Theorem \ref{Theorem:rigid2}(2). 

\begin{example}
\label{Example:trivial} The group $\mathbb{Z}_{2}$ of all \emph{dyadic integers} is  attained as the inverse limit of the inverse system 
\[(\mathbb{Z}/2\mathbb{Z},+)\longleftarrow(\mathbb{Z}/4\mathbb{Z},+) \longleftarrow(\mathbb{Z}/8\mathbb{Z},+) \longleftarrow(\mathbb{Z}/16\mathbb{Z},+)\longleftarrow\cdots\]
where the bonding maps $\mathbb{Z}/2^k\mathbb{Z}\to\mathbb{Z}/2^{k-1}\mathbb{Z}$ are given by $a\mapsto (a\mod 2^{k-1})$. This is clearly a profinite abelian group and hence Polish in the topology that it inherits as a subgroup of $\prod_k \mathbb{Z}/2^k$, the latter being endowed  with the product topology.  Notice that the map  $k\mapsto \big( (k\mod 2), (k\mod 4), \ldots \big)$ realizes $\mathbb{Z}$ as a Polishable subgroup of $\mathbb{Z}_{2}$.
In the group $\mathbb{Z}_{2}/\mathbb{Z}$ with a Polish cover
every definable homomorphism $\mathbb{Z}_{2}\rightarrow \mathbb{Z}_{2}$ that maps $\mathbb{Z}$ into $\mathbb{Z}$ has the form $y\mapsto cy$ for
some $c\in \mathbb{Z}$. Hence the Borel homomorphism $\varphi :\mathbb{Z}_{2}/\mathbb{Z}\rightarrow \mathbb{Z}_{2}/\mathbb{Z}$, $x+\mathbb{Z}\mapsto
f_{1}( x) +\mathbb{Z}$, where $f_{1}( x) \in \mathbb{Z}_{2}$ is such that $2f_{1}( x) -x\in \mathbb{Z}$, does \emph{not}
have a lift to a continuous homomorphism on $\mathbb{Z}_{2}$. Nevertheless, $\varphi $ is trivial, as it is induced by the continuous homomorphism $2\mathbb{Z}_{2}\rightarrow \mathbb{Z}_{2}$, $y\mapsto \frac{1}{2}y$ defined
on the clopen subgroup $2\mathbb{Z}_{2}$ of $\mathbb{Z}_{2}$. This shows that one cannot strengthen the conclusions of Theorem \ref{Theorem:rigid} and obtain a lift that is a continuous group homomorphism $G\rightarrow G'$.
\end{example}

\begin{example}
\label{Example:trivial2} Consider the groups with a Polish cover $\mathbb{Z}_{2}/\mathbb{Z}$ and $\mathbb{Z}_{2}^{\omega }/\mathbb{Z}^{\omega }$, where $\mathbb{Z}^{\omega }$ and $\mathbb{Z}_{2}^{\omega }$ are each endowed with the product topology. Consider the Borel homomorphism $\varphi :\mathbb{Z}_{2}/%
\mathbb{Z}\rightarrow \mathbb{Z}_{2}^{\omega }/\mathbb{Z}^{\omega }$ defined
by $$x+\mathbb{Z}\mapsto (f_{0}(x),f_{1}( x) ,f_{2}( x)
,\ldots )+\mathbb{Z}^{\omega },$$ where $f_{n}( x) \in \mathbb{Z}_{2}$ is such that $2^{n}f_{n}( x) -x\in \mathbb{Z}$ for $n\in
\omega $. Any clopen subgroup of $\mathbb{Z}_{2}$ is of the form $2^{k}\mathbb{Z}$ for some $k\in \omega $. A continuous homomorphism from $2^{k}\mathbb{Z}$
to $\mathbb{Z}_{2}^{\omega }$ that maps $2^{k}\mathbb{Z}$ to $\mathbb{Z}^{\omega }$ is of the form $y\mapsto ( 2^{-k}c_{0}y,2^{-k}c_{1}y,\ldots
) $ for some $c_{0},c_{1},\ldots \in \mathbb{Z}$. Thus $\varphi $ is
not induced by a continuous homomorphism defined on a clopen subgroup of $\mathbb{Z}_{2}$. This shows that in Theorem \ref{Theorem:rigid2}(1) it is not sufficient to assume that $N'$ is non-Archimedean.
\end{example}

\section{The definable content of the $\mathrm{lim}^1$-functor\label{Section:towers}}

Many computations and constructions in algebraic topology and  homological algebra naturally give rise to groups  with a Polish cover $G/N$ which satisfy the assumptions of Theorem \ref{Theorem:rigid}.  Examples include the computations of \v{C}ech cohomology groups of mapping telescopes, the  computations of Steenrod homology groups of solenoids, and the constructions of the invariant $\mathrm{Ext}(\Lambda,\mathbb{Z})$ of all extensions of finite rank torsion-free abelian groups $\Lambda$ by $\mathbb{Z}$. 
We treat the first two of these examples in later installments of this series (see \cite{BLPII} for the first); the study of the ``definable content"   of  $\mathrm{Ext}(\Lambda,\mathbb{Z})$ will occupy Section \ref{S:Ext} below.  In all these examples, the computation of the pertinent group $G/\Lambda$ involves an application of the first derived functor $\mathrm{lim}^1(-)$ of the  $\mathrm{lim}$-functor on towers of countable abelian groups.

 In this section we develop the theory of the ``definable content" of the $\mathrm{lim}^1$-functor. In particular, we show that 
 when applied to the category of \emph{towers of abelian Polish groups} it takes values in the category  of groups with a Polish cover. We  use  Theorem \ref{Theorem:rigid} to deduce   that the $\mathrm{lim}^1$-functor is fully faithful when it is further restricted to the full subcategory of all  \emph{filtrations}, and we provide an explicit description of the objects in its image in terms of pro-countable completions of countable abelian groups; see Theorem \ref{Theorem:lim1-injective} and Corollary \ref{Corollary:lim1-injective}. It follows that for any two filtrations $\boldsymbol{A},\boldsymbol{B}$,  $\mathrm{lim}^1\boldsymbol{A}$ and  $\mathrm{lim}^1\boldsymbol{B}$ are definably isomorphic if and only if $\boldsymbol{A}$ and $\boldsymbol{B}$ are isomorphic as objects in the category of towers of countable abelian groups; see Corollary \ref{Corollary:ProfiniteDefinableIsomoTowers}.

\subsection{Towers of Polish groups} A \emph{tower of Polish groups} is an inverse sequence $\boldsymbol{G}%
=\left( G^{\left( m\right) },p^{\left( m,m+1\right) }\right) $ of Polish
groups and continuous homomorphisms $p^{\left( m,m+1\right) }:G^{\left(
m+1\right) }\rightarrow G^{\left( m\right) }$, where  $m$ ranges over $\omega$. These morphisms determine the more general family of morphisms $p^{( m,m^{\prime })
}=p^{( m,m+1) }\circ \cdots \circ p^{( m^{\prime
}-1,m^{\prime }) }:G^{(
m+1) }\rightarrow G^{( m) }$ for $m<m^{\prime }<\omega$, along with $p^{\left(
m,m\right) }=1_{G^{\left( m\right) }}$, for each $m\in\omega$. 
 Let  $\boldsymbol{G}=\left( G^{\left( m\right) },p^{\left( m,m+1\right)}\right) $ and $\boldsymbol{H}=\left( H^{\left( m\right) },p^{\left(
m,m+1\right) }\right) $ be towers of Polish groups.  By an  \emph{inv-map   from $\boldsymbol{G}$ to $\boldsymbol{H}$} we mean a sequence $\left(
m_{k},f^{\left( k\right) }\right) _{k\in \omega }$ in which:

\begin{itemize}
\item $\left( m_{k}\right) $ is an increasing sequence in $\omega $, and

\item for all $k\in \omega $, $f^{\left( k\right) }\colon  G^{\left( m_{k}\right)}\to H^{\left( k\right)} $ is a continuous
homomorphism so that 
 $p^{\left( k,k+1\right) }f^{\left( k+1\right) }=f^{\left( k\right)
}p^{(m_{k},m_{k+1})}$.
\end{itemize}
It follows that $p^{\left( k_{0},k_{1}\right) }f^{\left(k_{1}\right) }=f^{\left( k_{0}\right) }p^{(m_{k_{0}},m_{k_{1}})}$ for all $
k_{0}<k_{1}$. The collections of inv-maps forms a category with towers of Polish groups as objects.  The \emph{identity} inv-map $\mathrm{id}_{\boldsymbol{G}}$  of $\boldsymbol{G}$ is  $\left( m_{k},f^{\left( k\right) }\right)$, where $m_{k}=k$ and $
f^{\left( k\right) }=1_{G^{\left( k\right) }}$. The \emph{composition} $
\left( k_{t},g^{\left( t\right) }\right)\!\circ \!\left( m_{k},f^{\left(
k\right) }\right) $ of the inv-map   $\left( m_{k},f^{\left( k\right) }\right) $
from $\boldsymbol{G}$ to $\boldsymbol{H}$ and the inv-map  $\left(
k_{t},g^{\left( t\right) }\right) $ from $\boldsymbol{H}$ to $\boldsymbol{L}$
is the inv-map  $\left( m_{k_{t}},g^{\left( t\right) }f^{\left( k_{t}\right)
}\right) $ from $\boldsymbol{G}$ to $\boldsymbol{L}$. So defined, the collection of all inv-maps forms a category. However, many natural functors on this category view various inv-maps as equivalent and therefore fail to be faithful. To remedy this, we define a congruence on inv-maps and pass to the related homotopy category. 

Let  $\left(m_{k},f^{\left( k\right) }\right)$ and $\left( m_{k}^{\prime },f^{\prime
\left( k\right) }\right) $ be inv-maps from $\boldsymbol{G}$ to $
\boldsymbol{H}$. We say that $\left( m_{k},f^{\left( k\right) }\right) $
and $\left( m_{k}^{\prime },f^{\prime \left( k\right) }\right) $ are \emph{congruent} if there exists an increasing sequence $\left( \tilde{m}
_{k}\right)$ such that $\tilde{m}_{k}\geq \max \left\{ m_{k},m_{k}^{\prime
}\right\}$ for $k\in \omega $ and $f^{\left( k\right) }p^{(m_{k},\tilde{m}
_{k})}=f^{\prime \left( k\right) }p^{\left( m_{k}^{\prime },\tilde{m}%
_{k}\right) }$ for every $k\in \omega$; \cite[Section 1.1]%
{mardesic_strong_2000}. It is easy to check that this defines an equivalence relation among inv-maps
from $\boldsymbol{G}$ to $\boldsymbol{H}$. A \emph{pro-map  from $\boldsymbol{G}$ to $\boldsymbol{H}$}, or simply a \emph{map from $\boldsymbol{G}$ to $\boldsymbol{H}$},  is the congruence class $\left[m_{k},f^{\left( k\right) }\right]$ of some 
inv-map $\left( m_{k},f^{\left( k\right) }\right) $ from $\boldsymbol{G}$ to $
\boldsymbol{H}$. The \emph{identity map} of $\boldsymbol{G}$ is the congruence class of the
identity inv-map of $\boldsymbol{G}$. The \emph{composition} of inv-maps $\left( m_{k},f^{\left( k\right) }\right) $ from $%
\boldsymbol{G}$ to $\boldsymbol{H}$ and $\left(k_{t},g^{\left( t\right)
}\right) $ from $\boldsymbol{H}$ to $\boldsymbol{L}$ is given by setting:
\begin{equation*}
\left[ k_{t},g^{\left( t\right) }\right] \circ \left[ m_{k},f^{\left(
k\right) }\right] =\left[ \left( k_{t},g^{\left( t\right) }\right) \circ
\left( m_{k},f^{\left( k\right) }\right) \right]
\end{equation*}
 It is easy to verify
that this indeed defines a category that has towers of Polish groups as objects. 

\begin{definition}\label{def:CategoryOfTowers}
The maps between towers of Polish groups comprise the morphisms in the \emph{category of towers of Polish groups}. Prominent full subcategories of this category are those of \emph{towers of abelian Polish groups} and of \emph{towers of countable abelian Polish groups}.
\end{definition}
We have some obvious  closure properties for these categories. For example, the \emph{trivial tower} $\boldsymbol{G}$, in which each $G^{\left( k\right) }$ is
the trivial group, is an initial and terminal object in the
category of towers of Polish groups. If $\boldsymbol{G}$ and $\boldsymbol{H}$ are towers, then let $\boldsymbol{G}\times \boldsymbol{H}$ denote the tower $\left( G^{\left( k\right) }\times
H^{\left( k\right) }\right) _{k\in \omega }$ with bonding maps the products of the bonding maps of $\boldsymbol{G}$ and $\boldsymbol{H}$. $\boldsymbol{G}\times \boldsymbol{H}$ is the
product and coproduct of $\boldsymbol{G}$ and $\boldsymbol{H}$ in the
category of towers of Polish groups. Finally,   the  (additively denoted) category of towers of abelian Polish groups forms an additive category. To see this,
if $[m_{k},f^{\left( k\right) }]$ and $[m_{k}^{\prime },f^{\prime \left(
k\right) }]$ are  maps from $\boldsymbol{G}$ to $\boldsymbol{H}$ then we may assume without loss of generality that $m_{k}=m_{k}^{\prime }$ for
every $k\in \omega $. Hence,   we may define $[m_{k},f^{\left( k\right)
}]+[m_{k}^{\prime },f^{\prime \left( k\right) }]$ to be the map $%
[m_{k},f^{\left( k\right) }+f^{\prime \left( k\right) }]$.

\subsection{The $\mathrm{lim}^{1}$-functor}  We now restrict our discussion to the category of towers of abelian Polish groups. To each tower  $\boldsymbol{A}=\left(
A^{\left( m\right) },p^{\left( m,m+1\right) }\right) $ of  abelian Polish groups one associates the \emph{inverse limit} $\underleftarrow{\mathrm{lim}} \,\boldsymbol{A}$ of $\boldsymbol{A}$ given by
\begin{equation}\label{eq:lim}
\underleftarrow{\mathrm{lim}}\,\boldsymbol{A}:=\{\left( x_{m} \right)\in \prod_{m\in \omega } A^{(m)}:\forall m,m^{\prime }\in
\omega ,x_{m}=p^{( m,m^{\prime }) }( x_{m^{\prime }})
\}\text{.}
\end{equation}
Clearly, $\underleftarrow{\mathrm{lim}}\,\boldsymbol{A}$ is a closed, hence Polish, abelian subgroup of the Polish product group. Moreover,  given any inv-map
$(m_{k},f^{\left( k\right) }):\boldsymbol{A}\rightarrow 
\boldsymbol{B}$, we have a 
continuous homomorphism $\underleftarrow{\mathrm{lim}}\,(m_{k},f^{\left(
k\right) }):\underleftarrow{\mathrm{lim}}\boldsymbol{A}\rightarrow 
\underleftarrow{\mathrm{lim}}\boldsymbol{B}$, defined by setting%
\begin{equation*}
\left( \underleftarrow{\mathrm{lim}}(m_{k},f^{\left( k\right) })\right)
\left( \left( x_{m}\right) _{m\in \omega }\right) =\left( f^{\left( k\right)
}\left( x_{m_{k}}\right) \right) _{k\in \omega }\text{.}
\end{equation*}
This induces an assignment $[m_{k},f^{\left( k\right)
}]\mapsto \underleftarrow{\mathrm{lim}}\,[m_{k},f^{\left( k\right) }]$, since any two congruent inv-maps induce the same homomorphism $\underleftarrow{\mathrm{lim}}\boldsymbol{A}\rightarrow 
\underleftarrow{\mathrm{lim}}\boldsymbol{B}$. 
It is straightforward to verify that the assignments $\boldsymbol{A}\mapsto 
\underleftarrow{\mathrm{lim}}\boldsymbol{A}$ and $[m_{k},f^{\left( k\right)
}]\mapsto \underleftarrow{\mathrm{lim}}\,[m_{k},f^{\left( k\right) }]$ define
an additive functor from the category of towers of abelian Polish groups to
the category of abelian Polish groups. This functor however fails to be \emph{right exact}. That is, the image of a short exact sequence 
\[0\rightarrow \boldsymbol{A}\rightarrow \boldsymbol{B} \rightarrow \boldsymbol{C}\rightarrow 0\]
under the $\underleftarrow{\mathrm{lim}}$-functor will generally fail to give an epimorphism $\underleftarrow{\mathrm{lim}}\boldsymbol{B}\rightarrow 
\underleftarrow{\mathrm{lim}}\boldsymbol{C}$; see \cite[Example 11.23]{mardesic_strong_2000}. The failure of right exactness is measured by the sequence $(\underleftarrow{\mathrm{lim}}^n)_{n\in\omega}$ of the right derived functors of $\underleftarrow{\mathrm{lim}}$; see \cite%
{jensen_foncteurs_1972,weibel_introduction_1994,mardesic_strong_2000}. For every tower $\boldsymbol{A}$ as above, the group $\underleftarrow{\mathrm{lim}}^n\boldsymbol{A}$ is computed as the $n$-th cohomology group of a cochain complex
\begin{equation*}
C^{\bullet }( \boldsymbol{A}) :=0\rightarrow C^{0}( 
\boldsymbol{A}) \overset{\delta^1 }{\rightarrow} C^{1}( \boldsymbol{%
A}) \overset{\delta^2 }{\rightarrow} C^{2}( \boldsymbol{%
A}) \overset{\delta^3 }{\rightarrow} C^{3}( \boldsymbol{%
A}) \rightarrow\cdots
\end{equation*}
which associated to $\boldsymbol{A}$ as in \cite[Section 11.5]{mardesic_strong_2000}. It turns out that within the context of towers, these functors vanish for $n\geq 2$; see \cite[Section 11.6]{mardesic_strong_2000}.  We proceed to the definition of  $\underleftarrow{\mathrm{lim}}^{1}\boldsymbol{A}$.

 Let  $\boldsymbol{A}=\left(
A^{\left( m\right) },p^{\left( m,m+1\right) }\right)$ be a tower of \emph{abelian} Polish groups. 
We follow \cite[Section 11.5]{mardesic_strong_2000} with the only exception that our coboundary operators $\delta^0,\delta^1$ inherit the orientation from $(m_0\leq m_1\leq m_2)^{op}:= m_2\leq^{*} m_1 \leq^{*} m_0$ rather than from $m_0\leq m_1\leq m_2$.  This choice of orientation is more natural for some identifications later on.
Of course, the resulting groups $\mathrm{Z}(\boldsymbol{A}), \mathrm{B}(\boldsymbol{A})$, and $\underleftarrow{\mathrm{lim}}^{1}\boldsymbol{A}$ stay the same.

\begin{itemize}
\item $C^{0}( \boldsymbol{A}), C^{1}( \boldsymbol{A})$, and $C^{2}( \boldsymbol{A})$, are the following abelian Polish group  endowed with the product topology:
\[C^{0}( \boldsymbol{A}):=\prod_{m\in \omega
}A^{\left( m\right)}, \quad \quad  C^{1}( \boldsymbol{A}):=\prod_{\substack{ ( m_0,m_1) \in \omega^2  \\ %
m_0\leq m_1}}A^{\left( m_0\right)}, \quad \quad C^{2}( \boldsymbol{A}):=\prod_{\substack{ ( m_0,m_1,m_2) \in \omega^3  \\
m_0\leq m_1\leq m_2 }}A^{\left( m_0\right)}.
\]
\item $\delta^1\colon C^{0}( \boldsymbol{A}) \to C^{1}( \boldsymbol{A})$ and  $\delta^2\colon C^{1}( \boldsymbol{A}) \to C^{2}( \boldsymbol{A})$  are the continuous group homomorphisms given by 
\begin{equation*}
\big(\delta^1 (x) \big)_{m_0,m_1}  = x_{m_0} - p^{( m_0,m_1)
}(x_{m_1}),   \text{ for all  } x=(x_m) \in C^{0}( \boldsymbol{A}),   \text{ and}    
\end{equation*}
\begin{equation*}
\big(\delta^2(x)\big)_{m_0,m_1,m_2}=x_{m_0, m_1}-x_{m_0, m_2 }+p^{( m_0,m_1)
}(x_{ m_1,m_2   }),   \text{ for all  } x=(x_{m_0,m_1}) \in C^{1}( \boldsymbol{A}).
\end{equation*}
\item $\mathrm{Z}(\boldsymbol{A})$ is the closed, hence Polish, subgroup $\mathrm{ker}(\delta^2)$ of  $C^{1}( \boldsymbol{A})$. Explicitly:
\[\mathrm{Z}(\boldsymbol{A}):=\big\{ (x_{m_0,m_1}) \in C^{1}( \boldsymbol{A}) \colon  x_{m_{0},m_{2}}=x_{m_{0},m_{1}}+p^{( m_{0},m_{1}) }(
x_{m_{1},m_{2}}), \text{ for all } m_0\leq m_1\leq m_2
 \big\}.\]
\item $\mathrm{B}(\boldsymbol{A})$ is the Polishable subgroup  $(\delta^1)(C^{0}( \boldsymbol{A}))$ of  $C^{1}( \boldsymbol{A})$; see Lemma \ref{Lemma:Borel-inverse}. Explicitly:
\[\mathrm{B}(\boldsymbol{A}):=\big\{ (x_{m_0,m_1}) \in C^{1}( \boldsymbol{A}) \colon   x_{m_0,m_1}= z_{m_0}-p^{m_0,m_1}(z_{m_1}), \text{ for some }  (z_m) \in  C^{0}( \boldsymbol{A}), \text{ and all } m_0\leq m_1
\big\}.\]
\end{itemize}

\begin{definition}\label{Definition:lim1}
Let  $\boldsymbol{A}=\left(A^{\left( m\right) },p^{\left( m,m+1\right) }\right)$  be a tower of abelian Polish groups. We denote by $\underleftarrow{\mathrm{lim}}^{1}\,\boldsymbol{A}$  the following group 
with a Polish cover: 
\[\underleftarrow{\mathrm{lim}}^{1}\,\boldsymbol{A}:= \mathrm{Z}(\boldsymbol{A})/ \mathrm{B}(\boldsymbol{A})\]
\end{definition}

\begin{remark}\label{Remark:limConsecutive}
Notice that every element $x=\left( x_{m_0,m_1}\right) $ of $\mathrm{Z}( 
\boldsymbol{A}) $ is completely determined by the values $x_{m,m+1}$
for $m\in \omega $, as each $x_{m,m}$ necessarily equals zero and for   $m_0<m_1$ we have:
\begin{equation*}
x_{m_0,m_1}=x_{m_0,m_0+1}+p^{( m_0,m_0+1) }(
x_{m_0+1,m_0+2}) +\cdots +p^{( m_0,m_1-1) }(
x_{m_1-1,m_1})
\end{equation*}
We will sometimes tacitly apply this remark in what follows.
\end{remark}
As in the case of the $\underleftarrow{\mathrm{lim}}$-functor,
 a map $[m_{k},f^{\left( k\right) }]:\boldsymbol{A}%
\rightarrow \boldsymbol{B}$ of towers of abelian Polish groups induces a continuous homomorphism $\mathrm{Z}(\boldsymbol{A})\to \mathrm{Z}(\boldsymbol{B})$ which maps  $\mathrm{B}(\boldsymbol{A})$ into $\mathrm{B}(\boldsymbol{B})$. This  induces a  definable homomorphism 
$\underleftarrow{\mathrm{lim}}^{1}[m_{k},f^{%
\left( k\right) }]:\underleftarrow{\mathrm{lim}}^{1}\boldsymbol{A}%
\rightarrow \underleftarrow{\mathrm{lim}}^{1}\boldsymbol{B}$ between groups with a Polish cover given by:
\begin{equation*}
\left( \underleftarrow{\mathrm{lim}}^{1}[m_{k},f^{\left( k\right) }]\right)
(\left( x_{m,m^{\prime }}\right) _{m\leq m^{\prime }})=(f^{\left( k\right)
}(x_{m_{k},m_{k^{\prime }}}))_{k\leq k^{\prime }}\text{.}
\end{equation*}%
It is straightforward to verify the following statement.

\begin{proposition}\label{Prop:lim^1PolishCover}
The assignments $\boldsymbol{A}\mapsto 
\underleftarrow{\mathrm{lim}}^{1}\boldsymbol{A}$, $[m_{k},f^{\left(
k\right) }]\mapsto \underleftarrow{\mathrm{lim}}^{1}[m_{k},f^{\left(
k\right) }]$, define an additive functor from the category of towers of
abelian Polish groups to the   category of 
  groups with 
Polish cover.
\end{proposition}
 The next lemma shows that  $\underleftarrow{\mathrm{lim}}^{1}\boldsymbol{A}$ fails to be Polish whenever it does not vanish. Hence, Proposition \ref{Prop:lim^1PolishCover} provides the right framework for studying the definable content of  the $\underleftarrow{\mathrm{lim}}^{1}$-functor.

\begin{lemma}
\label{Lemma:lim1-dense} For every  tower $\boldsymbol{A}=\left(A^{\left( m\right) },p^{\left( m,m+1\right) }
\right)$ of abelian Polish groups $\mathrm{B}(\boldsymbol{A}) $ is dense in $\mathrm{Z}(\boldsymbol{A})$.
\end{lemma}
\begin{proof}
Suppose that $n_{0}\in \omega $ and $a\in \mathrm{Z}(\boldsymbol{A})$
is such that $a_{n,n+1}=0$ for $n\geq n_{0}$. Then setting%
\begin{equation*}
b_{k}:=a_{k,n_{0}}=a_{k,n}\text{ for }n\geq n_{0}
\end{equation*}%
one obtains $b=( b_{k}) \in C^{0}( \boldsymbol{A}) $
such that $\delta^1 ( b) =a$.
\end{proof}

In  \cite{mcgibbon_phantom_1995} it is observed that if $\underleftarrow{\mathrm{lim}}^{1}\boldsymbol{A}\neq 0$  then $\underleftarrow{\mathrm{lim}}^{1}$ is uncountable. One consequence of the above lemma is that, in the Polish context, the observation from \cite{mcgibbon_phantom_1995} can be strengthened to the following; see Section \ref{S:EquivalenceRelation} for definitions.

Observe that if $N$ is a comeager subgroup of a Polish group $G$, then $N=G$%
. Indeed, if $g\in G$, then $N\cap Ng\neq \varnothing $ and hence $g\in
N^{-1}N\subseteq N$.

\begin{corollary}
Let $\boldsymbol{A}$ be a tower of abelian Polish groups. If $%
\underleftarrow{\mathrm{lim}}^{1}\boldsymbol{A}\neq 0$ then $E_{0}\leq _{B}%
\mathcal{R}(C^{0}(\boldsymbol{A})\curvearrowright \mathrm{Z}(\boldsymbol{A}))
$.
\end{corollary}

\begin{proof}
If $\underleftarrow{\mathrm{lim}}^{1}\boldsymbol{A}\neq 0$, then $\mathrm{B}(%
\boldsymbol{A})$ is a proper dense subgroup of $\mathrm{Z}\left( \boldsymbol{%
A}\right) $. Whence, it is not $G_{\delta }$. Therefore, the action $C^{0}(%
\boldsymbol{A})\curvearrowright \mathrm{Z}(\boldsymbol{A})$ has no $%
G_{\delta }$ orbit. The conclusion thus follows from \cite[Theorem 6.2.2]%
{gao_invariant_2009}
\end{proof}
We now record a useful criterion for the vanishing of $\underleftarrow{\mathrm{lim}}^{1}\boldsymbol{A}$. 
A tower of abelian Polish groups $\boldsymbol{A}$ is
\emph{epimorphic} when all the bonding maps $p^{\left( m,m+1\right) }:A^{\left(
m+1\right) }\rightarrow A^{\left( m\right) }$ are surjective. When $%
\boldsymbol{A}$ is an epimorphic tower then $\underleftarrow{\mathrm{lim}}%
^{1}\boldsymbol{A}\ $vanishes. More generally, $\underleftarrow{\mathrm{lim}}%
^{1}\boldsymbol{A}$ vanishes if $\boldsymbol{A}$
satisfies the \emph{Mittag-Leffler condition} condition, i.e., if for every $m\in \omega $
the decreasing sequence $\left( p^{\left( m,k\right) }( A^{(
k) }) \right) _{k\geq m}$ of subgroups of $A^{\left( m\right) }$
is eventually constant.

\begin{lemma}
\label{Lemma:ML}Suppose that $\boldsymbol{A}$ is a tower of abelian Polish
groups and consider the following assertions:
\begin{enumerate}
\item $\boldsymbol{A}$ satisfies the Mittag--Leffler condition;

\item $\boldsymbol{A}$ is isomorphic to an epimorphic tower.
\item $\underleftarrow{\mathrm{lim}}%
^{1}\boldsymbol{A}=0$ 
\end{enumerate}
Then (1) and (2) are equivalent and imply (3). If each $A^{(m)}$ in $\boldsymbol{A}$ is countable, then (1), (2), and (3) are equivalent.
\end{lemma}
\begin{proof}
For a proof see \cite[Ch. II, Section 6.2]{mardesic_shape_1982}.
\end{proof}
For an example of a tower of  abelian Polish groups where (3) does not imply (1) and (2), see \cite[Example 4.5]{mcgibbon_phantom_1995}.

\subsection{Monomorphic towers and filtrations}
For the remainder of this section we will restrict our attention to the category of towers of countable abelian Polish groups. A tower  $\boldsymbol{A}=\left(A^{\left( m\right) },p^{\left( m,m+1\right) }
\right)$ of countable abelian Polish groups  is  \emph{monomorphic} if  $p^{\left( m,m+1\right)}$ is an injection for all $m\in\omega$.  In this case, $\underleftarrow{\mathrm{lim}}^{1}%
\boldsymbol{A}$ admits general description as a group with a Polish cover; this is the content of Theorem \ref{Theorem:lim1-injective} below. 

\begin{definition}
Let $A$ be a countable abelian group. A \emph{filtration $\boldsymbol{A}=(A^{\left( m\right)})$ of $A$} is a tower $\left(A^{\left( m\right) },p^{\left( m,m+1\right)}
\right)$, so that  $A=A^{(0)}$,  $p^{\left( m,m+1\right)}$ is the inclusion map $A^{(m+1)}\subseteq A^{(m)}$, and $\bigcap_m A^{(m)}=0$.  The \emph{category of filtrations} is the  associated full subcategory of the category of towers of Polish groups. Its objects are all filtrations of all countable abelian groups. 
\end{definition}

Given any monomorphic tower $\boldsymbol{A}=\left(A^{\left( m\right) },p^{\left( m,m+1\right) }
\right)$, we can canonically assign to $\boldsymbol{A}$ a filtration $\boldsymbol{A}^{\mathrm{fil}}$. First, by replacing $\boldsymbol{A}$ with an isomorphic tower, we may always assume that $A^{\left( m+1\right) }\subseteq A^{\left( m\right) }$ for every $m\in \omega$ and that $p^{\left( m,m+1\right) }:A^{\left( m+1\right) }\rightarrow A^{\left(
m\right) }$ is the inclusion map. Let $A^{\left( \infty \right) }$  be
the intersection of these $A^{\left( m\right) }$ $(m\in \omega)$. Define $%
\boldsymbol{A}^{\mathrm{fil}}$ to be the tower of groups $\left( A^{\left(
m\right) }/A^{\left( \infty \right) },p^{\left( m,m+1\right) }\right) $
where $p^{\left( m,m+1\right) }:A^{\left( m+1\right) }/A^{\left( \infty
\right) }\rightarrow A^{\left( m\right) }/A^{\left( \infty \right) }$ is the
inclusion map. Notice that this defines a functor $\boldsymbol{A}\mapsto 
\boldsymbol{A}^{\mathrm{fil}}$ from the category of
monomorphic towers of countable abelian groups to the subcategory of filtrations. 

\begin{lemma}
\label{Lemma:reduce} If $\boldsymbol{A}$ is a monomorphic tower, then $\underleftarrow{\mathrm{lim}}^{1}\boldsymbol{A}$ and $\underleftarrow{\mathrm{lim}}^{1}\boldsymbol{A}%
^{\mathrm{fil}}$ are definably  isomorphic.
\end{lemma}
\begin{proof}
The quotient maps $A^{\left( m\right) }\rightarrow A^{\left( m\right)
}/A^{\left( \infty \right) }$ for $m\in \omega $ induce a morphism $\pi:\boldsymbol{A}\rightarrow \boldsymbol{A}^{\mathrm{fil}}$,
which induces in turn a surjective homomorphism $\bar{\pi}:\mathrm{\mathrm{lim}}^{1}\boldsymbol{A}\rightarrow \mathrm{\mathrm{lim}}
^{1}\boldsymbol{A}^{\mathrm{fil}}$ which lifts to a continuous homomorphism $\mathrm{Z}(\boldsymbol{A})\to\mathrm{Z}(\boldsymbol{A}^{\mathrm{fil}})$. Hence $\bar{\pi}$  is definable. We claim that it is also injective. Indeed, if $a\in \mathrm{Z}( \boldsymbol{A}) $ is such that 
$\bar{\pi} _{\boldsymbol{A}}( a) \in \mathrm{B}( 
\boldsymbol{A}^{\mathrm{fil}}) $, then there exists $c\in
C^{0}( \boldsymbol{A}) $ such that $a_{m,m+1}+A^{\left( \infty
\right) }=c_{m}-c_{m+1}+A^{\left( \infty \right) }$ for every $m\in \omega $%
. Hence there exist $r_{m}\in A^{\left( \infty \right) }$ such that $%
a_{m,m+1}+r_{m}=c_{m}-c_{m+1}$ for $m\in \omega $. Define then 
\begin{equation*}
d_{m}:=c_{m}+r_{0}+\cdots +r_{m-1}\in A^{\left( m\right) }
\end{equation*}%
for $m\in \omega $. Notice that%
\begin{eqnarray*}
d_{m}-d_{m+1} &=&c_{m}+r_{0}+\cdots +r_{m-1}-\left( c_{m+1}+r_{0}+\cdots
+r_{m}\right) \\
&=&c_{m}-c_{m+1}-r_{m}=a_{m}\text{.}
\end{eqnarray*}%
Hence $d$ witnesses that $a\in \mathrm{B}( \boldsymbol{A}) $. It follows that $\bar{\pi}$ is a definable isomorphism.
\end{proof}

\begin{remark}
It is easy to check that if we  set  $\pi _{
\boldsymbol{A}}:=\pi$ for the homomorphism
$\pi:\boldsymbol{A}\rightarrow \boldsymbol{A}^{\mathrm{fil}}$ defined in the above proof, then  $\boldsymbol{A}\mapsto \bar{\pi}_{\boldsymbol{A}}$ is a natural
transformation between the functors  $\boldsymbol{A}\mapsto \underleftarrow{\mathrm{lim}}^{1}\boldsymbol{A}$ and $\boldsymbol{A}\mapsto \underleftarrow{\mathrm{lim}}^{1}\boldsymbol{A}^{\mathrm{fil}}$. 
\end{remark}

By Lemma \ref{Lemma:reduce} we may restrict our study of monomorphic towers to the category of filtrations. In this category  $\underleftarrow{\mathrm{lim}}^{1}\boldsymbol{A}$ has a concrete description to which Theorem \ref{Theorem:rigid} directly applies, as we will see.

\subsection{Completions of topological groups and towers}

\begin{definition}
A topological group $G$ is \emph{pro-countable} if it is isomorphic to the inverse limit of a tower of
countable groups. 
\end{definition}

As we have noted, if $G$ is abelian then $G$ is pro-countable if and only if it is a non-Archimedean Polish group; see
 \cite[Lemma 2]{malicki_abelian_2016}.  

\begin{definition}
Let $\boldsymbol{A}$ be a  filtration of a countable abelian group $A$. Then $\boldsymbol{A}$ gives rise to  the pro-countable abelian group $\hat{A}$ which  is the inverse limit of the tower of
countable abelian groups $\left( A/A^{\left( m\right)}, f_*^{(m,m+1)} \right)$, where $f^{(m,m+1)}_*\colon A/A^{\left( m+1\right)}\to A/A^{\left( m\right)}$ is the epimorphism induced by the inclusion $f^{(m,m+1)}\colon A^{\left( m+1\right)} \to A^{\left( m\right)}$. 
We say that $\hat{A}$ is the \emph{completion of $A$ with respect to $\boldsymbol{A}$}.
\end{definition}
 Notice that the canonical homomorphism  $a\mapsto (a+A^{(0)},a+A^{(1)},\cdots)$  from $A$ to  $\hat{A}$  is injective, since the sequence $\left\{ A^{\left( m\right) }:m\in \omega \right\} $ has
trivial intersection. The resulting image of $A$ is clearly dense  in $\hat{A}$. This induces an assignment $\boldsymbol{A}
\mapsto \hat{A}/A$, which maps the tower $\boldsymbol{A}$ to the group with a Polish cover $\hat{A}/A$. This assignment  determines a functor from the category of filtrations to the
category of groups with a Polish cover. 
To see this first notice that if $\hat{A}_{m_0}$ is the completion of the subgroup $A^{(m_0)}\subseteq A$ with respect to the filtration $(A^{(m)}\colon m\geq m_0)$ then $\hat{A}_{m_0}$ is a clopen subgroup of $\hat{A}$. Since $A$ is dense in $\hat{A}$ we also have that $\hat{A}_{m_0}+A=\hat{A}$ and therefore $\hat{A}_{m_0}$ is an essential retract of $\hat{A}/A$; see the discussion preceding Definition \ref{Definition:trivial}. Notice now that if  $f=\left( m_{k},f^{\left( k\right)
}\right) $ is an inv-map from $\boldsymbol{A}$ to some filtration $\boldsymbol{B}$ then the homomorphism $f^{\left( 0\right) }:A^{(m_0)}\rightarrow B$
extends to a continuous homomorphism $\hat{f}:\hat{A}_{m_0}\rightarrow 
\hat{B}$ such that $\hat{f}( A^{(m_0)}) \subseteq B$. The resulting definable homomorphism $f\colon \hat{A}_{m_0}/A^{(m_0)}\to\hat{B}/B$ induces  a (trivial) definable homomorphism $f:\hat{A}/A\rightarrow \hat{B}/B$ by Lemma \ref{L:RetractProperty}.
It is easily verified that this defines an additive functor from the category of filtrations to the category of groups with a Polish cover. 

\subsection{$\underleftarrow{\mathrm{lim}}^{1}$ and completions}
 In  \cite[Section 4]{mcgibbon_phantom_1995} it is shown that, if $\boldsymbol{A}$ is a tower of countable abelian groups, then
 $\underleftarrow{\mathrm{lim}}^{1}\boldsymbol{A}$ vanishes if and only if $\hat{A}/A$ does as well. We have the following generalization, which was already known without the definability claim; see \cite[Remark 2.13]{GJ09}.

\begin{theorem}
\label{Theorem:lim1-injective}Suppose that $\hat{A}$ is the completion of a countable group $A$ with respect to a filtration $\boldsymbol{A}$ of $A$. Then there is a definable isomorphism between $\underleftarrow{\mathrm{lim}}^{1}\boldsymbol{A}$ and $\hat{A}/A$.
\end{theorem}

\begin{proof}
Let $a\in\mathrm{Z}(\boldsymbol{A})$ and notice that the sequence  $k\mapsto a_{0,k}=a_{0,1}+a_{1,2}+\cdots+ a_{k-1,k}$ is Cauchy  in $\hat{A}$ since
\[(a_{0,k}- a_{0,l})\in A^{(k)}+A^{(k+1)}+\cdots+A^{(l-1)}=A^{(l-1)}\subseteq \hat{A}^{(l-1)}, \text{ for all } k< l.\]
So let $\sigma( a) =\mathrm{lim}_{k\rightarrow \infty }a_{0,k}$ be the limit of the above sequence in $\hat{A}$.
Notice that this defines a continuous homomorphism $\sigma: \mathrm{Z}(\boldsymbol{A}) \rightarrow \hat{A}$. Moreover, if $a\in \mathrm{B}(\boldsymbol{A})$, then there exists $b\in
C^{0}( \boldsymbol{A})$ such that $a_{k,k+1}=b_{k}-b_{k+1}$ for
every $k\in \omega$, and hence, $\sigma( a)
=b_{0}\in A$. This shows that $\sigma$ induces  a (trivial) definable homomorphism $\underleftarrow{\mathrm{lim}}^{1}\boldsymbol{A}\rightarrow \hat{A}_{0}/A_{0}$. We now show that such a homomorphism is surjective.

If $b\in \hat{A}^{\left( 0\right) }$, then $b$ is the limit of a Cauchy
sequence $\left( b_{n}\right) _{n\in \omega }$ from $A$. After passing to a
subsequence, we may assume that $b_{n}-b_{n+1}\in A^{\left( n\right)}$ for
every $n\in \omega $. Define $a\in\mathrm{Z}( \boldsymbol{A}) $
by setting $a_{0,1}=b$ and $a_{n,n+1}=b_{n}-b_{n+1}$ for $n\geq 1$. Observe
that%
\begin{equation*}
a_{0,n+1}=a_{0,1}+a_{1,2}+\cdots +a_{n,n+1}=b_{0}+\left( b_{1}-b_{0}\right)
+\left( b_{2}-b_{1}\right) +\cdots +\left( b_{n}-b_{n-1}\right) =b_{n}
\end{equation*}%
for every $n\in \omega $. Hence 
\begin{equation*}
\sigma ( a) =\mathrm{lim}_{n}a_{0,n+1}=\mathrm{lim}_{n}b_{n}=b%
\text{.}
\end{equation*}%
We now show that the definable homomorphism $\underleftarrow{\mathrm{lim}}
^{1}\boldsymbol{A}\rightarrow \hat{A}/A$ is injective. Suppose that $a\in
\mathrm{Z}( \boldsymbol{A}) $ is such that $\sigma( a) \in A$. We claim that $a\in \mathrm{B}( \boldsymbol{A})$. Let   $m\in \omega$ and notice that $\hat{A}^{\left( m\right) }\cap
A=A^{\left( m\right) }$, where  $\hat{A}^{( m)}$ is the clopen subgroup of $\hat{A}$ corresponding to the completion of $A^{(m)}$ with respect to the filtration $\{A^{(l)}\colon m\leq l \}$.
Define
\begin{equation*}
b_{m}:=\mathrm{lim}_{k\rightarrow \infty }a_{m,k}\in \hat{A}^{\left(
m\right) }\cap A=A^{\left( m\right) },
\end{equation*}%
and notice that $a_{m,m+1}=b_{m}-b_{m+1}$, for every $m\in \omega $.
Thus $a\in \mathrm{B}(\boldsymbol{A})$.\end{proof}

\begin{remark}
It is easy to check that if we  set  $\sigma _{
\boldsymbol{A}}:=\sigma$ for the homomorphism
$\sigma( a) =\mathrm{lim}_{k\rightarrow \infty }a_{0,k}$ defined in the above proof, then  $\boldsymbol{A}\mapsto \sigma _{\boldsymbol{A}}$ is a natural transformation between the functors  $\boldsymbol{A}\mapsto \underleftarrow{\mathrm{lim}}^{1}\boldsymbol{A}$ and $\boldsymbol{A}\mapsto \hat{A}/A$ defined on the category of all filtrations of countable groups. 
\end{remark}

We may now invoke Theorem \ref{Theorem:rigid} to get the following rigidity theorem for the $%
\underleftarrow{\mathrm{lim}}^{1}$ functor, restricted to the category of filtrations of
countable groups. Recall that a functor between categories is \emph{fully
faithful }if it is bijective when restricted to hom-sets.

\begin{corollary}
\label{Corollary:lim1-injective}The functor $\boldsymbol{A}\mapsto 
\underleftarrow{\mathrm{lim}}^{1}\boldsymbol{A}$ from the category of
filtrations of countable abelian groups to the category of groups with a Polish cover is fully faithful.
\end{corollary}

\begin{proof}
In view of Theorem \ref{Theorem:lim1-injective}, it suffices to show that the
additive functor $\boldsymbol{A}\mapsto \hat{A}/A$ from the category of
filtrations of countable abelian groups to the category of groups with a Polish cover is fully faithful.

We begin by showing that such a functor is faithful. Suppose that $
\boldsymbol{A}$ and $\boldsymbol{B}$ are filtrations of countable groups $A$
and $B$, respectively. Let $[m_{k},f^{\left( k\right) }]$ be a morphism from $\boldsymbol{A}$ to $\boldsymbol{B}$ and let $f\colon \hat{A}/A\to \hat{B}/B$ be the corresponding definable homomorphism. As in the paragraph preceding Theorem \ref{Theorem:lim1-injective}, $f$ is a trivial definable homomorphism and since it is induced by
the extension $\hat{f}^{\left( 0\right) }:\hat{A}^{\left( m_{0}\right)
}\rightarrow \hat{B}$, where $\hat{A}^{\left( m_{0}\right) }\subseteq \hat{A}
$ is the completion of $A^{\left( m_{0}\right) }$ with respect to $\left\{
A^{\left( k\right) }:k\geq m_{0}\right\} $. In particular, $f( x+A) =\hat{f}^{\left( 0\right) }( x) +A$ for $x\in 
\hat{A}^{( m_{0}) }$.

Suppose now that $f$ is the zero homomorphism $f(\hat{A})=0$ from $\hat{A}%
/A$ to $\hat{B}/B$. Then
$\hat{f}^{( 0) }( \hat{%
A}^{\left( m_{0}\right) }) \subseteq B$. As $B$ is countable, by the
Baire Category Theorem, there is $b_{0}\in B$ such that $G(
b_{0}) =\{x\in \hat{A}^{\left( m_{0}\right) }:\hat{f}^{\left( 0\right)
}\left( x\right) =b_{0}\}$ is nonmeager. By Pettis' Lemma, $G( b_{0})
-G( b_{0}) $ contains an open neighborhood of $0$.  It follows that 
\begin{equation*}
\hat{A}^{\left( n_{0}\right) }\subseteq G( b_{0}) -G(
b_{0}) \subseteq \{x\in \hat{A}^{\left( m_{0}\right) }:\hat{f}^{\left(
0\right) }( x) =0\}\text{,}
\end{equation*}
for some $n_{0}\geq m_{0}$. It follows that $f^{\left( 0\right) }|_{A^{\left( n_{0}\right)
}}=0$, and therefore $[m_{k},f^{\left( k\right) }]$ is the zero morphism
from $\boldsymbol{A}$ to $\boldsymbol{B}$.

We now show that the functor is full. Suppose as above that $\boldsymbol{A}$
and $\boldsymbol{B}$ are filtrations of countable groups $A$ and $B$,
respectively. Let $f$ be a definable homomorphism from $\hat{A}/A$ to $\hat{B}/B$. We claim that there exists a morphism $[m_{k},f^{\left( k\right)
}]$ from $\boldsymbol{A}$ to $\boldsymbol{B}$ such that $f$ is the
definable homomorphism induced by $[m_{k},f^{\left( k\right) }]$. By Theorem \ref{Theorem:rigid}, there exist $m_{0}\in \omega $ and a continuous homomorphism 
$g:\hat{A}^{\left( m_{0}\right) }\rightarrow \hat{B}$ such that $f( x+A) =g( x) +A$ for every $x\in \hat{A}^{\left(
m_{0}\right) }$. Since $g$ is continuous, there exists an increasing
sequence $\left( m_{k}\right) $ in $\omega $ such that $g$ maps $\hat{A}%
^{\left( m_{k}\right) }$ to $\hat{B}^{\left( k\right) }$ for every $k\in
\omega $. One can then define $f^{\left( k\right) }:A^{\left( m_{k}\right)
}\rightarrow B^{\left( k\right) }$ to be the restriction of $g$ to $%
A^{\left( m_{k}\right) }$ for $k\in \omega $. Then $[m_{k},f^{\left(
k\right) }]$ is a morphism from $\boldsymbol{A}$ to $\boldsymbol{B}$ that
induces the Borel homomorphism $f$. This concludes the proof that the
functor is full.
\end{proof}

\begin{remark}
Notice that the conclusion of Corollary \ref
{Corollary:lim1-injective} does not hold for monomorphic towers of 
the form $A_0\geq A_1\geq A_2\geq \cdots$, when  $\bigcap_m A^{(m)}\neq 0$. Similarly, it does not generally hold for towers that are not monomorphic, as the $%
\underleftarrow{\mathrm{lim}}^{1}$ of any epimorphic tower vanishes.
\end{remark}

\begin{corollary}
\label{Corollary:ProfiniteDefinableIsomoTowers} Let $\boldsymbol{A}$ and $\boldsymbol{B}$ be filtrations of the countable abelian groups $A$ and $B$. Let also $\hat{A}$ and $\hat{B}$ be the associated completions. Then, the following are equivalent:
\begin{enumerate}
\item $\boldsymbol{A}$ and $\boldsymbol{B}$ are isomorphic objects in the category of filtrations;
\item $\hat{A}/A$ and $\hat{B}/B$ are definably isomorphic;
\item $\underleftarrow{\mathrm{lim}}^{1}\boldsymbol{A}$ is definably isomorphic to $\underleftarrow{\mathrm{lim}}^{1}\boldsymbol{B}$.
\end{enumerate} 
\end{corollary}
\begin{proof}
This is an immediate corollary of Corollary \ref{Corollary:lim1-injective} and Theorem \ref{Theorem:lim1-injective}.
\end{proof}

We should observe here that the above three equivalent conditions do not necessarily imply that the groups $\hat{A}$ and $\hat{B}$  are isomorphic; e.g. see Example \ref{Example:DefinablyIsomoNotIsomo}.

\section{Locally profinite completions of $\mathbb{Z}^{d}$ and $\mathbb{Q}^{d}$}
\label{Section:locally-profinite}

Consider the assignment $\hat{A}\mapsto \hat{A}/A$ which maps the completion $\hat{A}$ of each abelian group $A$ with respect to some filtration $\boldsymbol{A}$ to the corresponding quotient group $\hat{A}/A$.  When  $\hat{A}/A$ is viewed as an \emph{abstract group up to isomorphism} then $\hat{A}/A$ remembers little of the group structure  $\hat{A}$. 
For example, if $\hat{\mathbb{Z}}^d$ denotes the profinite completion of $\mathbb{Z}^d$ with respect to the filtration $((m!)\cdot\mathbb{Z}^d)_{m\in\omega}$, then by \cite{Nikolov2007}, the group $\hat{\mathbb{Z}}^d$ has the same finite quotients as $\mathbb{Z}^d$, and therefore $\hat{\mathbb{Z}}^{d'}$ and $\hat{\mathbb{Z}}^d$ are non-isomorphic when $d\neq d'$. However,  $\hat{\mathbb{Z}}^d/\mathbb{Z}^d$ and $\hat{\mathbb{Z}}^{d'}/\mathbb{Z}^{d'}$ are isomorphic as abstract groups since they  both
are $\mathbb{Q}$-vector spaces of the same dimension; see Theorem \ref{Corollary:iso-classification-Ext} below. 
 By Corollary \ref{Corollary:ProfiniteDefinableIsomoTowers}, the assignment $\hat{A}\mapsto \hat{A}/A$  provides a much stronger invariant for classifying completions (or filtrations) of countable abelian groups, if  $\hat{A}/A$ is instead viewed as  a \emph{group with a Polish cover up to definable isomorphism}. 

In this section we refine our analysis of the information captured by a definable isomorphism in the context of profinite completions of $\mathbb{Z}^d$.
As a corollary of Theorem \ref{Theorem:MainProfinite} below,  we show that the existence of a definable isomorphism between quotients $\hat{\mathbb{Z}}^d_{\boldsymbol{\Lambda}}/\mathbb{Z}^d$ of
profinite completions $\hat{\mathbb{Z}}^d_{\boldsymbol{\Lambda}}$ of $\mathbb{Z}^d$, where $\boldsymbol{\Lambda}$ is a filtration of $\mathbb{Z}^d$ consisting of rank-$d$ subgroups,  is  equivalent to the existence of an isomorphism between certain finite rank torsion-free abelian groups which are functorially associated to these profinite completions. This theorem very concretely implies that, in classifying $\hat{\mathbb{Z}}^d_{\boldsymbol{\Lambda}}$ up to $\mathbb{Z}^d$-preserving isomorphism,  $\hat{\mathbb{Z}}^d_{\boldsymbol{\Lambda}}/\mathbb{Z}^d$ considered up to \emph{definable} isomorphism is a strong invariant, particularly in comparison to these same groups $\hat{\mathbb{Z}}^d_{\boldsymbol{\Lambda}}/\mathbb{Z}^d$
 considered up to abstract isomorphism; see Corollary \ref{Corollary:nonisomorphic-Ext} below. 
 
 In general, despite the fact that $\hat{\mathbb{Z}}^d_{\boldsymbol{\Lambda}}/\mathbb{Z}^d$ is Ulam stable for each filtration $\boldsymbol{\Lambda}$ by Corollary \ref{Corollary:rigid}, one can often find  definable automorphisms of $\hat{\mathbb{Z}}^d_{\boldsymbol{\Lambda}}/\mathbb{Z}^d$ which do not lift to topological group automorphisms of $\hat{\mathbb{Z}}^d_{\boldsymbol{\Lambda}}$. In
the second part of this section we show that
 $\hat{\mathbb{Z}}^d_{\boldsymbol{\Lambda}}/\hat{\mathbb{Z}}^d$ is definably isomorphic to a group with a Polish cover $\hat{\mathbb{Q}}^d_{\Lambda}/\Lambda$, where  $\hat{\mathbb{Q}}^d_{\Lambda}$ is a certain \emph{locally profinite completion} of $\mathbb{Q}^d$ and $\Lambda$ is a countable subgroup of $\mathbb{Q}^d$ canonically associated to $\boldsymbol{\Lambda}$. We then show that when $\boldsymbol{\Lambda}$ is ``symmetric" enough then every definable automorphism of $\hat{\mathbb{Z}}^d_{\boldsymbol{\Lambda}}/\mathbb{Z}^d$, when transferred  to $\hat{\mathbb{Q}}^d_{\Lambda}/\Lambda$, lifts to a continuous group automorphism of $\hat{\mathbb{Q}}^d_{\Lambda}$.
This will play a crucial role in Section \ref{S:ActionsCocycleRig}.

\subsection{Profinite completions of $\mathbb{Z}^d$}

Recall that the \emph{rank} of an abelian group is the cardinality of a maximal linearly independent (over $\mathbb{Z}$) subset. 
We say that the group $A$ \emph{has no free summand} if for every decomposition $A=A_0\oplus A_1$, if $A_0$ is free then   $A_0=0$.
By a \emph{rank $d$ filtration} we mean any filtration $\boldsymbol{\Lambda}=(\Lambda^{(m)})$ of  $\mathbb{Z}^d$ with the additional property that each $\Lambda^{(m)}$ is a rank $d$ subgroup of $\mathbb{Z}^d$. We denote by $\mathrm{Fil}(\mathbb{Z}^{*})$ the category of finite-rank filtrations, with the tower maps $[m_k,f^{(k)}]$ as morphisms. Notice that if $A\leq\mathbb{Z}^d$, then
\[ A\text{ is of finite index in }\mathbb{Z}^d \iff A \text{ is rank }d \iff  A\text{ is isomorphic to }\mathbb{Z}^d. \]
We can therefore associate to each object $\boldsymbol{\Lambda}$ from $\mathrm{Fil}(\mathbb{Z}^{*})$  the \emph{profinite completion} $\hat{\mathbb{Z}}^d_{\boldsymbol{\Lambda}}$ of $\mathbb{Z}^d$ given by
\[\hat{\mathbb{Z}}^d_{\boldsymbol{\Lambda}}:=  \underleftarrow{\mathrm{lim}} \;  \big(\mathbb{Z}^d/\Lambda^{(m)}\big).\]
As in Section \ref{Section:towers}, notice that $\hat{\mathbb{Z}}^d_{\boldsymbol{\Lambda}}$ contains a canonical copy of $\mathbb{Z}^d$ as a dense subgroup. Conversely, every profinite completion of $\mathbb{Z}^d$  is of the above form for appropriately chosen $\boldsymbol{\Lambda}$.

By a \emph{rank $d$ cofiltration} we mean any increasing sequence  $\boldsymbol{\Lambda}_{\mathrm{co}}=(\Lambda_{(m)})$ of rank $d$ subgroups of $\mathbb{Q}^d$ with $\Lambda_{(0)}=\mathbb{Z}^d$ and so that $\Lambda:=\bigcup_m \Lambda_{(m)}$ has no free summand. A \emph{cotower map}  $[m_k, f_{(k)}]\colon (A_{(m)})\to (B_{(m)})$ between cofiltrations, also known as a \emph{morphism of inductive sequences}, is the congruence class of a sequence of group homomorphisms $f_{(k)}\colon A_{(k)}\to B_{(m_k)}$ with $f_{(k+1)}\upharpoonright A_{(k)}=f_{(k)}$, where $(m_k, f_{(k)})$ is congruent to $(m'_k, f'_{(k)})$ if for all $k\in \omega$ we have that
 $f_{(k)}=f'_{(k)}$ as maps from $A_{(k)}$ to $B_{(\max\{m_k,m'_k\} )}$.  Composition of cotower maps is defined in analogy with the composition of tower maps given in Section \ref{Section:towers}. We denote by $\mathrm{coFil}(\mathbb{Z}^{*})$ the category of finite-rank cofiltrations, with the cotower maps $[m_k,f_{(k)}]$ as morphisms.

We lastly consider the category $\mathrm{Groups}_{+}(\mathbb{Z}^{*},\mathbb{Q}^{*})$ whose objects are groups having no free summands satisfying $\mathbb{Z}^d\leq \Lambda\leq \mathbb{Q}^d$ for some $d\in\omega$ and whose morphisms are simply the homomorphisms between any two such groups. We are going to show that the following three  categories are equivalent: 
\[\mathrm{Fil}(\mathbb{Z}^{*}),\quad  \mathrm{coFil}(\mathbb{Z}^{*}), \quad \mathrm{Groups}_{+}(\mathbb{Z}^{*},\mathbb{Q}^{*}).\]

\begin{lemma}\label{Lemma:limFunctorEquivalence}
The functor $\underrightarrow{\mathrm{lim}}\colon \mathrm{coFil}(\mathbb{Z}^{*}) \to \mathrm{Groups}_{+}(\mathbb{Z}^{*},\mathbb{Q}^{*})$ which maps  $[m_k, f_{(k)}]\colon (A_{(m)})\to (B_{(m)})$  to the homomorphism $\bigcup_k f_{(k)}\colon\bigcup_k A_{(k)}\to \bigcup_k B_{(k)}$, is  fully faithful and essentially surjective.
\end{lemma}
\begin{proof}
The functor $\underrightarrow{\mathrm{lim}}$  is full since for every homomorphism $f\colon \bigcup_k A_{(k)}\to \bigcup_k B_{(k)}$ and every $k\in\omega$ the group $A_{(k)}$ is finitely-generated, and hence $f(A_{(k)})$ is contained in some $B_{(m_k)}$, for large enough $m_k\in\omega$. It is also faithful, since $\bigcup_k f_{(k)}$ being the constant $0$ homomorphism implies the same for each $f_{(k)}$.  We lastly show  that it is also  essentially surjective. 
For any group $\Lambda$ in  $\mathrm{Groups}_{+}(\mathbb{Z}^{*},\mathbb{Q}^{*})$, with  $\mathbb{Z}^d \leq\Lambda\leq \mathbb{Q}^d$, let $\{l_1,l_2,\ldots\}$ be an enumeration of $\Lambda$ and let $\Lambda_{(m)}$ be the smallest subgroup of  $\Lambda$ which contains $\mathbb{Z}^{d}$, as well as $l_1,\ldots,l_m$. Clearly $(\Lambda_{(m)})$ is a rank $d$ cofiltration, with $\underrightarrow{\mathrm{lim}}(\Lambda_{(m)})=\Lambda,$ and therefore  $\underrightarrow{\mathrm{lim}}$ is surjective on objects.
\end{proof}

\begin{remark}\label{Remark:ChoiceOfInverse}
By the previous lemma, the functor $\underrightarrow{\mathrm{lim}}\colon \mathrm{coFil}(\mathbb{Z}^{*}) \to \mathrm{Groups}_{+}(\mathbb{Z}^{*},\mathbb{Q}^{*})$ is an equivalence of categories. As a consequence $\underrightarrow{\mathrm{lim}}$ admits a (non-unique) inverse up to isomorphism. We define an explicit  such inverse $(\underrightarrow{\mathrm{lim}})^{-1}\colon \mathrm{Groups}_{+}(\mathbb{Z}^{*},\mathbb{Q}^{*})\to \mathrm{coFil}(\mathbb{Z}^{*})$ here. Fix an enumeration of $\mathbb{Q}^{<\omega}:=\bigcup_{d\in\omega}\mathbb{Q}^d$. For each object $\Lambda$ of $\mathrm{Groups}_{+}(\mathbb{Z}^{*},\mathbb{Q}^{*})$ we let $(\underrightarrow{\mathrm{lim}})^{-1}(\Lambda)$ be the  cofiltration built as in the proof of Lemma \ref{Lemma:limFunctorEquivalence} but with the unique choice of enumeration $\{l_1,l_2,\cdots\}$ of $\Lambda$ which agrees with the global enumeration of $\mathbb{Q}^{<\omega}$. Having chosen the assignment $\Lambda\mapsto(\underrightarrow{\mathrm{lim}})^{-1}(\Lambda)$, there is a canonical way to extend this definition to the desired  functor 
\[(\underrightarrow{\mathrm{lim}})^{-1}\colon \mathrm{Groups}_{+}(\mathbb{Z}^{*},\mathbb{Q}^{*})\to \mathrm{coFil}(\mathbb{Z}^{*}).\]
 \end{remark}

Next we define a contravariant functor $\mathrm{Adj}\colon\mathrm{coFil}(\mathbb{Z}^{*})\to  \mathrm{Fil}(\mathbb{Z}^{*})$. We will need a few definitions.

\begin{definition}
\label{Definition:dual}Consider the pairings $\left\langle \,\cdot
\,,\,\cdot \,\right\rangle :\mathbb{R}^{d}\times \mathbb{R}^{d}\rightarrow 
\mathbb{R}$ and $\left\langle \,\cdot \,,\,\cdot \,\right\rangle _{_{\mathbb{%
R}^{d}}}:\mathbb{R}^{d}\times \mathbb{R}^{d}\rightarrow \mathbb{R}/\mathbb{Z}
$ on $\mathbb{R}^{d}$ defined by%
\begin{align*}
\left\langle \left( a_{1},\ldots ,a_{d}\right) ,\left( b_{1},\ldots
,b_{d}\right) \right\rangle =& \;a_{1}b_{1}+\cdots +a_{d}b_{d} \\
\left\langle \left( a_{1},\ldots ,a_{d}\right) ,\left( b_{1},\ldots
,b_{d}\right) \right\rangle _{_{\mathbb{R}^{d}}}=& \;a_{1}b_{1}+\cdots
+a_{d}b_{d}+\mathbb{Z}
\end{align*}%
for $a,b\in \mathbb{R}^{d}$. For any subgroup $A$ of $\mathbb{R}^{d}$, the 
\emph{annihilator} of $A$ is  the following closed
subgroup  of $\mathbb{R}^{d}$:
\begin{equation*}
A_{\bot }=\left\{ x\in \mathbb{R}^{d}:\forall y\in A\,\left\langle
x,y\right\rangle _{\mathbb{R}^{d}}=0\right\}.
\end{equation*}
\end{definition}
 Observe that the assignment $A\mapsto A_{\bot }$
defines an order-reversing permutation of the set of finitely-generated
subgroups of $\mathbb{Q}^{d}$ with fixed point $\mathbb{Z}_{\bot }^{d}=%
\mathbb{Z}^{d}$. Moreover, if $f\colon \Lambda'\to \Lambda$ is a homomorphism between two groups  $\Lambda'\leq \mathbb{Q}^{d'}$, $\Lambda\leq \mathbb{Q}^d$ and $\Lambda$ is of rank $d$, then $f$ extends uniquely to a $\mathbb{Q}$-linear map from $\mathbb{Q}^{d'}$ to $\mathbb{Q}^{d}$. We can therefore identify $f$ with an element $F$ of $\textrm{M}_{d\times
d^{\prime }}\left( \mathbb{Q}\right) $, i.e., with a $d\times d^{\prime }$
matrix with entries in $\mathbb{Q}$. Let $F^{\mathrm{T}}$  be the transpose matrix. As a linear transformation, $F^{\mathrm{T}}$ maps $\mathbb{Q}^{d}$ to $\mathbb{Q}^{d'}$ so that $F^{\mathrm{T}}(\Lambda_{\bot})\subseteq \Lambda'_{\bot}$. Indeed, if $x\in \Lambda_{\bot}$ and $y\in \Lambda'$, then
\begin{equation*}
\left\langle F^{\mathrm{T}}(x) ,y\right\rangle
=\left\langle x,F(y) \right\rangle =0.
\end{equation*}
We define $f^{\mathrm{T}}\colon \Lambda_{\bot}\to \Lambda'_{\bot}$ to be the restriction of the linear transformation $F^{\mathrm{T}}$ to $\Lambda_{\bot}$. Since  $(F \cdot F')^{\mathrm{T}}=(F')^{\mathrm{T}}\cdot (F)^{\mathrm{T}}$ for every two composable matrices, the assignment $f\mapsto f^{\mathrm{T}}$ is a functor from the category of all homomorphisms  $f\colon \Lambda'\to \Lambda$, where  $\Lambda\leq \mathbb{Q}^d$ is of rank $d$ and $\Lambda'\leq \mathbb{Q}^{d'}$ is of rank $d'$, to its opposite category.

\begin{lemma}\label{Lemma:dual-perp}

\begin{enumerate}
\item[]
\item If $A$ is a rank $d$ finitely-generated subgroup of $\mathbb{Q}^d$, then $A_{\bot}$ is  naturally isomorphic to $\mathrm{Hom}\left( A,\mathbb{Z}\right)$.
\item If $A$ is a rank $d$ finitely-generated subgroup of $\mathbb{Q}^d$, then 
$(A_{\bot })_{\bot}=A$.
\item  If $(A_{(m)})_{m\in\omega}$ is a rank $d$ cofiltration then $(A^{(m)})_{m\in\omega}$, where $A^{(m)}:=(A_{(m)})_{\bot}$,
 is a rank $d$ filtration.
\item  If $(A^{(m)})_{m\in\omega}$ is a rank $d$ filtration then $(A_{(m)})_{m\in\omega}$, where $A_{(m)}:=(A^{(m)})_{\bot}$,
 is a rank $d$ cofiltration.
\end{enumerate}
\end{lemma}
\begin{proof}
For (1), notice that an element $b$ of $A_{\bot }$ defines an element $\varphi _{b}\in \mathrm{Hom%
}( A,\mathbb{Z}) $ given by $\varphi _{b}( a)
=\left\langle b,a\right\rangle $. The assignment $b\mapsto \varphi _{b}$
defines an injective group homomorphism $\Phi _{A}:A_{\bot }\rightarrow 
\mathrm{Hom}( A,\mathbb{Z}) $. We now show that $\Phi _{A}$ is
onto. Suppose that $f\in \mathrm{Hom}( A,\mathbb{Z}) $. Let $%
e_{1},\ldots ,e_{d}$ be the elements of the canonical basis of $\mathbb{Q}%
^{d}$, and fix a $\mathbb{Z}$-basis $v_{1},\ldots ,v_{d}$ of $A$. Fix $\psi
\in \mathrm{GL}_{d}( \mathbb{Q}) $ such that $\psi (
v_{i}) =e_{i}$ for $i\in \left\{ 1,2,\ldots ,d\right\} $. Define now $%
w=f( v_{1}) e_{1}+\cdots +f( v_{d}) e_{d}$, and $%
b:=\psi ^{\mathrm{T}}w\in \mathbb{Q}^{d}$. It follows that $f=\varphi _{b}$, since:
\begin{equation*}
\left\langle b,v_{i}\right\rangle =\left\langle \psi ^{\mathrm{T}%
}w,v_{i}\right\rangle =\left\langle w,\psi ( v_{i}) \right\rangle
=\left\langle w,e_{i}\right\rangle =f( v_{i}) \text{.}
\end{equation*}
For naturality one checks that if  $f\colon B\to A$ is a homomorphism from a rank $d'$ finitely-generated subgroup $B$ of $\mathbb{Q}^{d'}$ to  a rank $d$ finitely-generated subgroup   $A$ of $\mathbb{Q}^{d}$, then  for the induced $ \mathrm{Hom}(f,\mathbb{Z})\colon  \mathrm{Hom}(A,\mathbb{Z})\to  \mathrm{Hom}(B,\mathbb{Z})$ we have  $\mathrm{Hom}(f,\mathbb{Z}) \circ \Phi _{A}=\Phi _{B}\circ f^{\mathrm{T}}$.

The statement (2) follows from (1), since $a\mapsto (\varphi\mapsto\varphi(a))$ induces a natural isomorphism from $A$ to $\mathrm{Hom}(\mathrm{Hom}(A),\mathbb{Z})$.
For (3), notice first that by (1), each term $A^{(m)}$ is finitely-generated of rank $d$ since so is $\mathrm{Hom}(A^{(m)},\mathbb{Z})$. We also have that $A^{(0)}=\mathbb{Z}^d$ since $(\mathbb{Z}^d)_{\bot}=\mathbb{Z}^d$. Since $A\mapsto A_{\bot}$ is order reversing we have that $A^{(0)}\supseteq A^{(1)}\supseteq \cdots$. We are left to show that $\bigcap_{m}A^{(m)}=0$. Let  $\Lambda:=\bigcup_{m}A_{(m)}$. By
\cite[Theorem 8.47]{hofmann_structure_2013}, $\Lambda$ can be written as a direct sum $\Lambda_0\oplus\Lambda_1$, so that $\Lambda_0$ is finitely-generated, $\Lambda_1$ has no free summand, and $\mathrm{Hom}(\Lambda,\mathbb{Z})=\mathrm{Hom}(\Lambda_0,\mathbb{Z})$. Since  $(A_{(m)})_{m\in\omega}$ is a cofiltration we have that $\Lambda_0=0$, and therefore $\mathrm{Hom}(\Lambda_0,\mathbb{Z})=0$. But $\bigcap_m (A^{(m)})\subseteq (\Lambda_0)_{\bot}$, where $(\Lambda_0)_{\bot}=0$ by (1). A similar argument proves (4).
\end{proof}

\begin{lemma}\label{Lemma:Adj}
The assignment  $[m_k, f_{(k)}]\mapsto \mathrm{Adj}([m_k, f_{(k)}])$ which sends a  cotower map  $[m_k, f_{(k)}]$ between cofiltrations of finite rank to the tower map $[m_k, f^{(k)}]$, with $f^{(k)}:=f^{\mathrm{T}}_{(k)}$, is a fully faithful and essentially surjective contravariant functor from 
$\mathrm{coFil}(\mathbb{Z}^{*})$ to $\mathrm{Fil}(\mathbb{Z}^{*})$. \end{lemma}
\begin{proof}

Let $[m_k, f_{(k)}]\colon (A_{(m)})_m\to (B_{(m)})_m$ be a cotower map from a rank $d'$ cofiltration to a rank $d$ cofiltration. 
Notice that the sequence $(f_{(k)})_k$ is entirely determined by  $f_{(0)}$, since $\mathbb{Q}$-linear combinations of $A_{(0)}=\mathbb{Z}^{d'}$ span $\mathbb{Q}^{d'}$. This shows, on the one hand, that $(m_k, f^{(k)})$  is an inv-map from  $ (B^{(m)})_m$ to $(A^{(m)})_m$, and on the other hand, that  if  $[m'_k, f'_{(k)}]=[m_k, f_{(k)}]$ then $[m'_k, f'^{(k)}]=[m_k, f^{(k)}]$. 
Hence $[m_k, f_{(k)}]\mapsto \mathrm{Adj}([m_k, f_{(k)}])$ maps indeed elements of $\mathrm{coFil}(\mathbb{Z}^{*})$ to elements contravariantly $\mathrm{Fil}(\mathbb{Z}^{*})$. It is clearly a functor since $(F \cdot F')^{\mathrm{T}}=(F')^{\mathrm{T}}\cdot (F)^{\mathrm{T}}$ for every two composable matrices.

By (2), (3), and (4) of Lemma \ref{Lemma:dual-perp} we have that $\mathrm{Adj}$  is surjective on objects of $\mathrm{Fil}(\mathbb{Z}^{*})$. It is also fully faithful since every cotower map can be identified with a matrix $(d\times d')$-matrix which maps $\mathbb{Z}^{d'}$ to $\mathbb{Z}^{d}$ and in this category of matrices the transpose $F\mapsto F^{\mathrm{T}}$ is its own inverse.
\end{proof}

The following theorem is an immediate consequence of Lemmas \ref{Lemma:limFunctorEquivalence} and  \ref{Lemma:Adj}, Remark \ref{Remark:ChoiceOfInverse}, and Corollary \ref{Corollary:lim1-injective}. Notice that the inclusion of $\mathrm{Groups}_{+}(\mathbb{Z}^{*},\mathbb{Q}^{*})$ into the category of torsion-free finite rank abelian groups with no free summands is fully faithful and essentially surjective. In what follows we can therefore identify these categories.

\begin{theorem}\label{Theorem:MainProfinite}
The following composition of functors provides a fully faithful contravariant functor $\Lambda \mapsto \mathbb{\hat{Z}}_{\boldsymbol{\Lambda} }^{d}/\mathbb{Z}^{d}$:
\[\Lambda \overset{ (\underrightarrow{\mathrm{lim}})^{-1}}{\xmapsto{{\hspace*{1.2cm}}}} \boldsymbol{\Lambda}_{\mathrm{co}} \overset{ \mathrm{Adj}}{\xmapsto{{\hspace*{1.2cm}}}} \boldsymbol{\Lambda} \overset{\underleftarrow{\mathrm{lim}}^1}{\xmapsto{{\hspace*{1.2cm}}}} \mathbb{\hat{Z}}_{\boldsymbol{\Lambda}}^{d}/\mathbb{Z}^{d},\]
 from the category of finite-rank   torsion-free abelian groups with no free direct summands to the category of groups with a Polish cover. The first two functors  are equivalences between the categories: 
\[\mathrm{Groups}_{+}(\mathbb{Z}^{*},\mathbb{Q}^{*}), \quad \quad   \mathrm{coFil}(\mathbb{Z}^{*}), \quad \quad   \mathrm{Fil}(\mathbb{Z}^{*}).\]
\end{theorem}

The following immediate corollary provides a combinatorial criterion for the existence of a definable isomorphism between the quotients $\mathbb{\hat{Z}}_{\boldsymbol{A} }^{d}/\mathbb{Z}^{d}$, $\mathbb{\hat{Z}}_{\boldsymbol{B} }^{d}/\mathbb{Z}^{d}$, of any two profinite completions of $\mathbb{Z}^{d}$.

\begin{corollary}\label{Corollary:CombinatorialDescriptionOfDEfIso}
Let $\mathbb{\hat{Z}}_{\boldsymbol{A} }^{d}/\mathbb{Z}^{d}$ and $\mathbb{\hat{Z}}_{\boldsymbol{B} }^{d}/\mathbb{Z}^{d}$ be quotients of the profinite completions of $\mathbb{Z}^{d}$ with respect to the filtrations $\boldsymbol{A}$, $\boldsymbol{B}$. Let also $\Lambda_{\boldsymbol{A}}$, $\Lambda_{\boldsymbol{B}}$ be the countable groups $\underrightarrow{\mathrm{lim}}(\mathrm{Adj}(\boldsymbol{A}))$, $\underrightarrow{\mathrm{lim}}(\mathrm{Adj}(\boldsymbol{B}))$.
 The following are equivalent:
 \begin{enumerate}
 \item  $\mathbb{\hat{Z}}_{\boldsymbol{A} }^{d}/\mathbb{Z}^{d}$ and $\mathbb{\hat{Z}}_{\boldsymbol{B} }^{d}/\mathbb{Z}^{d}$ are definably isomorphic;
 \item  $\boldsymbol{A}$ and  $\boldsymbol{B}$ are isomorphic objects in the category of filtrations;
 \item $\Lambda_{\boldsymbol{A}}$ and $\Lambda_{\boldsymbol{B}}$ are isomorphic as discrete groups.
 \end{enumerate}
\end{corollary}

Notice that the collection $\mathrm{Obj}(\mathrm{Fil}(\mathbb{Z}^d))$ of all objects $\boldsymbol{A}$ in the full subcategory $\mathrm{Fil}(\mathbb{Z}^d)$  of $\mathrm{Fil}(\mathbb{Z}^{*})$, which consists of all filtrations on $\mathbb{Z}^d$, inherits from $(2^{\mathbb{N}})^{\mathbb{N}}$ a Polish topology. We write $\boldsymbol{A}\simeq_{\mathrm{pro}}\boldsymbol{B}$ to declare that $\boldsymbol{A}$ and $\boldsymbol{B}$ are isomorphic as objects in the category $\mathrm{Fil}(\mathbb{Z}^d)$. We have the following corollary:

\begin{corollary}\label{Corollary:proCategory}
Within the Borel reduction hierarchy, the classification problems $(\mathrm{Obj}(\mathrm{Fil}(\mathbb{Z}^d)),\simeq_{\mathrm{pro}})$ satisfy:
\[(\mathrm{Obj}(\mathrm{Fil}(\mathbb{Z}^1)),\simeq_{\mathrm{pro}})<_B (\mathrm{Obj}(\mathrm{Fil}(\mathbb{Z}^2)),\simeq_{\mathrm{pro}}) <_B (\mathrm{Obj}(\mathrm{Fil}(\mathbb{Z}^3)),\simeq_{\mathrm{pro}})<_B \cdots\]
\end{corollary}
\begin{proof}
This follows from Corollary \ref{Corollary:CombinatorialDescriptionOfDEfIso}, Remark \ref{Remark:ChoiceOfInverse}, and the results from \cite{adams_linear_2000},\cite{hjorth_1999_rank2},\cite{thomas_classification_2003}; see Example \ref{Example:Rank_d}.
\end{proof}

\subsection{Locally profinite completions of $\mathbb{Q}^d$} 
Given a definable homomorphism $f\colon \mathbb{\hat{Z}}_{\boldsymbol{A} }^{d}/\mathbb{Z}^{d}\to\mathbb{\hat{Z}}_{\boldsymbol{B} }^{d}/\mathbb{Z}^{d}$ it is not necessarily true that $f$ lifts to a continuous homomorphism from $ \mathbb{\hat{Z}}_{\boldsymbol{A} }^{d}$ to $\mathbb{\hat{Z}}_{\boldsymbol{B} }^{d}$. In particular, using the following example one can  construct a profinite completion $\mathbb{\hat{Z}}_{\boldsymbol{A}\times\boldsymbol{B} }^{2}$ of $\mathbb{Z}^2$ so that $\mathbb{\hat{Z}}_{\boldsymbol{A}\times\boldsymbol{B} }^{2}/\mathbb{Z}^2$ admits definable automorphisms which do not lift to a topological group isomorphism of $\mathbb{\hat{Z}}_{\boldsymbol{A}\times\boldsymbol{B} }^{2}$.

\begin{example}\label{Example:DefinablyIsomoNotIsomo}
Consider the filtrations $\boldsymbol{A}=(A^{(m)})_m$ and $\boldsymbol{B}=(B^{(m)})_m$ of $\mathbb{Z}$, where  $A^{(0)}=B^{(0)}=\mathbb{Z}$, $A^{(m)}=3\cdot 2^{m-1}\cdot\mathbb{Z}$, and $B^{(m)}= 2^{m-1}\cdot\mathbb{Z}$, for $m>0$. Then the  maps $f^{(k)}\colon A^{(k+1)}\to B^{(k)}$ and $g^{(k)}\colon B^{(k)}\to A^{(k)},$ with $f^{(k)}(x)=\frac{1}{3}\cdot x$  and $g^{(k)}(x)=3\cdot x$, provide an isomorphism between the objects  $\boldsymbol{A}$ and $\boldsymbol{B}$ in the category of filtrations. However,  $\hat{Z}_{\boldsymbol{A}}$ and $\hat{Z}_{\boldsymbol{B}}$ are not isomorphic since   $\hat{B}$ does not have any index $3$ subgroup.  
\end{example}

We now show how to replace each group with  Polish cover of the form $\mathbb{\hat{Z}}_{\boldsymbol{\Lambda}}^{d}/\mathbb{Z}^{d}$ with a definably isomorphic group with  Polish cover $\mathbb{\hat{Q}}_{\Lambda}^{d}/\Lambda$, which often enjoys stronger lifting properties.
 Fix $d\geq 1$ and let $\Lambda$ be a rank $d$ torsion-free abelian
group with no free summand. Let $\boldsymbol{\Lambda}_{\mathrm{co}}=(\Lambda_{(0)}\subseteq \Lambda_{(1)}\subseteq\cdots)$ be the cofiltration associated to $\Lambda$ via the  functor $(\underrightarrow{\mathrm{lim}})^{-1}$, and let $\boldsymbol{\Lambda}=(\Lambda^{(0)}\supseteq\Lambda^{(1)}\supseteq\cdots)$ be dual filtration. We denote by $\hat{\Lambda}_{(m)}$ the profinite completion of $\Lambda_{(m)}$ with respect to the filtration: 
\[(\Lambda_{(m)}\supseteq \Lambda^{(0)}\supseteq\Lambda^{(1)}\supseteq\cdots).\]
For each  $m\in \omega$, the inclusion map $\Lambda _{\left( m\right)
}\rightarrow \Lambda _{\left( m+1\right) }$ extends to a topological
embedding $\hat{\Lambda}_{\left( m\right) }\rightarrow \hat{\Lambda}_{\left(
m+1\right) }$.  We define $\mathbb{\hat{Q}}_{\Lambda }^{d}$ to be the inductive
limit of the sequence $(\hat{\Lambda}_{\left( m\right) })_{m\in \omega }$.
Notice that $\mathbb{\hat{Q}}_{\Lambda }^{d}$ is a locally profinite group that contains $\Lambda $ as a dense subgroup.
We say that   $\mathbb{\hat{Q}}_{\Lambda }^{d}$  is the \emph{locally profinite completion of $\mathbb{Q}^d$ with respect to $\Lambda$}.
Notice that $\mathbb{\hat{Z}}_{\Lambda }^{d}$  is an essential retract of
$\mathbb{\hat{Q}}_{\Lambda }^{d}$ with $\mathbb{\hat{Z}}_{\Lambda }^{d}\cap \Lambda =\Lambda^{(0)}=\Lambda_{(0)}=\mathbb{Z}
^{d}$. By Lemma \ref{L:RetractProperty}  the inclusion map $
\mathbb{\hat{Z}}_{\Lambda }^{d}\rightarrow \mathbb{\hat{Q}}_{\Lambda }^{d}$
induces a definable isomorphism between $\mathbb{\hat{Z}}_{\Lambda }^{d}/\mathbb{Z}^{d}$ and $\mathbb{\hat{Q}}_{\Lambda }^{d}/\Lambda$. The following example illustrates that the above procedure is a group-theoretic analogue of the construction of the field of fractions in the theory of rings.

\begin{example}
\label{Example:p-adic}
Let $p$ be a  prime number and let $\mathbb{Q}_p$ be the field of all $p$-adic numbers, i.e. the field of fractions of the ring $\mathbb{Z}_p$ of all $p$-adic integers --- see \cite[Chapter 1]{robert_course_2000} for a primer on the rings $\mathbb{Z}_{p}$ and $\mathbb{Q}_{p}$. Let also $\mathbb{Z}[1/p]$ be the subring of $\mathbb{Q}_p$ that is generated by $1/p$. If $\mathbb{Q}_p, \mathbb{Z}_p, \mathbb{Z}[1/p]$ are viewed as Polish abelian groups with respect to their additive structure then  $\mathbb{Q}_p$ is simply the locally profinite completion of $\mathbb{Q}$ with respect to $\Lambda:=\mathbb{Z}[1/p]$,  as above. More generally, if  $\Lambda:=\mathbb{Z}[1/p]^{d}\leq \mathbb{Q}^d$, then  $\mathbb{\hat{Z}}_{\Lambda }^{d}$ is isomorphic to $\mathbb{Z}_{p}^{d}$ and $\mathbb{\hat{Q}}_{\Lambda }^{d}$ is isomorphic to $\mathbb{Q}_{p}^{d}$.
\end{example}

Let  $\Lambda, \Lambda'\in\mathrm{Groups}_{+}(\mathbb{Z}^{*},\mathbb{Q}^{*})$ be finite rank torsion-free abelian groups with no free summands. A homomorphism $g\colon \Lambda'\to\Lambda$ is a  \emph{$\mathrm{T}$-homomorphism} if $g^{\mathrm{T}}$ maps  $\Lambda$ to $\Lambda'$. Here, $g^{\mathrm{T}}$ is the homomorphism from the divisible hull of $\Lambda $ to the divisible hull of $\Lambda '$ associate with the \emph{transpose} of the matrix with rational coefficients associated with $g$ with respect to some choice of maximal independent sets in $\Lambda '$ and $\Lambda $ respectively.  Let $\mathrm{Groups}_{+}^{\mathrm{T}}(\mathbb{Z}^{*},\mathbb{Q}^{*})$  
be the subcategory of $\mathrm{Groups}_{+}(\mathbb{Z}^{*},\mathbb{Q}^{*})$ which contains the same objects but whose arrows are precisely all $\mathrm{T}$-homomorphisms.  Notice that this new category contains strictly fewer isomorphisms than $\mathrm{Groups}_{+}(\mathbb{Z}^{*},\mathbb{Q}^{*})$.

\begin{example}\label{Examples:LocallyProfinite}
\begin{enumerate}
\item[]
\item Let $\Lambda=\Lambda'=\mathbb{Q}^d$ and notice that the    \emph{zero homomorphism} $(x,y)\mapsto (0,0)$ is a $\mathrm{T}$-homomorphism since its transpose is also $(x,y)\mapsto (0,0)$.
\item Let $\Lambda=\Lambda'=\mathbb{Z}[1/2]\oplus\mathbb{Z}[1/6]\leq \mathbb{Q}^2$ and let $\varphi$ be the assignment $(x,y)\mapsto (x,2x+y)$. Notice that $\varphi\colon \Lambda'\to \Lambda$ is an isomorphism. However,   the  assignment $\varphi^{\mathrm{T}}\colon (x,y)\mapsto (x+2y,y)$ will fail to map $(0,1/6)\in\Lambda$ to $\Lambda'$.
\end{enumerate}
\end{example}

Let now  $g\colon \Lambda'\to\Lambda$ be a  $\mathrm{T}$-homomorphism and notice that  $g^{\mathrm{T}}\colon \Lambda\to\Lambda'$ extends to a continuous homomorphism:
   \[\hat{g}^{\mathrm{T}}\colon\mathbb{\hat{Q}}_{\Lambda }^{d}\to \mathbb{\hat{Q}}_{\Lambda '}^{d}.\] 
To see this, notice that it will suffice to show that for each $m\in\omega$, there is some $n_m\in\omega$ so that  $\hat{g}^{\mathrm{T}}\upharpoonright\hat{\Lambda}_{(m)}$ is a continuous homomorphism from $\hat{\Lambda}_{(m)} $  to $\hat{\Lambda}'_{(n_m)}$. It is easy to see that any choice of $n_m$ with $g^{\mathrm{T}}(\Lambda_{(m)})\subseteq \Lambda'_{(n_m)}$ works.  Since  $\hat{g}^{\mathrm{T}}\colon\mathbb{\hat{Q}}_{\Lambda }^{d}\to \mathbb{\hat{Q}}_{\Lambda '}^{d}$ is a continuous homomorphism with $\hat{g}^{\mathrm{T}}(\Lambda)\subseteq \Lambda'$, it induces a definable homomorphism: 
 \[\bar{g}^{\mathrm{T}}\colon \mathbb{\hat{Q}}_{\Lambda }^{d}/\Lambda \rightarrow \mathbb{\hat{Q}}_{\Lambda'}^{d'}/\Lambda ^{\prime}.\]
It follows that the assignment  $\Lambda\mapsto \bar{\mathbb{Q}}^d_{\Lambda}/\Lambda$, $g\mapsto\bar{g}^{\mathrm{T}}$ is a contravariant functor from the category $\mathrm{Groups}_{+}^{\mathrm{T}}(\mathbb{Z}^{*},\mathbb{Q}^{*})$   to the category of groups with a Polish cover. Definable homomorphisms $\hat{\mathbb{Q}}^d_{\Lambda}/\Lambda\to \hat{\mathbb{Q}}^{d'}_{\Lambda}/\Lambda'$, which are of the form   $\bar{g}^{\mathrm{T}}$ for some $g\colon \Lambda'\to \Lambda$, admit by definition a lift $\hat{g}^{\mathrm{T}}\colon\hat{\mathbb{Q}}^d_{\Lambda}\to \hat{\mathbb{Q}}^{d'}_{\Lambda'}$ which has the additional property of being a continuous homomorphism. 
It turns out that every definable homomorphism which has this additional property is of the form $\bar{g}^{\mathrm{T}}$ for some $g\colon \Lambda'\to \Lambda$:

\begin{theorem}
\label{Theorem:trivial}The assignment $\Lambda \mapsto \mathbb{\hat{Q}}%
_{\Lambda }^{d}/\Lambda $, $g\mapsto \hat{g}^{\mathrm{T}}$ defines a fully
faithful contravariant functor from the category $\mathrm{Groups}_{+}^{\mathrm{T}}(\mathbb{Z}^{*},\mathbb{Q}^{*})$ to the category of groups with a Polish cover $G/N$, where morphisms $G/N\to G'/N'$ are continuous homomorphisms $G\to G'$ that map $N$ to $N'$.
\end{theorem}
\begin{proof}
We first show that the functor is faithful. Let $g\colon \Lambda'\to \Lambda$ be in  $\mathrm{Groups}_{+}^{\mathrm{T}}(\mathbb{Z}^{*},\mathbb{Q}^{*})$ and assume that $\bar{g}^{\mathrm{T}}\colon \mathbb{\hat{Q}}_{\Lambda }^{d}/\Lambda \rightarrow \mathbb{\hat{Q}}_{\Lambda'}^{d'}/\Lambda ^{\prime}$ is the zero homomorphism. Since $\hat{g}^{\mathrm{T}}$ is a continuous and $\Lambda'$ is countable, there exists $b\in\Lambda'$ so that the set
\[G(b):=\{x\in\mathbb{\hat{Q}}_{\Lambda }^{d} \colon \hat{g}^{\mathrm{T}}(x)=b\}\] 
is non-meager. By Pettis' Lemma  $G(b)-G(b)$ contains a clopen subgroup of $\mathbb{\hat{Q}}_{\Lambda }^{d}$. Thus, there is $m\in\omega$ so that 
  \[\mathrm{cl}(\Lambda^{(m)})\cap \hat{\mathbb{Z}}^d_{\boldsymbol{\Lambda}}=\mathrm{cl}(\Lambda^{(m)})\cap \hat{\Lambda}_{(0)}\subseteq G(b)-G(b).\]
 It follows that $g^{\mathrm{T}}(\Lambda^{(m)})=0$ and since $\Lambda^{(m)}$ is a rank $d$ subgroup of $\mathbb{Z}^d$, we have that $g^{\mathrm{T}}=0$. Hence $g=0$.
 
We now show that the functor is full.  Let $f\colon \mathbb{\hat{Q}}_{\Lambda }^{d} \rightarrow \mathbb{\hat{Q}}_{\Lambda'}^{d'}$ be a continuous homomorphism with $f(\Lambda)\subseteq \Lambda'$. 
In particular, $h:=f\upharpoonright\Lambda$ is a homomorphism from $\Lambda$ to $\Lambda'$. We set $g:=h^{\mathrm{T}}$ and we claim that $g$ is the desired $\mathrm{T}$-homomorphism. To this end, we only need to show that $g$ is a homomorphism from $\Lambda'$ to $\Lambda$. Notice that since $f\upharpoonright\hat{\Lambda}_{(0)}=\hat{\mathbb{Z}}^d_{\boldsymbol{\Lambda}}$ is continuous, there exists an increasing sequence $(m_k)_{k\in\omega}$ so that $h(\Lambda^{(m_k)})\subseteq (\Lambda')^{(k)}$. Hence,
\[g(\Lambda'_{(k)})=h^{\mathrm{T}}((\Lambda')^{(k)}_{\perp})\subseteq \Lambda^{(m_k)}_{\perp}= \Lambda_{(m_k)},\]
and therefore $g$ is a homomorphism from $\Lambda'$ to $\Lambda$.
\end{proof}

As a consequence we have the following corollary.

\begin{corollary}
\label{Corollary:trivial} Let $\Lambda$ be a subgroup of $\mathbb{Q}^d$ with the property that for all $\alpha\in\mathrm{Aut}(\Lambda)$ we have that $\alpha^{\mathrm{T}}\in\mathrm{Aut}(\Lambda)$. Then every definable automorphism of $\mathbb{\hat{Q}}
_{\Lambda }^{d}/\Lambda$ lifts to a topological group automorphism of $\mathbb{\hat{Q}}
_{\Lambda }^{d}$.
\end{corollary}

\section{The definable content of the Ext functor}\label{S:Ext}

Let $F$ and $B$ be two countable abelian groups. By an \emph{extension $\mathcal{E}$ of $B$ by $F$} we mean any short exact sequence 
\begin{center}
\begin{tikzcd}[ampersand replacement=\&]
  0 \arrow[r]  \& F \arrow[r] 
\& 
E \arrow[r] 
\& 
B \arrow[r] 
\& 0
\end{tikzcd}
\end{center}
of homomorphisms of abelian groups. We call $F$ the \emph{fiber} of the extension $\mathcal{E}$ and we identify it with its image in $E$. We also call $B$ the \emph{base} of the extension $\mathcal{E}$. The extensions $\mathcal{E}$ and $\mathcal{E}'$ are isomorphic if there is a group isomorphism $E\to E'$ which makes the following diagram commute.
\begin{center}
\begin{tikzcd}[ampersand replacement=\&]
 \& 
\& E \arrow[dr] \arrow[dd, , "\simeq ", dotted] 
\& 
\& \\
0 \arrow[r]  \& F \arrow[ur] \arrow[dr]
\& 
\& 
B \arrow[r] 
\& 0 \\
\& 
\& E' \arrow[ur]  
\& 
\& \end{tikzcd}
\end{center}

We denote by  $\mathrm{Ext}(B,F)$ the collection of all isomorphism classes of extensions of $B$ by $F$. The fact that $\mathrm{Ext}(B,F)$ admits an abelian group structure goes back to Schreier and Baer \cite{Schreier1926,Baer1934}. The induced bifunctor $\mathrm{Ext}$  is the first derived functor of the bifunctor $\mathrm{Hom}$,  and  it has played a fundamental role in the development of homological algebra since the seminal  work \cite{eilenberg_group_1942}  of Eilenberg and MacLane, where it was shown that classical invariants from algebraic topology can  be defined, and related, in terms of $\mathrm{Ext}$. 

While $\mathrm{Ext}(B,F)$ carries a natural group topology, in standard treatments of  $\mathrm{Ext}(B,F)$ this topology is ignored, and $\mathrm{Ext}(B,F)$ is  treated as a discrete group. One of the reasons may be that this natural topology is not always Hausdorff. This fact was already understood by 
 Eilenberg and MacLane in \cite{eilenberg_group_1942}, where  $\mathrm{Ext}(B,F)$  was termed a ``generalized topological group" whenever the natural topology was not Hausdorff. In this section we show how to view $\mathrm{Ext}$  as a bifunctor from the category of countable abelian groups to the category of \emph{abelian groups with a Polish cover}. The resulting bifunctor \emph{definable} $\mathrm{Ext}$ is a much stronger invariant for classifying algebraic systems. For example, for every $d,d'\in\mathbb{N}$ we have that 
  we have that  $\mathrm{Ext}(\mathbb{Q}^d,\mathbb{Z})$ and $\mathrm{Ext}(\mathbb{Q}^{d'},\mathbb{Z})$ are isomorphic but not definably isomorphic. More generally, the following is the main theorem of this section, whose  comparison with  Corollaries \ref{Corollary:iso-classification-Ext}, \ref{Corollary:smooth-classification-Ext}, and  \ref{Corollary:nonisomorphic-Ext}, illustrates how much coarser the usual $\mathrm{Ext}(\cdot,\mathbb{Z})$ functor is as a group invariant.

\begin{theorem}\label{Theorem:ExtMain}
The definable  $\mathrm{Ext}(\cdot,\mathbb{Z})$ functor is fully faithful when restricted to the category of finite rank torsion-free abelian groups with no free summands.
\end{theorem}

Let $F$ and $B$ be countable abelian groups. One may describe any extension $E$ of $B$ by $F$ directly using $F$ and $B$ alone. Indeed, if $s\colon B\to E$ is a section of the epimorphism $E\to B$ the multiplication table of $E$ is entirely determined by the unique function $c_s\colon B\times B\to F$ which satisfies
\[s(x)+s(y)=s(x+y)+c_s(x,y).\]
If $t\colon B\to E$ is a another section of the same epimorphism, then the 
function $h\colon B\to F$ with $h(x)=s(x)-t(x)$ is a witness to the fact that $c_s$ and $c_t$ represent the same extension, which is  entirely defined in terms of $F$and $B$. These observations give rise to the following concrete description of $\mathrm{Ext}(B,F)$ and related group $\mathrm{PExt}(B,F)$; see  \cite{eilenberg_group_1942,fuchs_infinite_1970,schochet_pext_2003}.\begin{itemize}
\item $\mathrm{Z}(B,F)$ is the closed subgroup of the abelian Polish group $F^{B\times B}$, consisting of all cocycles on $B$ with coefficients in $F$.  The topology on $F^{B\times B}$ is the product topology and $F$ is the discrete group. By a \emph{cocycle on $B$ with coefficients in $F$} we mean any function $c:B\times B\rightarrow F$ so  that  for all $x,y,z \in B$ we have: 
\begin{enumerate}
\item $c( x,0) =c( 0,y) =0$;
\item $c( x,y) +c( x+y,z) =c( x,y+z)+c( y,z)$;
\item $c( x,y) =c( y,x)$, for all $x,y \in B$.
\end{enumerate}
\item  $C(B,F)$ is the abelian Polish group $F^{B}$, where $F$ is endowed with the discrete topology and $F^{B}$ with the product topology. We have a continuous group homomorphism  $\delta\colon C(B,F) \to \mathrm{Z}(B,F)$, given by:
\[\delta(h)(x,y):= h(x)+h(y)-h(x+y).\]
\item $\mathrm{B}(B,F)$ is the Polishable Borel subgroup $\delta(C(B,F))$ of  $\mathrm{Z}(B,F)$; see Lemma \ref{Lemma:Borel-inverse}. Explicitly:
\[\mathrm{B}(B,F):=\big\{ c\in \mathrm{Z}(B,F) \colon c(x,y)=h(x)+h(y)-h(x+y),  \text{ for some } h\in C(B,F)\big\}.\]
A \emph{coboundary on $B$ with coefficients
in $F$} is a cocycle $c:B\times B\rightarrow F$, which lies in $\mathrm{B}(B,F)$.
\item $\mathrm{B}_{\mathrm{w}}(B,F)$ is the closed subgroup of  $\mathrm{Z}(B,F)$ consisting of all  $c\in 
\mathrm{Z}(C,A)$ with  $(c\upharpoonright S\times S)\in \mathrm{B}(S,F)$, for every finite subgroup $S$ of $B$.  A \emph{weak coboundary on $B$ with coefficients in $F$} is any element of  $\mathrm{B}_{\mathrm{w}}(B,F)$.
\end{itemize}

\begin{definition}
Let $F$ and $B$ be countable abelian groups. 
\begin{enumerate}
\item Let  $\mathrm{Ext}(B,F):= \mathrm{Z}(B,F) /\mathrm{B}(B,F)$ be  the group with a Polish cover  of all \emph{extensions of $B$ by $F$};
\item Let $\mathrm{PExt}(B,F):= \mathrm{B}_{\mathrm{w}}(B,F) /\mathrm{B}(B,F)$ be the group with a Polish cover of all \emph{pure extensions of $B$ by $F$};
\item Let $\mathrm{Ext}_{\mathrm{w}}(B,F):= \mathrm{Z}(B,F) /\mathrm{B}_{\mathrm{w}}(B,F)$ be the Polish group of all \emph{weak extension classes of $B$ by $F$}.
\end{enumerate}
\end{definition}
To justify the terminology in point (2) above,  recall that a subgroup $F$ of $E$ is called \emph{pure} if for all $\ell\in\mathbb{Z}$ we have that $F\cap (\ell E)=\ell F$. A well known property about extensions $0\to F\to E \to S\to 0$ of finite abelian groups $S$ is that they are trivial (that is, they correspond to coboundaries), if and only if $F$ is a pure subgroup of $E$; see \cite[Chapter V]{fuchs_infinite_1970}. Since $\mathrm{Ext}_{\mathrm{w}}(B,F)$ is a Polish group, the non-trivial definable content of $\mathrm{Ext}(B,F)$ concentrates in $\mathrm{PExt}(B,F)$. In fact, notice that under the assumption of Theorem \ref{Theorem:ExtMain}  that $B$ is  torsion-free, we have that every extension of $B$ by $F$ is pure. That is,  $\mathrm{Ext}(B,F)=\mathrm{PExt}(B,F).$ We can therefore concentrate on analyzing the definable content of $\mathrm{PExt}(B,F)$.

 It is not difficult
to see that $\mathrm{Ext}( -,-) $ and $\mathrm{PExt}(
-,-) $ are additive bifunctors from the category of countable abelian
groups to the category of abelian groups with a Polish cover,
which are contravariant in the first coordinate and covariant in the second
coordinate. Similarly, $\mathrm{Ext}_{\mathrm{w}}( -,-) $ is a
bifunctor to the category of Polish groups which is contravariant in the first
argument and covariant in the second argument. The next  lemma  expresses the groups with a Polish cover associated to  $\mathrm{PExt}(
-,-)$ in terms of the bifunctor $\mathrm{Hom}(
-,-)$.

\begin{lemma}
\label{Lemma:EML} Let $F$, $B$, and $G$ be countable abelian groups. Given any pure extension $E$ of $B$ by $F$:
\begin{equation*}
0\longrightarrow F\overset{f}{\longrightarrow }E\longrightarrow
B\longrightarrow 0,
\end{equation*}
with  $\mathrm{Pext}(E,G) =0$, the group with a Polish cover  $\mathrm{PExt}(
B,G)$ is definably isomorphic to
\begin{equation*}
\frac{\mathrm{Hom}(F,G) }{f^{\ast }( \mathrm{Hom}(
F,G)) }\text{,}
\end{equation*}
where $f^{*}:\mathrm{Hom}(E,G)\to\mathrm{Hom}(F,G)$ is the image of $f\colon F\to E$ under the contravariant functor $\mathrm{Hom}(-,G)$.
\end{lemma}
\begin{proof}
Let  $c_{E}\in \mathrm{B}_{\mathrm{w}}(B,F)$ be a weak coboundary representing the pure extension $E$. Notice that  $c_{E}$ induces a continuous group homomorphism $
\left( c_{E}\right) ^{\ast }:\mathrm{Hom}(F,G) \rightarrow 
\mathrm{B}_{\mathrm{w}}(B,G) $ defined by
\begin{equation*}
( c_{E}) ^{\ast }( \eta ) =\eta \circ c_{E}\text{.}
\end{equation*}
This induces a the group homomorphism $E^{\ast }:\mathrm{Hom}( F,G)
\rightarrow \mathrm{PExt}( B,G) $ given by%
\begin{equation*}
\left( E^{\ast }\right)( \eta ) =\eta \circ c_{E}+\mathrm{B}%
(B,F).
\end{equation*}
As shown in \cite[Theorem 53.7]{fuchs_infinite_1970}, $E^{\ast }$ defines the connecting morphism in the exact sequence%
\begin{equation*}
0\rightarrow \mathrm{Hom}\left(B,G\right)\rightarrow \mathrm{Hom}\left(
E,G\right) \rightarrow \mathrm{Hom}\left( F,G\right) \overset{E^{\ast }}{\rightarrow }\mathrm{PExt}\left( B,G\right) \rightarrow \mathrm{PExt}\left(
E,G\right) \rightarrow \mathrm{PExt}\left( F,G\right) \rightarrow 0.
\end{equation*}
The rest follows by the assumption that $\mathrm{PExt}(E,G) =0$.
\end{proof}

By an \emph{inductive sequence of countable abelian groups} we mean a collection  $(B_{(n)},\eta _{\left( n+1,n\right)
})_{n\in\omega}$ of homomorphisms  $\eta _{\left( n+1,n\right)
}:B_{(n)}\rightarrow B_{(n+1)}$ between countable abelian groups. As in the case of towers in Section \ref{Section:towers}, we also consider the induced maps $\eta _{\left( m,n\right)
}:B_{(n)}\rightarrow B_{(m)}$ for all $n\leq m$, where $\eta _{\left( n,n\right)}=\mathrm{id}_{B_{(n)}}$, and $\eta_{\left( m,n\right)}:=\eta_{\left( m,m-1\right)} \circ \cdots\circ \eta_{\left( n+1,n\right)}$ if $m>n$. Notice that cotowers are special types of inductive sequences. \emph{Morphisms of inductive sequences} are defined similarly to the cotower maps; see Section   \ref{Section:locally-profinite}. Notice that when $B,G$ are countable abelian groups then $\mathrm{Hom}(B,G) $ is a Polish abelian group. 
Moreover, since $\mathrm{Hom}(
-,-) $ is contravariant in the first argument, any inductive sequence 
\[B_0\rightarrow B_1\rightarrow B_2 \rightarrow\cdots\]
 of countable abelian groups, induces a tower of  Polish abelian groups:  
\[\mathrm{Hom}(B_{0},G) \leftarrow \mathrm{Hom}(B_{1},G) \leftarrow \mathrm{Hom}(B_{2},G) \leftarrow\cdots\]
The following theorem is essentially due to  C.U. Jensen; see \cite%
[Theorem 6.1]{schochet_pext_2003} and \cite{jensen_foncteurs_1972}. 

\begin{theorem}[Jensen]
\label{Theorem:Jensen} Let $G$ be a countable abelian group and let $\left(B_{(n)},\eta _{( n+1,n)}\right)$ be an inductive sequence of countable abelian groups with $\mathrm{PExt}(B_{(n)},G) =0$ for all $n\in\omega$. Then the groups with a Polish cover $\underleftarrow{\mathrm{lim}}^{1}( \mathrm{Hom}(B_{(n)},G)) $ and $\mathrm{PExt}(\underrightarrow{\mathrm{colim}}_{n}B_{(n)},G)$ are naturally definably isomorphic.
\end{theorem}
\begin{proof}
 Set $B:=\underrightarrow{\mathrm{colim}}%
_{n}B_{(n)}$ and let $\oplus_{n}B_{(n)}$ denote the direct sum of the groups $B_{(n)}$, indexed by $n\in\omega$:
\begin{equation*}
\bigoplus_{n \in \omega }B_{(n)}\text{.}
\end{equation*}
For every $n \in \omega $ and $b\in B_{(n)}$, let $be_{n }$ 
be the element of $\oplus_{n}B_{(n)}$ with all the coordinates equal to $0$ apart
from the $n $th coordinate, which is equal to $b$.
Notice that the homomorphism $\delta\colon \oplus_{n}B_{(n)}\to\oplus_{n}B_{(n)}$ which is defined on the generators $be_{n }$ of $\oplus_{n}B_{(n)}$ by $\delta(be_{n })=be_{n }-\eta_{(n+1,n)}(b)e_{n+1}$ is injective. Indeed, every  $b\in\oplus_{n}B_{(n)}$ is of the form  
\begin{equation}\label{eq:2}
b=b_0e_{n_0}+b_1e_{n_1}+\cdots+ b_ke_{n_k},  \text{ with } n_0<n_1<\cdots<n_k.
\end{equation}
But then, $\delta(b)=0$ implies that $b_0=0$; and inductively for each $i<k$,  $\delta(b)=0$, $b_0=0,\ldots, b_i=0$, imply $b_{i+1}=0$. It is also immediate that the cokernel of $\oplus_{n}B_{(n)}$ is $B$. Hence $\delta$ defines an extension of $B$ by $\oplus_{n}B_{(n)}$:
\begin{equation*}
0\longrightarrow \oplus_{n}B_{(n)} \overset{\delta}{\longrightarrow }\oplus_{n}B_{(n)} \longrightarrow B\longrightarrow 0\text{.}
\end{equation*}
This extension is also pure. Indeed,  using (\ref{eq:2}) above and by inducting on $i$  ($0\leq i\leq l$), one shows that if $\delta(b)$ is divisible by $\ell\in\mathbb{Z}$ in  $\oplus_{n}B_{(n)}$, then  each $b_i$  is divisible by $\ell$ in $B_{(n)}$, which implies that $b$ is divisible by $\ell$ in $\oplus_{n}B_{(n)}$.  Moreover, by standard computations (see, \cite[Proposition 3.3.4]{weibel_introduction_1994}) and since $\mathrm{PExt}(B_{(n)},G) =0$, for all $n$,  we have
 \begin{equation*}
\mathrm{PExt}\left( \oplus _{n}B_{(n)},G\right) \cong \prod_{n\in \omega }
\mathrm{PExt}( B_{(n)},G)=0. 
\end{equation*}
 By Lemma \ref{Lemma:EML} we
have that  $\mathrm{PExt}(B,G)$ is definably isomorphic to 
\begin{equation*}
\frac{\mathrm{Hom}(\oplus _{n}B_{(n)},G) }{\delta^{\ast }( 
\mathrm{Hom}( \oplus _{n}B_{(n)},G)) }\text{.}
\end{equation*}
We will show that the latter is definable isomorphic to  $\underleftarrow{\mathrm{lim}}^{1}\,\boldsymbol{A}:= \mathrm{Z}(\boldsymbol{A})/ \mathrm{B}(\boldsymbol{A})$, where
$\boldsymbol{A}$ is the tower of Polish groups  $A^{(n)}:=\mathrm{Hom}(B_{(n)},G)$ and bonding maps $p^{(n,n+1)}:=\mathrm{Hom}(\eta_{(n+1,n)},G)$. Consider the  map
\[F\colon \mathrm{Z}(\boldsymbol{A})\to \mathrm{Hom}(\oplus _{n}B_{(n)},G),\]
which sends $(a_{m_0,m_1})_{m_0\leq m_1}$ to the homomorphism $ ce_n\mapsto a_{n,n+1}(c)$. Clearly $F$ is a continuous and surjective homomorphism with $F(\mathrm{B}(\boldsymbol{A}))=\delta^{\ast }( 
\mathrm{Hom}( \oplus _{n}B_{(n)},G))$. The fact that it is injective follows from  Remark \ref{Remark:limConsecutive}. 
\end{proof}

Notice that every countable abelian group $B$ can be written as an increasing union of finitely-generated abelian groups $B_{(n)}$. By the classification theorem of finitely-generated abelian groups we have that $\mathrm{PExt}(B_{(n)},F)$ vanishes for each $n$. Hence, the above theorem can be used to calculate $\mathrm{PExt}(B,F)$ for every pair of countable abelian groups $B$ and  $F$. In the rest of this section we calculate  $\mathrm{Ext}( \Lambda ,\mathbb{Z}) $
for each finite-rank torsion-free abelian group $\Lambda $. 

Let  $\Lambda $ be a  finite-rank torsion-free abelian group. Notice that $\Lambda$ can be written as a direct sum $\Lambda _{0}\oplus \Lambda _{1}$ where $\Lambda _{0}$ has no free summand and $\Lambda _{1}$ is finitely-generated free abelian group $\mathbb{Z}^m$. We have that 
\begin{equation*}
\mathrm{Ext}( \Lambda ,\mathbb{Z}) =\mathrm{Ext}( \Lambda
_{0},\mathbb{Z}) \oplus \mathrm{Ext}( \Lambda _{1},\mathbb{Z}) =\mathrm{Ext}( \Lambda _{0},\mathbb{Z}) \text{.}
\end{equation*}
 Thus, without any loss of generality, we may  restrict our attention to the case where  $\Lambda$ has no free summand. Recall from Section \ref{Section:locally-profinite} that we have a fully faithful contravariant functor from the category  $\mathrm{Groups}_{+}(\mathbb{Z}^{*},\mathbb{Q}^{*})$ of finite rank torsion-free abelian groups with no free summands to the category of groups with a Polish cover which assigns to each group $\Lambda$ in $\mathrm{Groups}_{+}(\mathbb{Z}^{*},\mathbb{Q}^{*})$ the quotient $\mathbb{\hat{Z}}_{\boldsymbol{\Lambda}}^{d}/\mathbb{Z}^{d}$ of a profinite  completion $\mathbb{\hat{Z}}_{\boldsymbol{\Lambda}}^{d}$ of  $\mathbb{Z}^d$ by  $\mathbb{Z}^{d}$; see Theorem \ref{Theorem:MainProfinite}.
 The following theorem, without the definability claim, is proved for the special case of rank $1$ abelian groups in \cite[Appendix B]{eilenberg_group_1942}

 \begin{theorem}\label{Theorem:TF-Ext}
The functors implementing $\Lambda\mapsto \mathbb{\hat{Z}}_{\boldsymbol{\Lambda}}^{d}/\mathbb{Z}^{d}$ and $\Lambda\mapsto \mathrm{Ext}(\Lambda,\mathbb{Z})$ from the category 
$\mathrm{Groups}_{+}(\mathbb{Z}^{*},\mathbb{Q}^{*})$ to the category of groups with a Polish cover are naturally isomorphic. 
 \end{theorem}
\begin{proof}
Recall  that the functor implementing $\Lambda\mapsto \mathbb{\hat{Z}}_{\boldsymbol{\Lambda}}^{d}/\mathbb{Z}^{d}$ in  Section \ref{Section:locally-profinite}   is the composition of the functors $(\lim)^{-1}$, $\mathrm{Adj}$, and $\underleftarrow{\mathrm{lim}}^{1}$.  By Lemma \ref{Lemma:dual-perp}(1) we have that the functor  implementing $\Lambda\mapsto \mathbb{\hat{Z}}_{\boldsymbol{\Lambda}}^{d}/\mathbb{Z}^{d}$  is naturally isomorphic to the functor implementing   $\Lambda\mapsto \underleftarrow{\mathrm{lim}}^{1}( \mathrm{Hom}( \Lambda_{(n)} ,G) )$ in Theorem \ref{Theorem:Jensen}, where $(\Lambda_{(n)})=(\lim)^{-1}(\Lambda)$. But Theorem \ref{Theorem:Jensen} exhibits a definable isomorphism from $\underleftarrow{\mathrm{lim}}^{1} ( \mathrm{Hom}( \Lambda_{(n)} ,G) )$ to  $\mathrm{PExt}(\Lambda,G)$. By a routine diagram chasing of the proof of Theorem \ref{Theorem:Jensen} we see that the latter definable isomorphisms are components of a natural transformation from the functor implementing  $\Lambda\mapsto  \underleftarrow{\mathrm{lim}}^{1} ( \mathrm{Hom}( \Lambda_{(n)} ,G) )$  to the functor implementing $\Lambda\mapsto \mathrm{PExt}(\Lambda,G)$.
Finally, notice that since each $\Lambda$ is torsion-free we have that $\mathrm{Ext}(\Lambda,\mathbb{Z})=\mathrm{PExt}(\Lambda,\mathbb{Z})$.  
\end{proof} 

We have the following immediate corollary.

\begin{corollary}
\label{Corollary:TF-Ext}The  functor $\Lambda \mapsto 
\mathrm{Ext}( \Lambda ,\mathbb{Z}) $  from the category of 
finite-rank torsion-free abelian groups with no free summands to the category of groups with
a Polish cover is fully faithful.
In particular, two such groups $\Lambda $
and $\Lambda ^{\prime }$ are isomorphic, if and only if $\mathrm{Ext}(
\Lambda ,\mathbb{Z}) $ and $\mathrm{Ext}( \Lambda ^{\prime },%
\mathbb{Z}) $ are definably isomorphic.
\end{corollary}

The following corollaries demonstrate how forgetting the definable content of $\mathrm{Ext}$ results in much weaker invariants. We denote by $\mathcal{P}$ the set of prime
numbers. The following theorem is an adaptation of \cite[Theorem 2]{warfield_extensions_1972}.

\begin{theorem}
\label{Corollary:iso-classification-Ext}Let $\Lambda $ be a rank $d$
torsion-free abelian group with no free summands. For $p\in \mathcal{P}$ let $n_{p}\in \left\{
0,1,\ldots ,d\right\} $ be such that the \emph{$p$-corank of $\Lambda$}, $\left[ \Lambda :p\Lambda \right]$, is $p^{n_{p}}$. Then $\mathrm{Ext}( \Lambda ,\mathbb{Z}) $ is
isomorphic as a discrete group to%
\begin{equation*}
\mathbb{Q}^{\left( 2^{\aleph _{0}}\right) }\oplus \bigoplus_{p\in \mathcal{P}%
}\mathbb{Z}( p^{\infty }) ^{n_{p}}\text{.}
\end{equation*}%
In particular, if $\Lambda $ and $\Lambda ^{\prime }$ are finite rank
torsion-free abelian groups with no free summands, then $\mathrm{Ext}( \Lambda ,\mathbb{Z}) $ and $\mathrm{Ext}( \Lambda ^{\prime },\mathbb{Z}) $ are
isomorphic as discrete groups if and only if $\left[ \Lambda :p\Lambda %
\right] =\left[ \Lambda ^{\prime }:p\Lambda ^{\prime }\right] $ for every $%
p\in \mathcal{P}$.
\end{theorem}

\begin{corollary}
\label{Corollary:smooth-classification-Ext}The relation $E$ for finite-rank
torsion-free abelian groups defined by $\left( \Lambda ,\Lambda ^{\prime
}\right) \in E$ if and only if $\mathrm{Ext}( \Lambda ,\mathbb{Z}) $ and $\mathrm{Ext}( \Lambda ^{\prime },\mathbb{Z}) $ are
isomorphic as discrete groups, is smooth.
\end{corollary}

Now adopt the following notations. For every
sequence $\boldsymbol{m}=\left( m_{p}\right) _{p\in \mathcal{P}}\in \mathbb{N%
}^{\mathcal{P}}$, where $\mathbb{N}$ is the set of strictly positive
integers, define $\mathbb{Z}[\frac{1}{\mathcal{P}^{\boldsymbol{m}}}]$ to be
the set of rational numbers of the form $a/b$ where $a\in \mathbb{Z}$, $b\in 
\mathbb{N}$, and for every $p\in \mathcal{P}$ and $k\in \mathbb{N}$, if $%
p^{k}$ divides $b$ then $k\leq m_{p}$. Write $\boldsymbol{m}=^{\ast }%
\boldsymbol{n}$ if and only if $\{ p\in \mathcal{P}%
:m_{p}\neq m_{p}^{\prime }\} $ is finite. From Theorem \ref{Theorem:TF-Ext} we
obtain the following.

\begin{corollary}
\label{Corollary:nonisomorphic-Ext}Fix $d\geq 1$. For every $%
\boldsymbol{m},\boldsymbol{n}\in \mathbb{N}^{\mathcal{P}}$, $%
\mathrm{Ext}(\mathbb{Z}[\frac{1}{\mathcal{P}^{\boldsymbol{m}}}]^{d},\mathbb{Z%
})$ and $\mathrm{Ext}(\mathbb{Z}[\frac{1}{\mathcal{P}^{\boldsymbol{n}}}]^{d},%
\mathbb{Z})$ are isomorphic as discrete groups, and are Borel isomorphic if
and only if $\boldsymbol{m}=^{\ast }\boldsymbol{n}$.
In particular, the collection
\begin{equation*}
\left\{ \mathrm{Ext}(\mathbb{Z}[\frac{1}{\mathcal{P}^{\boldsymbol{m}}}]^{d},%
\mathbb{Z}):\boldsymbol{m}\in \mathbb{N}^{\mathcal{P}}\right\}
\end{equation*}%
contains a continuum of groups with a Polish cover that are pairwise
isomorphic as discrete groups but not definably isomorphic.
\end{corollary}

\begin{proof}
Fix $\boldsymbol{m}\in \mathbb{N}^{\mathcal{P}}$. For $\boldsymbol{m},%
\boldsymbol{n}\in \mathbb{N}^{\mathcal{P}}$ it follows from Baer's
classification of countable rank-$1$ torsion-free abelian groups \cite[%
Chapter 12, Theorem 1.1]{fuchs_abelian_2015} that $\mathbb{Z}[\frac{1}{%
\mathcal{P}^{\boldsymbol{m}}}]$ and $\mathbb{Z}[\frac{1}{\mathcal{P}^{%
\boldsymbol{n}}}]$ are isomorphic if and only if $\boldsymbol{m}%
=^{\ast }\boldsymbol{n}$. It follows from this and \cite[Chapter
12, Theorem 3.5]{fuchs_abelian_2015} that $\mathbb{Z}[\frac{1}{\mathcal{P}^{%
\boldsymbol{m}}}]^{d}$ and $\mathbb{Z}[\frac{1}{\mathcal{P}^{\boldsymbol{n}}}]^{d}$ are isomorphic if and only if $\boldsymbol{m}=^{\ast }%
\boldsymbol{n}$. Together with Corollary \ref{Corollary:CombinatorialDescriptionOfDEfIso}, this concludes the proof of the first assertion.
Observe now that for every $p\in \mathcal{P}$, the $p$-corank of $\mathbb{Z}[\frac{1}{%
\mathcal{P}^{\boldsymbol{m}}}]^{d}$ is equal to $d$. Thus the second
assertion is an immediate consequence of Theorem \ref{Corollary:iso-classification-Ext}.
\end{proof}

\begin{remark}
For Baer invariants $\boldsymbol{m}$ that have a prime with infinite multiplicity, it is no longer true that the groups $\mathrm{Ext}(\mathbb{Z}[\frac{1}{\mathcal{P}^{\boldsymbol{m}}}]^{d},\mathbb{Z})$ are all isomorphic as discrete groups.
\end{remark}

\section{Actions by definable automorphisms and Borel reduction complexity}

\label{S:ActionsCocycleRig} Let $\mathcal{G}$ denote the group with a Polish
cover $0\to N \to G \to G/N\to 0$. The group $\mathrm{Aut}(\mathcal{G})$ of
all \emph{definable automorphisms} of $\mathcal{G}$ is the automorphism
group of $\mathcal{G}$ in the category of groups with a Polish cover.
Explicitly, $\mathrm{Aut}(\mathcal{G})$ consists of those group
automorphisms $\varphi\colon G/N\to G/N$ which admit a Borel map $\hat{%
\varphi}:G\to G$ as a lift. This defines a canonical action $\mathrm{Aut}( 
\mathcal{G}) \curvearrowright G/N$ of $\mathrm{Aut}(\mathcal{G})$ on the
quotient $G/N$.

\begin{definition}
\label{Definition:Borel action} By a \emph{definable action} of a discrete
group $\Gamma$ on a group with a Polish cover $\mathcal{G}$ we mean a group
homomorphism $\varphi :\Gamma \rightarrow \mathrm{\mathrm{Aut}}(\mathcal{G})$%
. The assignment $\gamma \mapsto \varphi _{\gamma }$ induces an action $%
(\varphi_{\gamma})_{\gamma} : \Gamma \curvearrowright G/N$ on the quotient $%
G/N$. If $\Gamma \curvearrowright G/N$ is definable, we let $\mathcal{R}%
(\Gamma \curvearrowright G/N)$ be the equivalence relation on $G$ so that
for $x,y\in G$ we have: 
\begin{equation*}
x,y\in \mathcal{R}(\Gamma \curvearrowright G/N) \iff \exists \gamma\in
\Gamma \; \big(\varphi_{\gamma}(Nx)=Ny\big). 
\end{equation*}
\end{definition}

Notice that when both $\Gamma $ and $N$ above are countable, then $\mathcal{R%
}(\Gamma \curvearrowright G/N)$ is a countable Borel equivalence relation.
Hence a natural question is that of where these relations sit within the
Borel reduction hierarchy. The main goal of this section is to address this
question for certain definable actions of the form: 
\begin{equation}
\mathcal{R}(\mathrm{Aut}(\mathrm{Ext}(\Lambda ,\mathbb{Z}))\curvearrowright 
\mathrm{Ext}(\Lambda ,\mathbb{Z})),\quad \mathcal{R}(\mathrm{Aut}(\mathbb{%
\hat{Z}}_{\boldsymbol{\Lambda }}^{d}/\mathbb{Z}^{d})\curvearrowright \mathbb{%
\hat{Z}}_{\boldsymbol{\Lambda }}^{d}/\mathbb{Z}^{d}),\quad \mathcal{R}(%
\mathrm{Aut}(\mathbb{\hat{Q}}_{\Lambda }^{d}/\Lambda )\curvearrowright 
\mathbb{\hat{Q}}_{\Lambda }^{d}/\Lambda ),  \label{Eq:3relations}
\end{equation}%
for the groups with a Polish cover $\mathrm{Ext}(\Lambda ,\mathbb{Z})$, $%
\mathbb{\hat{Z}}_{\boldsymbol{\Lambda }}^{d}/\mathbb{Z}^{d},$ $\mathbb{\hat{Q%
}}_{\Lambda }^{d}/\Lambda $, as in Sections \ref{Section:locally-profinite}
and \ref{S:Ext}. Notice that $\mathrm{Aut}(\mathrm{Ext}(\Lambda ,\mathbb{Z}%
))\cong \mathrm{Aut}\left( \Lambda \right) $ by Theorem \ref{Theorem:ExtMain}%
, and hence it is countable being $\Lambda $ finite-rank.\ Likewise, we have
that $\mathrm{\mathrm{Aut}}(\mathbb{\hat{Q}}_{\Lambda }^{d}/\Lambda )\cong 
\mathrm{\mathrm{Aut}}\left( \Lambda \right) $ by Theorem \ref%
{Theorem:trivial}, and $\mathrm{\mathrm{Aut}}(\mathbb{\hat{Z}}_{\boldsymbol{%
\Lambda }}^{d}/\mathbb{Z}^{d})\cong \mathrm{\mathrm{Aut}}(\mathbb{\hat{Q}}%
_{\Lambda }^{d}/\Lambda )$ where $\Lambda =\underrightarrow{\mathrm{lim}}(%
\mathrm{Adj}(\boldsymbol{\Lambda }))$ by Theorem \ref{Theorem:TF-Ext}.

The first of these problems is particularly natural and it can be thought of
as the \textquotedblleft base-free" analogue of the extension problem we
show in Section \ref{S:Ext}. Indeed, for any countable abelian groups $B,F$
the Polish space $\mathrm{Z}(B,F)$ parametrizes all ways of describing an
extension $E$ of the base $B$ by $F$, and the coset equivalence relation
induced by $\mathrm{B}(B,F)$ was used to identify which such descriptions
were isomorphic by an isomorphism which preserves $B$ up to identity. In
contrast, $\mathcal{R}(\mathrm{Aut}(\mathrm{Ext}(B,F))\curvearrowright 
\mathrm{Ext}(B,F))$, represents the problem of classifying all ways of
describing extensions of $B$ by $F$ up to \emph{base-free isomorphisms}.
That is, isomorphisms which preserve the base $B$ only up to an automorphism
of $B$.

In the context of Theorem \ref{Theorem:TF-Ext}, and since $\mathbb{\hat{Z}}_{%
\boldsymbol{\Lambda} }^{d}/\mathbb{Z}^d$ and $\mathbb{\hat{Q}}_{\Lambda
}^{d}/\Lambda$ are definably isomorphic when $\Lambda=\underrightarrow{%
\mathrm{lim}}(\mathrm{Adj}(\boldsymbol{\Lambda}))$, the next lemma implies
that the above three equivalence relations are of the exact same Borel
reduction complexity when $\Lambda$ is fixed.

\begin{definition}
Two definable actions $(\varphi_{\gamma})_{\gamma} : \Gamma \curvearrowright
G/N$ and $(\varphi^{\prime }_{\gamma^{\prime }})_{\gamma^{\prime }} :
\Gamma^{\prime }\curvearrowright G^{\prime }/N^{\prime }$ are \emph{%
definably isomorphic} if there is a group isomorphism $\alpha :\Gamma
\rightarrow \Gamma ^{\prime }$ and a definable isomorphism $\psi
:G/N\rightarrow G^{\prime }/N^{\prime }$ so that for all $\gamma \in \Gamma $
we have: 
\begin{equation*}
\psi \circ \varphi _{\gamma }=\varphi _{\alpha \left( \gamma \right)
}^{\prime }\circ \psi.
\end{equation*}
\end{definition}

The following lemma implies that isomorphic definable actions generate
equivalence relations of equal complexity in the Borel reduction hierarchy.
The proof follows directly from the fact that, for countable Borel
equivalence relations, being classwise Borel isomorphic is the same as being
Borel bireducible; see \cite[Theorem 2.5]{motto_ros_complexity_2012}.

\begin{lemma}
\label{Lemma:BorelClasswiseIsomorphic} Let $\Gamma \curvearrowright G/N$ and 
$\Gamma^{\prime }\curvearrowright G^{\prime }/N^{\prime }$ be definably
isomorphic actions on the groups with a Polish cover $\mathcal{G}=(N,G,G/N)$
and $\mathcal{G}^{\prime }=(N^{\prime },G^{\prime },G^{\prime }/N^{\prime })$%
, where $\Gamma,\Gamma^{\prime },N,N^{\prime }$ are countable. Then, $%
\mathcal{R}(\Gamma \curvearrowright G/N)$ and $\mathcal{R}(\Gamma^{\prime
}\curvearrowright G^{\prime }/N^{\prime })$ are classwise Borel isomorphic.
In particular, if $G/N$ is definably isomorphic to $G^{\prime }/N^{\prime }$%
, then $\mathcal{R}(\mathrm{Aut}(\mathcal{G}) \curvearrowright G/N)$ and $%
\mathcal{R}(\mathrm{Aut}(\mathcal{G})^{\prime }\curvearrowright G^{\prime
}/N^{\prime })$ are classwise Borel isomorphic.
\end{lemma}

Many results in the complexity theory of Borel reductions between
equivalence relations $R$ depend on expressing the pertinent equivalence
relation as an \emph{orbit equivalence relation} of a continuous (or measure
preserving) action of a Polish group on the underlying space of $R$. In
general, there is no natural continuous or (Haar) measure preserving action
on $\mathrm{Z}(\Lambda ,\mathbb{Z}),\hat{\mathbb{Z}}_{\boldsymbol{\Lambda }%
}^{d}$, or $\hat{\mathbb{Q}}_{\Lambda }^{d}$ which induces the equivalence
relations from (\ref{Eq:3relations}) above. However, when $\Lambda $ has the
additional property that $\alpha \in \mathrm{Aut}(\Lambda )\iff \alpha ^{%
\mathrm{T}}\in \mathrm{Aut}(\Lambda )$ for every $\alpha \in \mathrm{GL}_{d}(%
\mathbb{Q})$, then Corollary \ref{Corollary:trivial} implies that $\mathcal{R%
}(\mathrm{Aut}(\mathbb{\hat{Q}}_{\Lambda }^{d}/\Lambda )\curvearrowright 
\mathbb{\hat{Q}}_{\Lambda }^{d}/\Lambda )$ is simply the orbit equivalence
relation 
\begin{equation*}
\mathcal{R}(\mathrm{Aut}\left( \Lambda \right) \ltimes \Lambda
\curvearrowright \mathbb{\hat{Q}}_{\Lambda }^{d})
\end{equation*}%
of a continuous action of the countable group $\mathrm{Aut}\left( \Lambda
\right) \ltimes \Lambda $ on the locally profinite space $\mathbb{\hat{Q}}%
_{\Lambda }^{d}$. In particular, if we set $\Lambda :=\mathbb{Z}%
[1/p]^{d}\leq \mathbb{Q}^{d}$ to be as in Example \ref{Example:p-adic}, we
have that 
\begin{equation*}
\mathcal{R}(\mathrm{Aut}(\mathbb{Q}_{p}^{d}/\mathbb{Z}[1/p]^{d})%
\curvearrowright \mathbb{Q}_{p}^{d}/\mathbb{Z}[1/p]^{d})
\end{equation*}%
is classwise Borel isomorphic to the orbit equivalence relation 
\begin{equation*}
\mathcal{R}(\mathrm{GL}_{d}(\mathbb{Z}[1/p])\ltimes \mathbb{Z}%
[1/p]^{d}\curvearrowright \mathbb{Q}_{p}^{d}),
\end{equation*}%
associated with the \emph{affine action} \textrm{GL}$_{d}(\mathbb{Z}%
[1/p])\ltimes \mathbb{Z}[1/p]^{d}\curvearrowright \mathbb{Q}_{p}^{d}$. The
following is the main result of this section. The proof will rely on ideas
and results from \cite{coskey_borel_2010,adams_linear_2000,
thomas_classification_2003,thomas_classification_2011,
hjorth_classification_2006,ioana_cocycle_2011,ioana_orbit_2016} which we
review below.

\begin{theorem}
\label{Theorem:affine-action-intro}Fix $m,d\geq 1$, and prime numbers $%
p,q\geq 2$. Let $\Gamma $ be a subgroup of $\mathrm{GL}_{m}(\mathbb{Z}[1/q]%
\mathbb{)}$ containing a finite index subgroup of $\mathrm{SL}_{m}( \mathbb{Z%
}) $, and $\Delta $ be a subgroup of $\mathrm{GL}_{d}(\mathbb{Z}[1/p]\mathbb{%
)}$.

\begin{enumerate}
\item If $m>d$, then 
\begin{equation*}
\mathcal{R}(\Gamma \ltimes \mathbb{Z}[1/q]^{m}\curvearrowright \mathbb{Q}%
_{q}^{m})\nleq _{B}{}\mathcal{R}(\Delta \ltimes \mathbb{Z}%
[1/p]^{d}\curvearrowright \mathbb{Q}_{p}^{d})\text{.}
\end{equation*}

\item If $p,q$ are \emph{distinct}, $m\geq 3$, and $d\geq 2$ then%
\begin{equation*}
\mathcal{R}(\Gamma \ltimes \mathbb{Z}[1/q]^{m}\curvearrowright \mathbb{Q}%
_{q}^{m})\nleq _{B}{}\mathcal{R}(\Delta \ltimes \mathbb{Z}%
[1/p]^{d}\curvearrowright \mathbb{Q}_{p}^{d})\text{.}
\end{equation*}

\item If $\Delta $ is abelian then $\mathcal{R}(\Delta \ltimes \mathbb{Z}%
[1/p]^{d}\curvearrowright \mathbb{Q}_{p}^{d})$ is hyperfinite and not smooth.

\item If $d\geq 2$ then $\mathcal{R}(\Gamma \ltimes \mathbb{Z}%
[1/q]^{d}\curvearrowright \mathbb{Q}_{q}^{d})$ is not treeable.
\end{enumerate}
\end{theorem}

The proof of Theorem \ref{Theorem:affine-action-intro} will be concluded at
the very end of this section. We recall now some definitions regarding
measure preserving dynamics which will be used throughout the rest of this
section.

A \emph{standard atomless probability space $X$} is a standard Borel space
endowed with an atomless probability measure on its Borel $\sigma$-algebra.
A \emph{probability-measure-preserving }(pmp) \emph{equivalence relation} on 
$X$ is the orbit equivalence relation $\mathcal{R}\left( \Gamma
\curvearrowright X\right) $ associated with a measure-preserving action of a
countable group $\Gamma$ on such an $X$. We say that $\mathcal{R}( \Gamma
\curvearrowright X) $ is \emph{ergodic} if the action $\Gamma
\curvearrowright X$ is ergodic, i.e., if the only invariant sets under the
action are of measure 0 or 1. If $E$ is a pmp equivalence relation on $X$
and $F$ is a Borel equivalence relation on a standard Borel space $Y,$ then
an \emph{almost-everywhere} (a.e.) \emph{homomorphism} from $E$ to $F$ is a
Borel function $f:X\rightarrow Y$ such that, for some conull subset $X_{0}$
of $X$, $f\upharpoonright X_0$ is a Borel homomorphism from $%
E\upharpoonright X_0$ to $F$.

\begin{definition}
Let $E$ be a pmp equivalence relation on $X$, and $F$ be a Borel equivalence
relation. Two a.e.\ homomorphisms $f_{0},f_{1}$ from $E$ to $F$ are \emph{%
almost everywhere} (a.e.) \emph{$F$-homotopic} if there exists a conull
subset $X_{0}$ of $X$ such that, for every $x\in X_{0}$, $f_{0}(x)$ and $%
f_{1}(x)$ belong to the same $F$-equivalence class. The pmp equivalence
relation $E$ is $F$-\emph{ergodic} if it is ergodic and every a.e.\
homomorphism from $E$ to $F$ is a.e.\ $F$-homotopic to a constant map.
\end{definition}

Let $\Gamma \curvearrowright X$ and $\Delta \curvearrowright Y$ be two
actions. If $\varphi\colon \Gamma\to \Delta$ is a group homomorphism and $%
f\colon X\to Y$ is a function so that for all $\gamma\in \Gamma$ and $x\in X$
we have $f\left( \gamma \cdot x\right) =\phi \left( \gamma \right) \cdot
f\left( x\right)$, then we say that $\left( \phi ,f\right) $ is a \emph{%
homomorphism of permutation groups}. We similarly define the notion of an 
\emph{a.e.\ homomorphism of permutation groups}, whenever $X$ is endowed
with a probability measure.

\subsection{Hyperfiniteness and treeability}

We now prove items (3) and (4) of Theorem \ref{Theorem:affine-action-intro}
in a slightly more general setup. Let $S$ be a nonempty set of primes and
let $\mathbb{Z}[1/S]$ be the subring of $\mathbb{Q }$ generated by $1/p$
where $p$ ranges in $S$ and let $\mathbb{Z}_{S}$ be the product of $\mathbb{Z%
}_{p}$ where $p$ ranges in $S$. Let also 
\begin{equation*}
\mathbb{Q}_{S}:=\big\{ (x_p)_p\in\prod_{p\in S}\mathbb{Q}_{p} \colon x_p\in 
\mathbb{Z}_{p} \text{ for all but finitely many } p\in S \big\}.
\end{equation*}
be the \emph{restricted product} of the $p$-adic numbers $\mathbb{Q}_{p}$
with respect to the subrings $\mathbb{Z}_{p}$ of the $p$-adic integers,
where $p$ ranges in $S$. Since $\mathbb{Q}$ is a subring of $\mathbb{Q}_{S}$%
, $\mathbb{Q}_{S}$ may be viewed as a $\mathbb{Q}$-vector space. This
determines an action of $\mathrm{GL}_{d}(\mathbb{Q})$ on $\mathbb{Q}_{S}$.
We also have an action of $\mathbb{Z}[1/S]^{d}$ on $\mathbb{Q}_{S}^{d}$ by
translation. Together these actions induce an action of the \emph{affine
group} $\mathrm{GL}_{d}(\mathbb{Z}[1/S])\ltimes \mathbb{Z}[1/S]^{d}$ on $%
\mathbb{Q}_{S}^{d}$. The following propositions generalize (3) and (4) of
Theorem \ref{Theorem:affine-action-intro}

\begin{proposition}
\label{Proposition:hyperfinite}Fix $d\geq 1$. Let $S$ be a nonempty set of
primes, and $\Gamma $ be an abelian subgroup of $\mathrm{GL}_{d}(\mathbb{Z}%
[1/S])$. Then $\mathcal{R}(\Gamma \ltimes \mathbb{Z}[1/S]^{d}%
\curvearrowright \mathbb{Q}_{S}^{d})$ is hyperfinite and not smooth.
\end{proposition}

\begin{proof}
Notice first that $\mathbb{Q}_{S}^{d}$ is Polish, as a locally compact,
metrizable $K_{\sigma}$; see \cite[Theorem 5.3]{kechris_classical_1995}.
Since $\Gamma $ is abelian, $\Gamma \ltimes \mathbb{Z}[1/S]$ is nilpotent.
Therefore $\mathcal{R}(\Gamma \ltimes \mathbb{Z}[1/S]^{d}\curvearrowright 
\mathbb{Q}_{S}^{d})$ is hyperfinite by the main result of \cite%
{schneider_locally_2013}. As the action $\Gamma \ltimes \mathbb{Z}%
[1/S]\curvearrowright \mathbb{Q}_{S}\ $has dense orbits, $\mathcal{R}(\Gamma
\ltimes \mathbb{Z}[1/S]^{d}\curvearrowright \mathbb{Q}_{S}^{d})$ is not
smooth.
\end{proof}

\begin{proposition}
\label{Proposition:treeable}Fix $d\geq 2$, a nonempty set of primes $S$, and
a subgroup $\Gamma $ of $\mathrm{GL}_{d}(\mathbb{Z}[1/S])$ containing a
finite index subgroup of $\mathrm{SL}_{d}(\mathbb{Z})$. Then $\mathcal{R}%
(\Gamma \ltimes \mathbb{Z}[1/S]^{d}\curvearrowright \mathbb{Q}_{S}^{d})$ is
not treeable.
\end{proposition}

\begin{proof}
Let $\Delta :=\Gamma \cap \mathrm{SL}_{d}\left( \mathbb{Z}\right) $. As $%
\mathcal{R}(\Delta \ltimes \mathbb{Z}^{d}\curvearrowright \mathbb{Z}%
_{S}^{d}) $ is a subequivalence relation of the restriction of $\mathcal{R}%
(\Gamma \ltimes \mathbb{Z}[1/S]^{d}\curvearrowright \mathbb{Q}_{S}^{d})$ to $%
\mathbb{Z}_{S}^{d}$, it suffices by \cite[Proposition 3.3]%
{jackson_countable_2002} to show that $\mathcal{R}(\Delta \ltimes \mathbb{Z}%
^{d}\curvearrowright \mathbb{Z}_{S}^{d})$ is not treeable. Let $\mu_s$ be
the Haar measure on $\mathbb{Z}_{S}^{d}$ and $\mu_p$ be the Haar measure on $%
\mathbb{Z}_{p}^{d}$ for all $p\in S$. Notice that for every nontrivial
element $g$ of $\Delta$ the subgroup of $\mathbb{Z}^d$ consisting of fixed
points for $g$ is of infinite index in $\mathbb{Z}^d$. Hence, for all $p\in
S $, the $\mu_p$-measure preserving action $\Delta \ltimes \mathbb{Z}^d
\curvearrowright \mathbb{Z}_{p}^{d}$ is free almost everywhere \cite[Lemma
1.7]{ioana_cocycle_2011}, and therefore the $\mu_S$-measure preserving
action $\Delta \ltimes \mathbb{Z}^{d}\curvearrowright \mathbb{Z}_{S}^{d}$ is
free almost everywhere. Furthermore, $\Delta \ltimes \mathbb{Z}^{d}$ has
property (T) for $d\geq 3$, and $\mathbb{Z}^{2}\leq \Delta \ltimes \mathbb{Z}%
^{2}$ has the relative property (T). Hence, $\Delta \ltimes \mathbb{Z}^{d}$
does \emph{not }have the Haagerup property \cite[Chapter 1]%
{cherix_groups_2001}. Hence, the relation $\mathcal{R}(\Delta \ltimes 
\mathbb{Z}^{d}\curvearrowright \mathbb{Z}_{S}^{d})$ is not treeable \cite[%
Proposition 6]{ueda_notes_2006}.
\end{proof}

\begin{remark}
\label{Remark:treeable}The same conclusions as in Proposition \ref%
{Proposition:treeable} holds for every subgroup $\Gamma $ of $\mathrm{GL}%
_{d}(\mathbb{Z}[1/S])$ such that $\mathbb{Z}^{d}\leq \left( \Gamma \cap 
\mathrm{SL}_{d}(\mathbb{Z})\right) \ltimes \mathbb{Z}^{d}$ has relative
property (T). When $d=2$, this is equivalent to the assertion that $\Gamma
\cap \mathrm{SL}_{2}\left( \mathbb{Z}\right) $ is not virtually cyclic by 
\cite[Section 5, Example 2]{burger_kazhdan_1991}.
\end{remark}

\subsection{Comparing affine actions of different dimension}

Let $\Gamma\curvearrowright A$ be an action of a countable group $\Gamma$ on
a countable abelian group $A$ be automorphisms, and let $\boldsymbol{A}%
=(A^{(n)})$ be a filtration on $A$, consisting of $\Gamma$-invariant
finite-index subgroups of a $A$. This induces an action $\Gamma
\curvearrowright \hat{A}$ of $\Gamma$ on the associated profinite completion 
$\hat{A}$ of $A$ by (Haar measure-preserving) topological group automorphism
of $\hat{A}$ . We additionally have the translation action of $A$ on $\hat{A}
$. Together these actions induce a (Haar) measure-preserving action of the
semidirect product $\Gamma \ltimes A$ on $\hat{A}$. Since $A$ is dense in $%
\hat{A}$, the action $\Gamma \ltimes A\curvearrowright\hat{A}$ is ergodic.
Additionally, if for every nontrivial element $g$ of $\Gamma $ the subgroup
of $A$ consisting of fixed points for $g$ is of infinite index, then the
action $\Gamma \ltimes A\curvearrowright \hat{A}$ is a.e.\ free \cite[Lemma
1.7]{ioana_cocycle_2011}.

The following is an immediate consequence of \cite[Theorem B]%
{ioana_cocycle_2011}, where: the pair $\Gamma_0\leq \Gamma$ in \cite%
{ioana_cocycle_2011} corresponds to the pair $A\leq \Gamma \ltimes A$; the 
\emph{profinite action} $\Gamma\curvearrowright X$ in \cite%
{ioana_cocycle_2011} corresponds to the action $\Gamma \ltimes
A\curvearrowright \hat{A}$; and the cocycle $w\colon \Gamma \times X \to
\Lambda$ in \cite{ioana_cocycle_2011} corresponds to the unique map $%
(g,x)\mapsto h$ below, with values in $\Delta$, so that $h \cdot f(x)=f(g
\cdot x)$.

\begin{proposition}[Ioana]
\label{Corollary:profinite-Ioana} Let $\Gamma$ be a finitely-generated group
acting on a countable abelian group $A$ by automorphisms so that, for every
non-trivial $g\in\Gamma$, the group of elements of $A$ fixed by $g$ has
infinite index in $A$. Let also $\hat{A}$ be the profinite completion of $A$
with respect to some filtration $\boldsymbol{A}=(A^{(n)})$ on $A$,
consisting of $\Gamma$-invariant finite-index subgroups. Let finally $f:\hat{%
A} \rightarrow Y$ be an a.e.\ homomorphism from $\mathcal{R}(\Gamma \ltimes
A\curvearrowright \hat{A})$ to $\mathcal{R}( \Delta \curvearrowright Y) $,
where $\Delta \curvearrowright Y$ is some free Borel action of a countable
group $\Delta$ on a standard Borel space $X$. We have the following:

\begin{enumerate}
\item If the pair $A\leq \Gamma \ltimes A$ has the relative property (T),
then there exist an $n\in \omega$, a group homomorphism $\phi
:A^{(n)}\rightarrow \Delta $, and a Borel function $f^{\prime }:\hat{A}%
^{(n)}\rightarrow Y$, where $\hat{A}^{(n)}$ is the closure of $A^{(n)}$
inside $\hat{A}$, such that:

\begin{itemize}
\item $f^{\prime }$ is a.e.\ $\mathcal{R}\left( \Delta \curvearrowright
Y\right) $-homotopic to $f\upharpoonright\hat{A}^{(n)}$, and

\item $\left( \phi ,f^{\prime }\right) $ is an a.e.\ homomorphism of
permutation groups from $A^{(n)}\curvearrowright \hat{A}^{(n)}$ to $\Delta
\curvearrowright Y$.
\end{itemize}

\item If the group $\Gamma \ltimes A$ has property (T), then there exist an $%
n\in \omega $, a group homomorphism $\phi :\Gamma \ltimes A^{(n)}\rightarrow
\Delta $, and a Borel function $f^{\prime }:\hat{A}^{(n)}\rightarrow Y$,
where $\hat{A}^{(n)}$ is the closure of $A^{(n)}$ inside $\hat{A}$, such
that:

\begin{itemize}
\item $f^{\prime }$ is a.e.\ $\mathcal{R}\left( \Delta \curvearrowright
Y\right) $-homotopic to $f\upharpoonright\hat{A}^{(n)}$, and

\item $\left( \phi ,f^{\prime }\right) $ is an a.e.\ homomorphism of
permutation groups from $\Gamma \ltimes A^{(n)}\curvearrowright \hat{A}%
^{(n)} $ to $\Delta \curvearrowright Y$.
\end{itemize}
\end{enumerate}
\end{proposition}

Recall that $E_{0}$ stands for the relation of eventual equality of binary
sequences. One can regard $E_{0}$ as the orbit equivalence relation
associated with a continuous free action of a countable group on a Polish
space as follows. Set 
\begin{equation*}
B_{0}:=\bigoplus_{n\in \omega }\mathbb{Z}/2\mathbb{Z}\quad \text{ and }\quad
B:=\prod_{n\in \omega }\mathbb{Z}/2\mathbb{Z}.
\end{equation*}%
Then $B_{0}$ is a countable subgroup of the Polish group $B$ and $E_{0}$ is
the coset relation $\mathcal{R}\left( B_{0}\curvearrowright B\right) $.
Notice that $B_{0}$ is a \emph{Boolean group}, i.e., for all $b\in B_{0}$ we
have that $2b=0$.

The following is a consequence of Proposition \ref{Corollary:profinite-Ioana}%
. Notice that if $X$ is a compact abelian group, and $A\subseteq X$ is a
countable dense subgroup, then the translation action $A\curvearrowright X$
is ergodic.

\begin{lemma}
\label{Lemma:E0-ergodic} Let $\Gamma \curvearrowright A$, $\boldsymbol{A}%
=(A^{(n)})$, and $\hat{A}$ be as in the statement of Proposition \ref%
{Corollary:profinite-Ioana} and assume that for every $n\in \omega $, $%
A^{(n)}$ does not have an infinite Boolean group as quotient. If the pair $%
A\leq \Gamma \ltimes A$ has relative property (T), then $\mathcal{R}(\Gamma
\ltimes A\curvearrowright \hat{A})$ is $E_{0}$-ergodic.
\end{lemma}

\begin{proof}
Let $f:\hat{A}\rightarrow B$ be an a.e.\ homomorphism from $\mathcal{R}%
(\Gamma \ltimes A\curvearrowright \hat{A})$ to $E_{0}=\mathcal{R}\left(
B_{0}\curvearrowright B\right) $. By Proposition \ref%
{Corollary:profinite-Ioana}, there exist $n\in \omega $, a group
homomorphism $\phi :A^{(n)}\rightarrow B_{0}$, and a Borel function $g:\hat{A%
}^{(n)}\rightarrow B$ such that $g$ is a.e.\ $E_{0}$-homotopic to $%
f\upharpoonright \hat{A}^{(n)}$, and $\left( \phi ,g\right) $ is an a.e.\
homomorphism of permutation groups from $A^{(n)}\curvearrowright \hat{A}%
^{(n)}$ to $B_{0}\curvearrowright B$. By assumption, $K:=\mathrm{\mathrm{Ker}%
}\left( \phi \right) $ has finite index in $A^{(n)}$. Since $\left( \phi
,g\right) $ is an a.e.\ homomorphism of permutation groups, we have that $%
\phi \left( a\right) \cdot g\left( x\right) =g\left( a\cdot x\right) $ for
every $a\in A^{\left( n\right) }$ and $x\in \hat{A}^{(n)}$. Let $\hat{K}$ be
the closure of $K$ in $\hat{A}^{\left( n\right) }$, which is an open
subgroup of $\hat{A}^{\left( n\right) }$ since $K$ is a finite-index
subgroup of $A^{\left( n\right) }$. Thus, for a.e.\ $x\in \hat{K}$ and for
every $a\in K$ we have that $g\left( a\cdot x\right) =g\left( x\right) $.
Since $K$ is dense in $\hat{K}$, the translation $K\curvearrowright \hat{K}$
is ergodic. By the above, the function $g|_{\hat{K}}:\hat{K}\rightarrow B$
is a.e.\ $K$-invariant. By ergodicity, this implies that $g|_{\hat{K}}$ is
constant on a conull subset $W$ of $\hat{K}$. Since $\hat{K}$ is an open
subgroup of $\hat{A}^{\left( n\right) }$, $W$ is not null in $\hat{A}%
^{\left( n\right) }$. Since $g$ and $f\upharpoonright \hat{A}^{(n)}$ are $%
E_{0}$-homotopic, this implies that there exists a $B_{0}$-orbit $%
B_{0}+b\subseteq B$ such that $f\left( W\right) \subseteq B_{0}$. Since $f:%
\hat{A}\rightarrow B$ is an a.e.\ homomorphism from $\mathcal{R}(\Gamma
\ltimes A\curvearrowright \hat{A})$ to $E_{0}=\mathcal{R}\left(
B_{0}\curvearrowright B\right) $, $f^{-1}\left( B_{0}+b\right) $ is (up to
null sets) a $\Gamma \ltimes A$-invariant subset of $\hat{A}$ containing $W$%
, and hence not null. Since the action $\Gamma \ltimes A\curvearrowright 
\hat{A}$ is ergodic, we must have that $f^{-1}\left( B_{0}+b\right) $ is
conull, concluding the proof.
\end{proof}

\begin{theorem}
\label{Theorem:increase-dimension-general} Let $\Gamma \curvearrowright A$, $%
\boldsymbol{A}=(A^{(n)})$, and $\hat{A}$ be as in the statement of
Proposition \ref{Corollary:profinite-Ioana} so that moreover:

\begin{enumerate}
\item $A^{(n)}$ does not have an infinite Boolean group as quotient;

\item $\Gamma $ is a subgroup of $\mathrm{SL}_{m}\left( \mathbb{Z}\right) $,
with $m\geq 3$, so that $\Gamma \ltimes A$ has property (T).
\end{enumerate}

If $q$ is a prime number and $1\leq d<m$, then $\mathcal{R}(\Gamma \ltimes
A\curvearrowright \hat{A})$ is $\mathcal{R}(\mathrm{GL}_{d}(\mathbb{Q}%
)\ltimes \mathbb{Q}^{d}\curvearrowright \mathbb{Q}_{q}^{d})$-ergodic.
\end{theorem}

\begin{remark}
Notice that if $m\geq 3$, then $\mathrm{SL}_{m}(\mathbb{R})\ltimes \mathbb{R}%
^{m}$ has property (T) \cite[Corollary 1.4.16]{bekka_kazhdans_2008}. Since $%
\mathrm{SL}_{m}(\mathbb{Z})\ltimes \mathbb{Z}^{m}$ is a lattice in $\mathrm{%
SL}_{m}(\mathbb{R})\ltimes \mathbb{R}^{m}$, it also has property (T).
\end{remark}

The proof of Theorem \ref{Theorem:increase-dimension-general} is modeled
after the proofs of \cite[Theorem 1.1]{thomas_classification_2011} and \cite[%
Theorem 3.6]{coskey_borel_2010}. We will use the following lemmas. The first
lemma is an instance of \cite[Theorem 4.4]{thomas_classification_2011}; see
also the proof of \cite[Theorem 4.3]{thomas_classification_2011}.

\begin{lemma}[Thomas]
\label{Lemma:thomas-nonembed-Q}Suppose that $m\geq 2$ and $G$ is an
algebraic $\mathbb{Q}$-group of dimension less than $m^{2}-1$. If $\Gamma$
is a finite-index subgroup of $\mathrm{SL}_{m}(\mathbb{Z})$ and $\psi:\Gamma
\rightarrow G(\mathbb{Q})$ is a homomorphism, then $\mathrm{\mathrm{ker}}(
\psi ) $ has finite index in $\Gamma$.
\end{lemma}

\begin{lemma}
\label{Lemma:increase-dimension-general1} Let $\Delta $ be a subgroup of $%
\mathrm{GL}_{d}(\mathbb{Q})\ltimes \mathbb{Q}^{d}$, and let $Y$ be a $\Delta 
$-invariant Borel subset of $\mathbb{Q}_{q}^{d}$, so that the action $\Delta
\curvearrowright Y$ is free. Under the assumptions of Theorem \ref%
{Theorem:increase-dimension-general}, we have that every a.e.\ homomorphism
from $\mathcal{R}(\Gamma \ltimes A\curvearrowright \hat{A})$ to $\mathcal{R}%
(\Delta \curvearrowright Y)$ is $\mathcal{R}(\mathrm{GL}_{d}(\mathbb{Q}%
)\ltimes \mathbb{Q}^{d}\curvearrowright \mathbb{Q}_{q}^{d})$-homotopic to a
constant map.
\end{lemma}

\begin{proof}
Suppose that $f:\hat{A}\rightarrow X$ is an a.e.\ homomorphism from $%
\mathcal{R}(\Gamma \ltimes A\curvearrowright \hat{A})$ to $\mathcal{R}%
(\Delta \curvearrowright Y)$. By Proposition \ref{Corollary:profinite-Ioana}%
, after replacing $A$ with $A_{n}$ for some $n\in \omega $, we can assume
that there exist a group homomorphism $\phi :\Gamma \ltimes A\rightarrow
\Delta $, and a Borel function $g:\hat{A}\rightarrow Y$, such that $g$ is $%
\mathcal{R}(\Delta \curvearrowright Y)$-homotopic to $f$, and $\left( \phi
,g\right) $ is an a.e.\ homomorphism of permutation groups from $\Gamma
\ltimes A\curvearrowright \hat{A}$ to $\Delta \curvearrowright Y$. Let $\pi :%
\mathrm{GL}_{d}(\mathbb{Q)}\ltimes \mathbb{Q}^{d}\rightarrow \mathrm{GL}_{d}(%
\mathbb{Q})$ be the canonical quotient map. By Lemma \ref%
{Lemma:thomas-nonembed-Q}, after replacing $\Gamma $ with a finite-index
subgroup, we can assume without loss of generality that $\Gamma $ is
entirely contained in the kernel of $\left( \pi \circ \phi \right) $. By
Lemma \ref{Lemma:E0-ergodic}, $\mathcal{R}(\Gamma \ltimes A\curvearrowright 
\hat{A})$ is $E_{0}$-ergodic. But if $\Lambda :=\left( \pi \circ \phi
\right) \left( \Gamma \ltimes A\right) =\left( \pi \circ \phi \right) \left(
A\right) $, then $g:\hat{A}\rightarrow Y$ is an a.e.\ homomorphism from $%
\mathcal{R}(\Gamma \ltimes A\curvearrowright \hat{A})$ to $\mathcal{R}%
(\Lambda \ltimes \mathbb{Q}^{d}\curvearrowright \mathbb{Q}_{q}^{d})$. As $%
\Lambda \ltimes \mathbb{Q}^{d} $ is nilpotent, we have that $\mathcal{R}%
(\Lambda \ltimes \mathbb{Q}^{d}\curvearrowright \mathbb{Q}_{q}^{d})$ is
Borel reducible to $E_{0}$ by the main result of \cite%
{schneider_locally_2013}. As $\mathcal{R}(\Gamma \ltimes A\curvearrowright 
\hat{A})$ is $E_{0}$-ergodic, it follows that $g$, and hence $f$, are $%
\mathcal{R}(\Lambda \ltimes \mathbb{Q}^{d}\curvearrowright \mathbb{Q}%
_{q}^{d})$-homotopic to a constant map.
\end{proof}

\begin{proof}[Proof of Theorem \protect\ref%
{Theorem:increase-dimension-general}]
Suppose that $f:\hat{A}\rightarrow \mathbb{Q}_q^d$ is an a.e.\ homomorphism
from $\mathcal{R}(\Gamma \ltimes A\curvearrowright \hat{A})$ to $\mathcal{R}(%
\mathrm{GL}_{d}(\mathbb{Q})\ltimes \mathbb{Q}^{d}\curvearrowright \mathbb{Q}%
_{q}^{d})$. We will show that $f$ is a.e.\ $\mathcal{R}(\mathrm{GL}_{d}(%
\mathbb{Q})\ltimes \mathbb{Q}^{d}\curvearrowright \mathbb{Q}_{q}^{d})$%
-homotopic to a constant map. We identify $\mathbb{Q}_{q}^{d}$ with the
tensor product $\mathbb{Q} _{q}\otimes _{\mathbb{Q}}\mathbb{Q}^{d}$. By a $%
\mathbb{Q}$-subspace of $\mathbb{Q}_{q}^{d}$ we mean any $\mathbb{Q}_{q}$%
-linear subspace of $\mathbb{Q}_{q}^{d}$ that is of the form $\mathbb{Q}%
_{q}\otimes _{\mathbb{Q}}V$, where $V\subseteq \mathbb{Q} ^{d}\subseteq 
\mathbb{Q}_{q}^{d}$ is a $\mathbb{Q}$-linear subspace of $\mathbb{Q}^{d}$.
An affine $\mathbb{Q}$-variety is a subset of $\mathbb{Q}_{q}^{d}$ of the
form $a+V$, where $V$ is a $\mathbb{Q} $-subspace and $a\in \mathbb{Q}^{d}$.

\begin{clm}
Suppose that $M\in M_{d}(\mathbb{Q)}$ is a $d\times d$ matrix with rational
coefficients and set 
\begin{equation*}
\mathrm{\mathrm{Ker}}_{\mathbb{Q}_{q}}\left( M\right) :=\left\{ x\in \mathbb{%
Q}_{q}^{d}:Mx=0\right\} \quad\text{ and }\quad \mathrm{\mathrm{Ker}}_{%
\mathbb{Q}}\left( M\right) :=\left\{ x\in \mathbb{Q} ^{d}:Mx=0\right\} \text{%
.}
\end{equation*}%
Then $\mathrm{\mathrm{Ker}}_{\mathbb{Q}_{q}}\left( M\right) =\mathbb{Q}%
_{p}\otimes \mathrm{Ker}_{\mathbb{Q}}\left( M\right) $. In particular, $%
\mathrm{\mathrm{Ker}}_{\mathbb{Q}_{q}}\left( M\right) $ is $\mathbb{Q}$%
-subspace of $\mathbb{Q}_{q}^{d}$.
\end{clm}

\begin{proof}[Proof of Claim]
By the Gaussian elimination procedure, we can assume that $M$ is in reduced
row echelon form. It is then clear from the reduced row echelon form of $M$
that the dimension of $\mathrm{\mathrm{Ker}}_{\mathbb{Q}}\left( M\right) $
as a $\mathbb{Q}$-vector space is equal to the dimension of $\mathrm{\mathrm{%
Ker}}_{\mathbb{Q}_{q}}\left( M\right) $ as a $\mathbb{Q}_{q}$ -vector space.
Thus, $\mathbb{Q}_{q}\otimes \mathrm{Ker}_{\mathbb{Q}}\left( M\right) $ and $%
\mathrm{Ker}_{\mathbb{Q}_{q}}\left( M\right) $ are $\mathbb{Q}_{q}$-vector
spaces of the same dimension. Since $\mathbb{Q}_{q}\otimes \mathrm{Ker}_{%
\mathbb{Q}}\left( M\right) \subseteq \mathrm{Ker}_{\mathbb{Q} _{q}}\left(
M\right) $, we must have $\mathbb{Q}_{q}\otimes \mathrm{Ker}_{\mathbb{Q}%
}\left( M\right) =\mathrm{Ker}_{\mathbb{Q}_{q}}\left( M\right) $.
\end{proof}

\begin{clm}
Suppose that $M\in M_{d}(\mathbb{Q)}$ is a $d\times d$ matrix with rational
coefficients, and $t\in \mathbb{Q}^{d}$. Set 
\begin{equation*}
W_{\mathbb{Q}_{q}}=\left\{ x\in \mathbb{Q}_{q}^{d}:Mx=t\right\} \quad \text{
and } \quad W_{\mathbb{Q}}=\left\{ x\in \mathbb{Q}^{d}:Mx=t\right\} \text{.}
\end{equation*}
Then $W_{\mathbb{Q}}$ is nonempty if and only if $W_{\mathbb{Q}_{q}}$ is
nonempty. In this case, $W_{\mathbb{Q}_{q}}$ is an affine $\mathbb{Q}$%
-variety of $\mathbb{Q}_{q}^{d}$.
\end{clm}

\begin{proof}[Proof of Claim]
Again, by the Gaussian elimination procedure, we can assume that $M$ is in
reduced row echelon form. It is then clear that $W_{\mathbb{Q}}$ is nonempty
if and only if $W_{\mathbb{Q}_{q}}$ is nonempty. Suppose thus that $W_{%
\mathbb{Q}_{q}}$ (or, equivalently, $W_{\mathbb{Q}}$) is nonempty. Pick $%
x_{0}\in W_{\mathbb{Q}}$. Then we have that $W_{\mathbb{Q}_{q}}=\mathrm{%
\mathrm{Ker}}_{\mathbb{Q}_{q}}\left( M\right) +x_{0}$. By the previous
claim, $\mathrm{\mathrm{Ker}}_{\mathbb{Q}_{q}}\left( M\right) $ is a $%
\mathbb{Q}$-subspace of $\mathbb{Q}_{q}^{d}$. Since $x_{0}\in \mathbb{Q}
^{d} $, we have that $W_{\mathbb{Q}_{q}}=\mathrm{\mathrm{Ker}}_{\mathbb{Q}%
_{q}}\left( M\right) +x_{0}$ is an affine $\mathbb{Q}$-variety of $\mathbb{Q}%
_{q}^{d}$.
\end{proof}

\begin{clm}
$\mathrm{Fix}\left( \gamma \right) :=\left\{ x\in \mathbb{Q} _{q}^{d}:\gamma
\cdot x=x\right\} $ is an affine $\mathbb{Q}$-variety, for all $\gamma \in 
\mathrm{GL}_{d}(\mathbb{Q})\ltimes \mathbb{Q}^{d}$.
\end{clm}

\begin{proof}[Proof of Claim]
Notice that there exist $M\in M_{d}(\mathbb{Q)}$ and $t\in \mathbb{Q}^{d}$
such that $\gamma \cdot x=Mx+t$ for $x\in \mathbb{Q}_{q}^{d}$. Thus, 
\begin{equation*}
\mathrm{Fix}\left( \gamma \right) =\left\{ x\in \mathbb{Q}_{q}^{d}:\left(
I-M\right) x=t\right\} \text{.}
\end{equation*}%
This is an affine $\mathbb{Q}$-variety by the previous claim.
\end{proof}

Let $\mathcal{A}_{\mathbb{Q}}$ be the set of all affine $\mathbb{Q}$%
-varieties of $\mathbb{Q}_{q}^{d}$, ordered by inclusion. As the
intersection of affine $\mathbb{Q}$-varieties is an affine $\mathbb{Q}$%
-variety, for every $y\in Y$ there is a (unique) smallest $\mathbb{Q}$%
-variety $V_{y}\in \mathcal{A}_{\mathbb{Q}}$ containing $y$. If $\gamma \in 
\mathrm{GL}_{d}(\mathbb{Q)}\ltimes \mathbb{Q}^{d}$ and $W\in \mathcal{A}_{%
\mathbb{Q}}$, then $\gamma \cdot W:=\left\{ \gamma \cdot x:x\in W\right\} $
is also an affine $\mathbb{Q}$-variety. Thus, $V_{\gamma \cdot y}=\gamma
\cdot V_{y}$, for every $y\in \mathbb{Q}_{q}^{d}$.

\begin{clm}
Fix $V\in \mathcal{A}_{\mathbb{Q}}$ and let $Y:=\left\{ y\in \mathbb{Q}%
_{q}^{d}:V_{y}=V\right\} \subseteq V$. If $\Delta $ is the group of affine
transformations of $V$ obtained as setwise stabilizers of $V$ in $\mathrm{GL}%
_{d}(\mathbb{Q)}\ltimes \mathbb{Q}^{d}$. Then the action $\Delta
\curvearrowright Y$ is free.
\end{clm}

\begin{proof}[Proof of Claim]
Suppose that $\delta \left( e\right) =e$, for some $\delta \in \Delta $ and $%
e\in Y$. We have that $e\in \mathrm{Fix}\left( \delta \right) =\left\{ x\in 
\mathbb{Q}_{q}^{d}:\gamma \cdot x=x\right\} $. By the previous claim, $%
\mathrm{Fix}\left( \delta \right) $ is an affine $\mathbb{Q}$-variety of $%
\mathbb{Q}_{q}^{d}$. Therefore, $V=V_{y}\subseteq \mathrm{Fix}\left( \delta
\right) $. This shows that $\gamma \cdot y=y$ for every $y\in V $.\ Thus $%
\delta $ is the trivial element of $\Delta $.
\end{proof}

Since $\mathcal{A}_{\mathbb{Q}}$ is countable, we can assume without loss of
generality that there exists $V\in \mathcal{A}_{\mathbb{Q}}$ such that $%
V_{f\left( x\right) }=V$ for every $x\in \hat{A}$. To see this, as in the
proof of \cite[Lemma 5.1]{thomas_classification_2002}, pick $V\in \mathcal{A}%
_{\mathbb{Q}}$ such that $X_{0}:=\{x\in \hat{A}:V_{f\left( x\right) }=V\}$
is nonnull. By ergodicity of the action $\Gamma \ltimes A\curvearrowright 
\hat{A}$ we have that%
\begin{equation*}
X_{1}:=\bigcup_{\alpha \in \Gamma \ltimes A}\alpha \cdot X_{0}
\end{equation*}%
has full measure. Let $c:X_{1}\rightarrow X_{0}$ be a Borel function such
that $c\left( x\right) \in \left( \Gamma \ltimes A\right) \cdot x\cap X_{0}$
for every $x\in X_{1}$. Let also $x_{0}$ be a point in $X_{0}$. Without loss
of generality replace $f$ with the Borel function $g$ defined by%
\begin{equation*}
g:x\mapsto \left\{ 
\begin{array}{ll}
\left( f\circ c\right) \left( x\right) & x\in X_{1} \\ 
f( x_{0} ) & x\in \hat{A}\setminus X_{1}\text{.}%
\end{array}%
\right.
\end{equation*}%
This function satisfies $V_{g\left( x\right) }=V$ for every $x\in \hat{A}$.

So let us thus assume that $V_{f\left( x\right) }=V\in \mathcal{A}_{\mathbb{Q%
}}$ for every $x\in \hat{A}$. Let $\left( v_{1},\ldots ,v_{d}\right) $ be
the canonical basis of $\mathbb{Q}_{q}^{d}$ over $\mathbb{Q}_{q}$. 
Define $Y:=\left\{ y\in V:V_{y}=V\right\} $. Let also $\Delta $ to be the
group of affine transformations of $V$ obtained as restrictions to $V$ of
elements of the setwise stabilizer of $V$ inside $\mathrm{GL}_{d}(\mathbb{Q}%
)\ltimes \mathbb{Q}^{d}$. 

\begin{clm}
The action $\Delta \curvearrowright Y$ is free, and $f:\hat{A}\rightarrow Y$
is an a.e.\ homomorphism from $\mathcal{R}(\Gamma \ltimes A\curvearrowright 
\hat{A})$ to $\mathcal{R}(\Delta \curvearrowright Y)$.
\end{clm}

\begin{proof}[Proof of Claim]
It follows from the previous claim that the action $\Delta \curvearrowright
Y $ is free. We now show that the function $f:\hat{A}\rightarrow Y$ is a
Borel homomorphism from $\mathcal{R}(\Gamma \ltimes A\curvearrowright \hat{A}%
)$ to $\mathcal{R}(\Delta \curvearrowright Y)$. Since $f$ is an a.e.\
homomorphism from $\mathcal{R}(\Gamma \ltimes A\curvearrowright \hat{A})$ to 
$\mathcal{R}(\mathrm{GL}_{d}(\mathbb{Q)}\ltimes \mathbb{Q}%
^{d}\curvearrowright \mathbb{Q}_{q}^{d})$, there is a conull subset $X_{0}$
of $\hat{A}$ such that, whenever $x_{0},x_{1}\in X_{0}$ and $\gamma \in
\Gamma \ltimes A$ are such that $\gamma \cdot x_{0}=x_{1}$, there exists $%
\delta \in \mathrm{GL}_{d}(\mathbb{Q})\ltimes \mathbb{Q}^{d}$ such that $%
\delta \cdot f\left( x_{0}\right) =f(x_{1})$. Suppose now that $%
x_{0},x_{1}\in X_{0}$ and $\gamma \in \Gamma \ltimes A$ are such that $%
\gamma \cdot x_{0}=x_{1}$. Then there exists $\delta \in \mathrm{GL}_{d}(%
\mathbb{Q})\ltimes \mathbb{Q}^{d}$ such that $\delta \cdot f(x_{0})=f(x_{1})$%
. Since $f(x_{0}),f(x_{1})\in Y=\left\{ y\in V:V_{y}=V\right\} $, we have
that $V=V_{f\left( x_{1}\right) }=V_{\delta \cdot f(x_{0})}=\delta \cdot
V_{f(x_{0})}$. Therefore, $\delta $ belongs to the setwise stabilizer of $V$
inside $\text{ GL}_{d}(\mathbb{Q})\ltimes \mathbb{Q}^{d}$. Thus, $\delta
|_{V}\in \Delta $ and $\delta |_{V}\cdot f(x_{0})=f(x_{1})$. This shows that 
$f:\hat{A}\rightarrow Y$ is an a.e.\ homomorphism from $\mathcal{R}(\Gamma
\ltimes A\curvearrowright \hat{A})$ to $\mathcal{R}(\Delta \curvearrowright
Y)$.
\end{proof}

The conclusion now follows from Lemma \ref{Lemma:increase-dimension-general1}%
.
\end{proof}

\subsection{Comparing affine actions associated to different primes}

Let $L$ be a closed subgroup of a compact group $K$. The space $K/L$ is then
endowed with a normalized Haar measure, and the $K$-action is
measure-preserving. If $\Gamma$ is a countable subgroup of $K$, one may
consider the induced action of $\Gamma $ on $K/L$. Such an action is ergodic
if and only if $\Gamma $ is dense in $K$ by \cite[Lemma 4.1.1]%
{margulis_discrete_1991}. We will need the following consequence of \cite[
Lemma 2.3]{coskey_borel_2010}, which is suggested in \cite[Remark 2.5]%
{coskey_borel_2010} as it plays a crucial role in the proof of \cite[Theorem
3.6]{coskey_borel_2010}.

\begin{lemma}[Coskey]
\label{Lemma:homog2}For $i\in \left\{ 0,1\right\} $, let $L_i$ be a closed
subgroup of a compact group $K_i$. If $\Gamma _{0}$ is a countable dense
subgroup of $K_{0}$ and $\left( \phi ,f\right) $ is an a.e.\ homomorphism of
permutation groups from $\Gamma _{0}\curvearrowright K_{0}/L_{0}$ to $%
K_{1}\curvearrowright K_{1}/L_{1}$, then there exist a closed subgroup $%
\tilde{K}_{1}$ of $K_{1}$, a transitive $\tilde{K} _{1}$-space $\tilde{X}%
_{1}\subseteq K_{1}/L_{1}$, a closed normal subgroup $H$ of $\tilde{K}_{1}$,
and a continuous surjective homomorphism $\Phi:K_{0}\rightarrow \tilde{K}%
_{1}/H$ such that:

\begin{enumerate}
\item $\Phi |_{\Gamma _{0}}=\pi \circ \phi $, where $\pi :\tilde{K}%
_{1}\rightarrow \tilde{K}_{1}/H$ is the quotient map;

\item $H$ is the kernel of the action $\tilde{K}_{1}\curvearrowright \tilde{X%
}_{1}$;

\item $\left( \Phi ,f\right) $ is an a.e.\ homomorphism of permutation
groups from $K_{0}\curvearrowright K_{0}/L_{0}$ to $\tilde{K}_{1}/H
\curvearrowright \tilde{X}_{1}$.
\end{enumerate}
\end{lemma}

\begin{proof}
Define $\tilde{K}_{1}$ to be the closure of $\phi( \Gamma _{0}) $ inside $%
K_{1}$. Since $\Gamma _{0}$ is dense in $K_{0}$, the action $\Gamma
_{0}\curvearrowright K_{0}/L_{0}$ is ergodic. Hence there exists $z\in
K_{1}/L_{1}$ such that, $\{x\in K_{0}/L_{0}:f\left( x\right) \in \tilde{K}%
_{1}\cdot z\}$ is conull. Define $\tilde{X}_{1}=\tilde{K}_{1}\cdot z$ and $H$
to be the kernel of the action $\tilde{K}_{1}\curvearrowright \tilde{X}_{1}$%
. Let $\pi :\tilde{K}_{1}\rightarrow \tilde{K}_{1}/H$ be the quotient map.
By \cite[Lemma 2.3]{coskey_borel_2010} applied to $f:K_{0}/L_{0}\rightarrow 
\tilde{X}_{1}$ and $\pi \circ \phi :\Gamma _{0}\rightarrow \tilde{K}_{1}$,
we get the desired homomorphism $\Phi :K_{0}\rightarrow \tilde{K}_{1}/H$.
\end{proof}

Recall that, if $p$ is a prime number, then a pro-$p$ group is a profinite
group $G$ such that every open subgroup has index equal to a power of $p$.
Notice that closed subgroups and quotients of pro-$p$ groups are pro-$p$
groups. A profinite group $G$ is a virtually pro-$p$ group if it contains a
clopen pro-$p$ subgroup.

Fix $m\geq 2$, $k\geq 1$, and a prime $p\geq 3$, and let $\mathrm{SL}_{m}(
p^{k}\mathbb{Z}_{p}) $ be the kernel of the canonical surjective
homomorphism $\mathrm{SL}_{m}( \mathbb{Z}_{p}) \rightarrow \mathrm{SL}_{m}( 
\mathbb{Z}/p^{k}\mathbb{Z}) $. Then $\mathrm{SL}_{m}( p^{k}\mathbb{Z}_{p}) $
is a clopen torsion-free pro-$p$ subgroup of $\mathrm{SL}_{m}( \mathbb{Z}%
_{p}) $ when $k\geq 1$ and $p\geq 3$, or $k\geq 2$ and $p=2$; see \cite[%
Theorem 5.2]{sautoy_horizons_2000}. In particular, \textrm{SL}$_{m}( \mathbb{%
Z}_{p}) $ is a virtually pro-$p$ group.

\begin{lemma}
\label{Lemma:distinguish-pair-primes1} Let $p,q$ be distinct prime numbers,
let $m\geq 3$, and let $n,d\geq 1$. If $\Gamma $ is a subgroup of $\mathrm{SL%
}_{m}(\mathbb{Z})$ with property (T) such that $\Gamma \ltimes \mathbb{Z}%
^{m} $ has property (T), $\Delta $ is a subgroup of $\mathrm{GL}_{d}(\mathbb{%
Q)}\ltimes \mathbb{Q}^{d}$, and $Y\subseteq \mathbb{Q}_{q}^{d}$ is a Borel $%
\Delta $-invariant set such that the action $\Delta \curvearrowright Y$ is
free, then every a.e.\ homomorphism $f$ from $\mathcal{R}(\Gamma \ltimes
p^{n}\mathbb{Z}^{m}\curvearrowright p^{n}\mathbb{Z}_{p}^{m})$ to $\mathcal{R}%
(\Delta \curvearrowright Y)$ is $\mathcal{R}(\mathrm{GL}_{d}(\mathbb{Q)}%
\ltimes \mathbb{Q}^{d}\curvearrowright \mathbb{Q}_{q}^{d})$-homotopic to a
constant map.
\end{lemma}

\begin{proof}
By Lemma \ref{Lemma:E0-ergodic}, it suffices to consider the case when $%
d\geq 2$.

Suppose that $f:p^{n}\mathbb{Z}_{p}^{m}\rightarrow Y$ is an a.e.
homomorphism from $\mathcal{R}(\Gamma \ltimes p^{n}\mathbb{Z}%
^{m}\curvearrowright p^{n}\mathbb{Z}_{p}^{m})$ to $\mathcal{R}(\Delta
\curvearrowright Y)$. Then by Proposition \ref{Corollary:profinite-Ioana},
after replacing $n$ with a larger integer, we can assume that there exist a
group homomorphism $\phi :\Gamma \ltimes p^{n}\mathbb{Z}^{m}\rightarrow
\Delta $, and a Borel function $g:p^{n}\mathbb{Z}_{p}^{m}\rightarrow Y$,
such that $g$ is $\mathcal{R}(\Delta \curvearrowright Y)$-homotopic to $f$,
and $\left( \phi ,g\right) $ is an a.e.\ homomorphism of permutation groups
from $\Gamma \ltimes p^{n}\mathbb{Z}^{m}\curvearrowright p^{n}\mathbb{Z}%
_{p}^{m}$ to $\Delta \curvearrowright Y$. It suffices to show that $g$ is
a.e.\ $\mathcal{R}(\mathrm{GL}_{d}(\mathbb{Q)}\ltimes \mathbb{Q}%
^{d}\curvearrowright \mathbb{Q}_{q}^{d})$-homotopic to a constant map.

\begin{clm}
We may assume without loss of generality that $\phi \left( \Gamma \right)
\subseteq \mathrm{SL}_{d}( \mathbb{Z}) \ltimes \mathbb{Z}^{d}$
\end{clm}

\begin{proof}[Proof of Claim]
Let $\pi :\mathrm{GL}_{d}(\mathbb{Q}_{q})\ltimes \mathbb{Q}%
_{q}^{d}\rightarrow \mathrm{GL}_{d}(\mathbb{Q}_{q})$ be the canonical
quotient map. As $\Gamma $ has property (T), after replacing $\Gamma $ with
a finite index subgroup we can assume that $\left( \pi \circ \phi \right)
\left( \Gamma \right) \subseteq \mathrm{SL}_{d}(\mathbb{Q})$. As $\Gamma $
is finitely-generated, there exists a finite set $F=\left\{ q_{1},\ldots
,q_{t}\right\} $ of prime numbers such that 
\begin{equation*}
\phi \left( \Gamma \right) \subseteq \mathrm{SL}_{d}(\mathbb{Z}[1/F])\ltimes 
\mathbb{Z}[1/F]^{d}\subseteq \mathrm{SL}_{d}(\mathbb{Z}_{F})\ltimes \mathbb{Q%
}_{F}^{d}\text{.}
\end{equation*}%
For $1\leq \ell \leq t$, let $\tau _{\ell }:\mathrm{SL}_{d}(\mathbb{Q}%
_{F})\ltimes \mathbb{Q}_{F}^{d}\rightarrow \mathrm{SL}_{d}(\mathbb{Q}%
_{q_{\ell }})$ be the canonical quotient map. Consider the homomorphism $%
\tau _{\ell }\circ \phi :\Gamma \rightarrow \mathrm{SL}_{d}(\mathbb{Q}%
_{q_{\ell }})$. We have that by \cite[Theorem VIII.3.10]%
{margulis_discrete_1991} the Zariski closure of $\left( \tau _{\ell }\circ
\phi \right) \left( \Gamma \right) $ inside $\mathrm{SL}_{d}(\mathbb{Q}%
_{q_{\ell }})$ is semisimple. Hence, $\left( \tau _{\ell }\circ \phi \right)
\left( \Gamma \right) $ has compact closure inside $\mathrm{SL}_{d}(\mathbb{Q%
}_{q_{\ell }}\mathbb{)}$ by \cite[VII.5.16]{margulis_discrete_1991}. Since 
\textrm{SL}$_{d}\left( \mathbb{Z}_{q_{\ell }}\right) $ is an open subgroup
of \textrm{SL}$_{d}(\mathbb{Q}_{q_{\ell }})$, we have that \textrm{SL}$%
_{d}\left( \mathbb{Z}_{q_{\ell }}\right) \cap \left( \tau _{\ell }\circ \phi
\right) \left( \Gamma \right) $ is a finite index subgroup inside $\left(
\tau _{\ell }\circ \phi \right) \left( \Gamma \right) $. Hence, after
replacing $\Gamma $ with a finite index subgroup, we can assume that $\left(
\tau _{\ell }\circ \phi \right) \left( \Gamma \right) \subseteq \mathrm{SL}%
_{d}(\mathbb{Z}_{q_{\ell }})$ for every $\ell \in \left\{ 1,2,\ldots
,t\right\} $. This implies that $\left( \pi \circ \phi \right) \left( \Gamma
\right) \subseteq \mathrm{SL}_{d}(\mathbb{Z})$ and hence $\phi \left( \Gamma
\right) \subseteq \mathrm{SL}_{d}(\mathbb{Z})\ltimes \mathbb{Z}^{d}[1/F]$.
Since $\Gamma $ is finitely-generated, we have that $\phi \left( \Gamma
\right) \subseteq \mathrm{SL}_{d}(\mathbb{Z})\ltimes \frac{1}{N}\mathbb{Z}%
^{d}$ for some $N\geq 1$. Since $\mathrm{SL}_{d}(\mathbb{Z})\ltimes \mathbb{Z%
}^{d}$ is a finite index subgroup of $\mathrm{SL}_{d}(\mathbb{Z})\ltimes 
\frac{1}{N}\mathbb{Z}^{d}$, after replacing $\Gamma $ with a finite index
subgroup we can assume that $\phi \left( \Gamma \right) \subseteq \mathrm{SL}%
_{d}(\mathbb{Z})\ltimes \mathbb{Z}^{d}$.
\end{proof}

Let $G_{0}$ be the closure of $\Gamma $ inside $\mathrm{SL}_{m}\left( 
\mathbb{Z}_{p}\right) $, and let $G_{1}$ be the closure of $\phi \left(
\Gamma \right) $ inside $\mathrm{SL}_{d}\left( \mathbb{Z}_{q}\right) \ltimes 
\mathbb{Z}_{q}^{d}$. Since $\mathrm{SL}_{m}\left( \mathbb{Z}_{p}\right) $ is
virtually pro-$p$ and $\mathrm{SL}_{d}\left( \mathbb{Z}_{q}\right) \ltimes 
\mathbb{Z}_{q}^{d}$ is virtually pro-$q$, after replacing $\Gamma $ with a
finite index subgroup we can assume that $G_{0}$ is a pro-$p$ group and $%
G_{1}$ is a pro-$q$ group.

By Lemma \ref{Lemma:homog2} there exists a $G_{1}$-invariant closed subset $%
\tilde{Y}$ of $Y$ such that for a.e.\ $x\in p^{n}\mathbb{Z}_{p}^{m}$, $%
g(x)\in \tilde{Y}$, and if $H$ is the kernel of the action $%
G_{1}\curvearrowright \tilde{Y}$ and $p:G_{1}\rightarrow G_{1}/H$ is the
projection map, then $p\circ \phi :\Gamma \rightarrow G_{1}/H$ extends to a
continuous homomorphism $\Phi :G_{0}\rightarrow G/H$ such that $\left( \Phi
,g\right) $ is a homomorphism of permutation groups from $%
G_{0}\curvearrowright p^{n}\mathbb{Z}_{p}^{m}$ to $G_{1}/H\curvearrowright 
\mathbb{Z}_{q}^{d}$. Since $G_{0}$ is a pro-$p$ group and $G_{1}/H$ is a pro-%
$q$ group, we have that $\Phi $ is trivial. Hence, $p\circ \phi $ is trivial
and $\phi (\Gamma )\subseteq H$.

Therefore, if $\Lambda :=\left( \pi \circ \phi \right) \left( \Gamma \right)
\subseteq \mathrm{SL}_{d}\left( \mathbb{Z}\right) $, then we have that $g$
is an a.e.\ homomorphism from $\mathcal{R}\left( \Gamma \ltimes p^{n}\mathbb{%
Z}^{m}\curvearrowright p^{n}\mathbb{Z}_{p}^{m}\right) $ to $\mathcal{R}%
(\Lambda \ltimes \mathbb{Q}^{d}\curvearrowright \mathbb{Q}_{q}^{d})$. As $%
\Lambda \ltimes \mathbb{Q}^{d}$ is nilpotent we have that $\mathcal{R}%
(\Lambda \ltimes \mathbb{Q}^{d}\curvearrowright \mathbb{Q}_{q}^{d})$ Borel
reducible to $E_{0}$ by the main result of \cite{schneider_locally_2013}. As 
$\mathcal{R}\left( \Gamma \ltimes p^{n}\mathbb{Z}^{m}\curvearrowright p^{n}%
\mathbb{Z}_{p}^{m}\right) $ is $E_{0}$-ergodic, it follows that $g$, and
hence $f$, are $\mathcal{R}(\mathrm{GL}_{d}(\mathbb{Q)}\ltimes \mathbb{Q}%
^{d}\curvearrowright \mathbb{Q}_{q}^{d})$-homotopic to the constant map.
\end{proof}

\begin{theorem}
\label{Theorem:distinguish-pair-primes} Let $p,q$ be distinct prime numbers,
let $m\geq 3$, and let $n,d\geq 1$. If $\Gamma$ is subgroup of $\mathrm{SL}%
_{m}( \mathbb{Z}) $ with property (T), then $\mathcal{R}( \Gamma \ltimes
p^{n}\mathbb{Z}\curvearrowright p^{n}\mathbb{Z}_{p}) $ is $\mathcal{R}(%
\mathrm{GL}_{d}(\mathbb{Q)}\ltimes \mathbb{Q}^{d}\curvearrowright \mathbb{Q}%
_{q}^{d})$-ergodic.
\end{theorem}

\begin{proof}
One may proceed as in the proof of Theorem \ref%
{Theorem:increase-dimension-general}, using Lemma \ref%
{Lemma:distinguish-pair-primes1} in the place of Lemma \ref%
{Lemma:increase-dimension-general1}.
\end{proof}

\subsection{Conclusion}

We may now conclude with the proof of Theorem \ref%
{Theorem:affine-action-intro}.

\begin{proof}[Proof of Theorem \protect\ref{Theorem:affine-action-intro}]
Define $E:=\mathcal{R}(\Gamma \ltimes \mathbb{Z}[1/q]^{m}\curvearrowright 
\mathbb{Q}_{q}^{m})$ and $F:=\mathcal{R}(\Delta \ltimes \mathbb{Z}%
[1/p]^{d}\curvearrowright \mathbb{Q}_{p}^{d})$.

(1): When $d=1$ and $m=2$, then the conclusion follows from items (3) and
(4). When $m\geq 3$, Theorem \ref{Theorem:increase-dimension-general} in the
case when $A=\mathbb{Z}[1/q]^m$ and $A^{(n)}=q^{n}\mathbb{Z}^m$ implies that 
$\mathcal{R}( (\Gamma \cap \mathrm{SL}_{m}( \mathbb{Z}) )\ltimes \mathbb{Z}%
^m \curvearrowright \mathbb{Z}_{q}^{m}) $ is $F$-ergodic. The rest follows
from the fact that the inclusion map $\mathbb{Z}_{q}^{m}\subseteq \mathbb{Q}%
_{q}^{m}$ is a Borel homomorphism from $\mathcal{R}(( \Gamma \cap \mathrm{SL}%
_{m}( \mathbb{Z})) \ltimes \mathbb{Z}\curvearrowright \mathbb{Z}_{q}^{m}) $
to $E$ which is not $E$-homotopic to the constant map.

(2): By Theorem \ref{Theorem:distinguish-pair-primes}, $\mathcal{R}\left(
\left( \Gamma \cap \mathrm{SL}_{m}\left( \mathbb{Z}\right) \right) \ltimes 
\mathbb{Z}\curvearrowright \mathbb{Z}_{q}^{m}\right) $ is $F$-ergodic. The
rest follows as in (1) above.

(3): This is a consequence of Proposition \ref{Proposition:hyperfinite}.

(4):\ This is a consequence of Proposition \ref{Proposition:treeable}.
\end{proof}



\providecommand{\bysame}{\leavevmode\hbox to3em{\hrulefill}\thinspace}
\providecommand{\MR}{\relax\ifhmode\unskip\space\fi MR }
\providecommand{\MRhref}[2]{%
  \href{http://www.ams.org/mathscinet-getitem?mr=#1}{#2}
}
\providecommand{\href}[2]{#2}

\end{document}